\documentclass{article}
\usepackage{amsmath}
\usepackage{amsthm}
\usepackage{amssymb}
\usepackage{hyperref}
\usepackage{verbatim}
\usepackage{graphicx}
\usepackage[usenames,dvipsnames,svgnames,table]{xcolor}
\usepackage{mathrsfs}  
\usepackage{accents} 
\usepackage{mathtools}

\usepackage[hmargin=3cm,vmargin=3.5cm]{geometry}
\def\a{\alpha}

\def\de{\delta}
\def\De{\Delta}
\def\ep{\epsilon}

\def\la{\lambda}
\def\La{\Lambda}
\def\si{\sigma}

\def\om{\omega}

\def\th{\theta}
\newcommand{\Th}{\Theta}

\def\CC{{\cal C}}
\def\DD{{\cal D}}

\def\DD{{\cal D}}

\def\SS{{\cal S}}

\newcommand{\N}[0]{\mathbb{N}}

\newcommand{\R}[0]{\mathbb{R}}
\newcommand{\Z}[0]{\mathbb{Z}}

\newcommand{\T}[0]{\mathbb{T}}

\newcommand{\supp}{\mathrm{supp} \,}

\newcommand{\fr}[2]{\frac{#1}{#2}}

\newcommand{\ALI}[1]{\begin{align*} #1 \end{align*}}

\newcommand{\leqc}[0]{\lesssim}

\newcommand{\pr}[0]{\partial}
\newcommand{\nb}{\nabla}

\DeclareMathAlphabet{\mathpzc}{OT1}{pzc}{m}{it}

\newcommand{\hq}[0]{\hat{q}}

\newcommand{\Pmhq}[0]{P_{(-m,\hat{q}]}}

\newcommand{\ali}[1]{ \begin{align} #1 \end{align} }

\def\Xint#1{\mathchoice
   {\XXint\displaystyle\textstyle{#1}}%
   {\XXint\textstyle\scriptstyle{#1}}%
   {\XXint\scriptstyle\scriptscriptstyle{#1}}%
   {\XXint\scriptscriptstyle\scriptscriptstyle{#1}}%
   \!\int}
\def\XXint#1#2#3{{\setbox0=\hbox{$#1{#2#3}{\int}$}
     \vcenter{\hbox{$#2#3$}}\kern-.5\wd0}}

\def\dashint{\Xint-}

\newcommand{\Fsupp}[0]{{\mathcal F}\mbox{supp}\,}

\newcommand{\mrg}[1]{\mathring{#1}}

\newtheorem{thm}{Theorem}
\newtheorem{lem}{Lemma}[section]

\newtheorem{prop}{Proposition}[section]

\newtheorem{conj}{Conjecture}

\theoremstyle{definition}
\newtheorem{defn}{Definition}[section]

\theoremstyle{remark}

\title{ On the conservation laws and the structure of the nonlinearity for SQG and its generalizations }
\author{  Philip Isett\thanks{Department of Mathematics, Caltech. }, Andrew Ma\thanks{ASUS Robotics and AI Center. }
}
\date{ }

\begin{document}
\maketitle
\begin{abstract}
Using a new definition for the nonlinear term, we prove that all weak solutions to the SQG equation (and mSQG) conserve the angular momentum.  This result is new for the weak solutions of [Resnick, '95] and rules out the possibility of anomalous dissipation of angular momentum.  We also prove conservation of the Hamiltonian under conjecturally optimal assumptions, sharpening a well-known criterion of [Cheskidov-Constantin-Friedlander-Shvydkoy, '08].  Moreover, we show that our new estimate for the nonlinearity is optimal and that it characterizes the mSQG nonlinearity uniquely among active scalar nonlinearities with a scaling symmetry.

\end{abstract}

\tableofcontents

\section{Introduction}

Many important models in fluid dynamics are active scalar equations, which 
in general have the form
\ALI{
\pr_t \th + T^\ell \th \nb_\ell \th &= 0, \qquad \widehat{T^\ell \th}(\xi) = m^\ell(\xi) \hat{\th}(\xi),
}
where the velocity field advecting the scalar is given by a Fourier multiplier applied to the scalar, and the summation convention is used to indicate a sum over repeated indices.  Of primary interest to this paper are the SQG equation ($T^\ell = -\epsilon^{\ell a} \nb_a (-\De)^{-1/2}$), the 2D Euler equation in vorticity form ($T^\ell = \epsilon^{\ell a} \nb_a \De^{-1}$), and the mSQG family, which interpolates between the two $T^\ell = -\ep^{\ell a} \nb_a (-\De)^{-1+ \de}$, $\de \geq 0$.  Here $\ep^{\ell a}$ is the Levi-Civita symbol, which is characterized by $\ep^{j\ell} = -\ep^{\ell j}$, $\ep^{12} = 1$.  This family of equations has been studied from the point of view of existence of weak solutions \cite{resnick1995,marchand2008existence} well-posedness or ill-posedness of the equation and patch dynamics \cite{chae2012generalized,yu2019remarks,cordoba2018uniqueness,cordoba2022non,jeong2023well,zlatos2023local}, constructing invariant measures \cite{nahmod2018global}, constructing global solutions \cite{cordoba2019global,hmidi2023emergence}, constructing blow up of patch solutions \cite{kiselev2016finite}, and for applications in a geophysical context (see \cite{constantin2008global,pierrehumbert1994spectra,schorghofer2000energy,smith2002turbulent} and the references therein).  When we consider more general equations, we will assume that the multiplier $m$ is smooth away from the origin in $\widehat{\R}^n$ unless otherwise stated, and that it has real-symmetry $m(-\xi) = \overline{m(\xi)}$ so that $T^\ell \th$ is real-valued whenever $\th$ is.

Fundamental to studying these equations are the conservation laws that are satisfied by the solutions.  When the velocity field is incompressible, which is equivalent to the multiplier being angular ($m(\xi) \cdot \xi = 0$), sufficiently regular solutions conserve all $L^p$ norms as the measure of the set $\{ x : a \leq \th(x) \leq b \}$ remains invariant under the evolution of the equation for any $a, b \in \R$.  For sufficient conditions for $L^2$ norm conservation for mSQG we refer to \cite{wang2023energySQG}.

Of equal importance is the Hamiltonian of the flow, which exists in 2 dimensions when the multiplier is angular and odd. For the mSQG equations, the Hamiltonian is given up to a constant by
\ali{
H(t) &= \int |\La^{-1+\de} \th(t,x)|^2 dx,
}
where $\La$ is the Fourier multiplier $\La = \sqrt{-\De}$.  (For a more general angular and odd multiplier $m$, the Hamiltonian exists but may not be coercive; see \cite{isettVicol} for a definition.)  Note that the Hamiltonian is well-defined for solutions that are square integrable in time taking values in the Sobolev space $\dot{H}^{-1+\de}$ (i.e. $\th$ is of class $L_t^2 \dot{H}^{-1+\de}$), which, as we shall soon elaborate, is the lowest regularity in which the nonlinearity in the equation is well-defined.

Finally, solutions to the mSQG equations with sufficient regularity and decay conserve the mean, impulse and angular momentum defined by
\ali{
M(t) = \underbrace{\int \th(t,x) dx}_{\textrm{mean}}, \qquad \vec{I}(t) = \underbrace{\int_{\R^2} x \, \th(t,x) dx}_{\textrm{impulse}}, \qquad A(t) = \underbrace{\int_{\R^2} |x|^2 \th(t,x) dx}_{\textrm{angular momentum}} \label{eq:conservationLaws}
}
The first theorem we present concerns the conservation of these last three quantities.  We state a more precise version of this result in Theorem~\ref{thm:angMomentum} below.
\begin{thm} \label{thm:angMoment1} Every weak solution to the mSQG equation ($0 \leq \de \leq 1/2$) (of class $L_t^2 \dot{H}^{-1+\de}(I \times \R^2)$) 
conserves the mean, impulse and angular momentum (defined in a generalized sense).
\end{thm}
 There appears to be no precedent for this type of criterion for angular momentum conservation in the literature, and the proof is quite different from the better-studied conservation of energy.  The result is new for the SQG equation, and in particular applies to the weak solutions constructed by Resnick \cite{resnick1995} and Marchand \cite{marchand2008existence}.  

Since $L_t^2 \dot{H}^{-1+\de}$ is the lowest regularity in which the nonlinearity is defined (as we will elaborate later), this result implies that there are no weak solutions that violate conservation of the mean, impulse or angular momentum among any class that can presently be defined.  Thus, in contrast to the situation for the Hamiltonian, anomalous dissipation of angular momentum is not possible. 

Our next main theorem concerns the optimal regularity under which one can show the Hamiltonian is conserved.  We refer to Section~\ref{sec:basicfunctionSpaces} below for the definition of the Besov space.
\begin{thm}[Onsager singularity theorem] \label{thm:OnsagSingularity} Let $0 \leq \de < 1/2$ and let $\th$ be a weak solution to mSQG of class $L_t^3 \dot{H}^{-1+\de}$ with $\La^{-1+\de} \th$ of class $L_t^3 \dot{B}_{3,c(\N)}^{(1+\de)/3}$ on either $\T^2$ or on $\R^2$.  Then the Hamiltonian $H(t)$ is almost everywhere constant in time.  In particular, $\th \in L_t^\infty \dot{H}^{-1+\de}$.  Furthermore, the same conservation law holds true in the case of SQG ($\de = 1/2$) if we assume $\th \in L_{t,x}^3 \cap L_t^3 \dot{H}^{-1/2}$.
\end{thm}

The problem of establishing criteria for energy conservation and other conservation laws under optimal hypotheses has attracted a large amount of attention since the basic results of \cite{eyink, CET}.  The literature is now too vast to survey here, but we note the results \cite{duchonRobert,energyCons, IOheat, bardos2018onsager,feireisl2017regularity,drivas2018onsager,drivas2019onsager,wang2023energy,deR2023dissipation}, which pertain specifically to the Euler equations.  The paper \cite{energyCons} by Cheskidov, Constantin, Friedlander and Shvydkoy establishes energy conservation in the incompressible Euler equations for velocity fields of class $L_t^\infty L_x^2 \cap L_t^3 B_{3,c(\N)}^{1/3}$.  Our theorem expresses a slightly sharper result, as we assume only $L_t^3 L_x^2 \cap L_t^3 \dot{B}_{3,c(\N)}^{1/3}$ control of the velocity field (which satisfies $\| v\|_{L^2} \approx \| \La^{-1} \th \|_{L^2}$), while the $L_t^\infty L_x^2$ regularity of the velocity field is a conclusion that follows from the conservation of energy.  This improvement is obtained by a more careful study of what we call the ``test function step'' in the proof.

Our results are conjecturally sharp and our approach is novel in that Theorem~\ref{thm:OnsagSingularity} is the first proof of energy conservation at the critical regularity that proceeds at the level of the vorticity equation.  In this way, we obtain a proof of energy conservation that unifies both the 2D Euler and SQG equations.  Previously, Hamiltonian conservation for SQG solutions in $L_{t,x}^3$ in a periodic domain was proven in \cite{isettVicol}.  By the recent solution of the 2D Onsager conjecture in \cite{giri20232d}, the regularity exponent $1/3$ in the 2D Euler result is sharp, and by the generalized Onsager conjecture (see e.g. \cite{klainerman2017nash,buckmaster2020convex}), the assumption on the regularity is expected by analogy to be sharp for the full range $0 \leq \de \leq 1/2$.  The best known results on constructing solutions to SQG that violate Hamiltonian conservation are obtained in \cite{buckShkVicSQG,isett2021direct}; see also \cite{ma2022some} for the mSQG generalization.  

In forthcoming work by the first author and S. Looi, which builds on the methods of the present paper, the full Onsager conjecture for Hamiltonian conservation in the SQG equation will be solved \cite{isettLooi}.  This proof constructs solutions to SQG that fail to conserve the Hamiltonian with any regularity $(-\Delta)^{-1/4} \th \in C_t C_x^{1/2 - \ep}$, and in particular uses the double divergence form of the nonlinearity first obtained in the present work.  We also remark that Theorem~\ref{thm:OnsagSingularity} and its proof extend in a straightforward way to prove Hamiltonian conservation under the same assumptions for any active scalar equation whose multiplier $m$ is odd, angular and degree $-1 + 2 \de$ homogeneous, $0 < \de \leq 1/2$. 


	
	Theorem~\ref{thm:OnsagSingularity} is the first to prove Hamiltonian conservation for mSQG in the range $0 < \de < 1/2$, which encounters new analytical difficulties.  Namely, although there is a simple proof of energy conservation available for Euler ($\de = 0$) and also a short, rather different proof of Hamiltonian conservation for SQG ($\de = 1/2$), such a direct approach for the mSQG family ($0 < \de < 1/2$) does not seem to be available.  Part of the difficulty is to take advantage of the special cancellations in the nonlinearity that allow the nonlinear term to be defined when the solution has negative regularity.  
	
	In fact, when the multiplier is odd and homogeneous, the nonlinearity satisfies the following bound.
\begin{thm} \label{thm:oddBound}  Let $m$ be odd and degree $-1 + 2\de$ homogeneous, $\de > 0$.  Then there is a constant depending on $m$ such that 
\ali{
\left| \int \phi(x) \nb_\ell [T^\ell[\th] \th] dx \right| &\lesssim_m \| \nb^2 \phi \|_{L^\infty} \| \th \|_{\dot{H}^{-1+\de}}^2 \label{eq:oddMultBound}
}
for all $\phi \in \SS(\R^d)$ in the Schwartz class and for all
\ALI{
\th \in \SS_0(\R^d) \coloneqq \left\{ f \in \SS(\R^d) ~:~ \supp \hat{f} \mbox{ is a compact subset of } \widehat{\R^d} \setminus \{0 \} \right \}
}
The same result holds also on $\T^d$, where we consider $\phi$ and $\th \in C^\infty(\T^d)$.
\end{thm}

Theorem~\ref{thm:oddBound} gives a new way of defining the nonlinearity in the active scalar equation for solutions of class $\th \in L_t^2 \dot{H}^{-1+\de}$ (see equation~\eqref{eq:defnEqn} below).  
This ability to formulate the nonlinearity in a weak function space is crucial to the proof of existence of weak solutions \cite{resnick1995,marchand2008existence} to the SQG equation by compactness methods.  
Note that by scaling considerations (counting derivatives), if we restrict to $\th \in \dot{H}^{-1+ \de}$, we cannot expect to use fewer than $2$ derivatives in $L^\infty$ for the norm on the test function $\phi$.  The method of proof also applies to a large class of more general nonlinearities with similar structure (see e.g. Section~\ref{sec:remkGen}).

The standard approach to approaching an estimate such as in Theorem~\ref{thm:oddBound} would be based on the theory of Calder\'{o}n commutators.  That is, because the operator $T$ is anti-symmetric, we can rewrite the nonlinearity in a commutator form
\ALI{
\int \nb_\ell \phi T^\ell \th ~\th dx &= -\int [T^\ell, \nb_\ell \phi] \th dx - \int \nb_\ell \phi \th T^\ell \th dx \\
\int \nb_\ell \phi T^\ell \th ~\th dx &= -\fr{1}{2} \int [T^\ell, \nb_\ell \phi] \th dx
}
To prove Theorem~\ref{thm:oddBound} one then must show that the commutator $[T^\ell, \nb_\ell \phi]$ has one degree of $L^2$ smoothing compared to $T^\ell$, with a constant depending on the Lipschitz norm of $\nb \phi$.  Although the previous definitions of weak solution to SQG require control of at least three derivatives of $\phi$ in their definition (see \cite{marchand2008existence, cheng2021non}), we do believe that the machinery of Calder\'{o}n commutators (e.g. the $1D$ analysis in \cite[Chapter 4]{muscalu138classical}) can be generalized to prove the bound in \eqref{eq:oddMultBound}, at least in the case of $\R^d$, but we could not find a reference for this bound.  

Our definition of the nonlinearity, however, is independent of the commutator structure, and our method of estimating it extends easily to the periodic setting and to any degree of homogeneity for the multiplier.  
Moreover, we obtain the following improved estimate in the special case of mSQG.
\begin{thm}[mSQG bound] \label{thm:mSQGbound} For any $0 \leq \de < \infty$ and $T^\ell = \ep^{\ell a} \nb_a (-\De)^{-1+\de}$, one has the bound
\ali{
\left| \int \phi(x) T^\ell \th ~ \nb_\ell \th dx \right| &\lesssim_\de \| \mrg{\nb}^2 \phi \|_{L^\infty} \| \th \|_{\dot{H}^{-1+\de}}^2 \label{eq:specialmSQGBound}
}
for all $\phi \in \SS(\R^2)$, $\th \in \SS_0(\R^2)$, where $\mrg{\nb}^2_{j\ell} \phi {=} \nb_j \nb_\ell \phi {-} \fr{1}{2} \De \phi \de_{j\ell}$ is the trace-free part of the Hessian of $\phi$.
\end{thm}
This theorem is easy to prove in the case of $2D$ Euler (write the nonlinearity as $\mbox{div } \mbox{div } v \otimes v^\perp$, $v^\ell = T^\ell \th$), but for $\de > 0$ (in particular for SQG) it apparently requires a more involved analysis.

We show in the following theorem that the norm $\| \mrg{\nb}^2 \phi \|_{L^\infty}$ cannot be replaced by any weaker norm.
\begin{thm} \label{thm:sharpXBd}  Let $0 \leq \de < \infty$.  
If there exists a bound for the mSQG nonlinearity of the form
\ALI{
\left| \int \phi(x) \ep^{\ell a} \nb_a (-\De)^{-1+\de} \th \nb_\ell \th dx \right| &\lesssim_\de \| \phi \|_{X} \| \th \|_{\dot{H}^{-1+\de}}^2 
}
for all $\phi \in C_c^\infty(\R^2)$ and $\th \in \SS_0(\R^2)$ for some semi-norm $\| \cdot\|_X$, then we must have $\| \mrg{\nb}^2 \phi \|_{L^\infty} \lesssim_\de \| \phi \|_X$.
\end{thm}
Our next theorem shows that the estimate of Theorem~\ref{thm:mSQGbound} completely characterizes the mSQG nonlinearity among active scalar nonlinearities with similar translation and scaling symmetries.  This theorem shows that the special estimate of Theorem~\ref{thm:mSQGbound} 
in terms of the trace-free part of the Hessian is quite delicate and cannot be expected to hold for more general nonlinearities.
\begin{thm}[Characterization of the nonlinearity] \label{thm:nonlinearCharacterization}  Let $0 \leq \de \leq 1/2$, suppose that $m$ is degree $-1 + 2\de$ homogeneous, not necessarily smooth away from $0$, and assume that the bound
\ALI{
\left| \int \phi T^\ell[\th] \nb_\ell \th dx \right| &\lesssim \| \mrg{\nb}^2 \phi \|_{L^\infty} \| \th \|_{\dot{H}^{-1+\de}}^2 
}
holds for $\phi \in C_c^\infty$ and $\th \in \SS_0(\R^2)$.  Then $m^\ell$ is a constant multiple of the mSQG multiplier.
\end{thm}

Along the way to proving Theorem~\ref{thm:nonlinearCharacterization} we elaborate how the nonlinearity is not bounded on any Sobolev space rougher than $\dot{H}^{-1+\de}$.  The theorem actually applies to the range $0 \leq \de < \infty$ with $\de \neq 1$, but our method breaks down in the special case $\de = 1$ (which is also the case where the mSQG nonlinearity vanishes) and we do not know if the theorem holds in this case, although we are able to make some conclusions about the structure of $m$ for a potential counterexample (in particular the velocity field cannot be incompressible).  

The special estimate \eqref{eq:specialmSQGBound} is intimately tied to the conservation of impulse and angular momentum for the mSQG equation, but is not quite strong enough to prove these conservation laws.  
Namely, the functions $\{ 1, x_1, x_2, |x|^2 \}$ that appear in the definitions of mean, impulse and angular momentum (equation~\eqref{eq:conservationLaws}) span precisely those functions for which the trace-free part of the Hessian vanishes.  It is tempting to try to prove angular momentum conservation by applying the estimate in Theorem~\ref{thm:mSQGbound} to a version of $|x|^2$ that has been cut off in space.  However, the bound in Theorem~\ref{thm:mSQGbound} is not strong enough for this strategy to succeed in proving angular momentum conservation.

Nonetheless, what we will show in Theorem~\ref{thm:angMomentum} below is that conservation of angular momentum and impulse holds for {\bf all} weak solutions in $L_t^2 \dot{H}^{-1+\de}$, which, as we have noted before, is the lowest regularity in which the nonlinearity makes sense.  
We now explain the precise formulation. 

Upon first sight, the definition of angular momentum seems to require that $|x|^2 |\th(t)| \in L^1$ to make sense, but we can actually define this quantity for a much rougher class of data.
\begin{defn}  We say a distribution $f \in \DD'(\R^2)$ is in the {\bf angular momentum class} if it can be written in the form $f(x) = A(x) + \nb_j B^j(x) + \nb_j \nb_\ell C^{j\ell}(x)$ where $|x|^2 |A| + |x| |B| + |C| \in L^1(\R^2 \setminus B_R)$ for some ball $B_R$ around the origin of radius $R$.
\end{defn}
Angular momentum can be defined for any distribution in the angular momentum class as follows.
\begin{defn}  Let $f \in \DD'(\R^2)$.  We say $f$ has well-defined angular momentum if, for any bump function $\chi \in C_c^\infty(\R^2)$ that is equal to $1$ in a neighborhood of $0$, the limit
\ALI{
\lim_{R \to \infty} \int {\chi\left(\fr{x}{R}\right)} |x|^2 f(x) \textrm{`}dx\textrm{'} 
}
exists and is independent of the choice of $\chi$.  Here `$dx$' means that the integral is interpreted in the sense of distributions.
\end{defn}
Generalized impulse can be defined similarly.  It is a simple exercise to show that any distribution in the angular momentum class has well-defined angular momentum and impulse\footnote{In fact, one can also have terms of the form  $\nb_i \nb_j \nb_\ell D^{ij\ell}$ with $|x|^{-1} |D| \in L^1(\R^2 \setminus B_R)$ to have well-defined angular momentum and impulse.}.  (See for example the proof of Lemma~\ref{lem:conserveLem} below.)  Our theorem on conservation of impulse and angular momentum is the following.

\begin{thm}[Conservation laws] \label{thm:angMomentum}  If $0 \leq \de < \infty$ and $\th$ is an mSQG solution in $L_t^2 \dot{H}^{-1+\de}(I \times \R^2)$, $I$ an open interval, then for any time $t_0 \in \bar{I}$, $\th(t_0)$ has well-defined angular momentum (impulse) if and only if $\th(t)$ has well-defined angular momentum (impulse) at every time $t \in I$ and the angular momentum (impulse) is constant in time.  Furthermore, $\th(t)$ belongs to the angular momentum class at a time $t_0 \in \bar{I}$ if and only if $\th(t)$ belongs to the angular momentum class at every time $t \in \bar{I}$.
\end{thm}

At this point we pause to state precisely our definition of weak solution.  We also expound the remark that, much like for other evolution equations, one can define a unique restriction of an mSQG solution to every time slice in its domain in such a way that the family of distributions $\th(t)$ are continuous in $t$ with respect to weak convergence in $\DD'$.  To make these remarks we need the definition of the fixed-time nonlinearity from Theorem~\ref{thm:oddBound}.
\begin{defn} \label{defn:nonlinear} We introduce the notation $B(\psi, \th)$ to denote the unique extension of the linear-quadratic form 
$B(\psi, \th) =  \int \psi T^\ell[\th] \nb_\ell \th dx$ 
to the space $C_c^\infty \times \dot{H}^{-1+\de}$, whose existence follows from Theorem~\ref{thm:oddBound}.  We use $B(\psi, \th_1, \th_2)$ to denote the unique trilinear form, symmetric in $\th_1$, $\th_2$, such that $B(\psi, \th, \th) = B(\psi, \th)$.
\end{defn}
Specifically, we define $B(\psi, \th) = \lim_{n \to \infty} B(\psi, \th_n)$ where each $\th_n \in \SS_0(\R^2)$ and $\th_n \to \th$ in $\dot{H}^{-1+\de}$.

In terms of this definition, a weak solution to an active scalar equation with odd, angular multiplier is a distribution of class $\th \in L_t^2 \dot{H}^{-1+\de}(I \times \R^2)$ such that for all $\phi(t,x) \in C_c^\infty$ one has
\ali{
-\int \pr_t \phi(t,x) \, \th(t,x)\, \textrm{`}dx dt\textrm{'} + \int B(\phi(t), \th(t))dt &= 0, \label{eq:defnEqn}
}
where again we use the single quotes to denote the pairing of a distribution with a test function.  The time-integral is well-defined by virtue of the estimate of Theorem~\ref{thm:oddBound} and an approximation of $\th(t)$ by simple functions of time taking values in $\dot{H}^{-1+\de}$.  We define the restriction of $\th$ to a time slice $t_0 \in \bar{I}$ by the equation $\th(t_0,x) = \th(t,x) - \int_{t_0}^t B( \cdot, \th(s) ) ds$.  
At first sight the distribution on the right hand side (call it $u$) appears to depend on $t$, but the definition of a weak solution implies that $\pr_t u = 0$ in the weak sense, so that $u$ can be identified with a unique time-independent distribution on $\R^2$.

\paragraph{Optimality of the definition}

With these definitions in hand, the final question we wish to address in this work is: {\it Is our definition of the nonlinearity the best possible definition?}  To compare different definitions, there are four criteria one can consider:
\begin{enumerate}
	\item  Which definition applies at the lowest possible regularity for $\th$?
	\item  Which definition yields the best possible estimate for the nonlinearity?
	\item  Which definition yields the best possible continuity properties for the nonlinearity?
	\item  Which definition yields a continuous extension to the largest possible space of test functions? 
\end{enumerate}

Regarding the first point, it will be shown in the course of proving Theorem~\ref{thm:nonlinearCharacterization} that the nonlinearity does not make sense for $\th$ below $\dot{H}^{-1+\de}$, so that our definition is optimal in this regard.  Regarding the second point, Theorem~\ref{thm:sharpXBd} shows that the bound of Theorem~\ref{thm:mSQGbound} is the best possible one.  (Recall that in contrast the existing definitions \cite{marchand2008existence,cheng2021non} require at least $3$ derivatives of the test function.)  Regarding points three and four, we have the following theorems on continuous extensions of the nonlinearity.

Our first theorem shows that the test function need not even be $C^2$ but can be any distribution whose weak-Hessian is bounded.  Moreover, the continuity is stronger than what is implied by the boundedness of Theorem~\ref{thm:oddBound}.
\begin{thm}\label{thm:extension}  Let $m$ be odd, angular and degree $-1+2\de$ homogeneous.  Let $\dot{W}^{2,\infty}$ be the (non-Hausdorff) space of distributions with weak Hessian in $L^\infty$.  Then there is a unique extension of the linear-quadratic form $B(\psi, \th)$ to $\dot{W}^{2,\infty} \times \dot{H}^{-1+\de}$ that satisfies
\ALI{
B(\Psi, \Th) &= \lim_n B(\Psi_n, \Th_n)
}
whenever $\nb^2\Psi_n$ converges to $\nb^2 \Psi$ in $L^\infty$ weak-*, and $\Th_n \to \Th$ strongly in $\dot{H}^{-1+\de}$.  The extended linear-quadratic form also satisfies a bound of the form $|B(\Psi,\Th)| \lesssim \| \nb^2 \Psi \|_{L^\infty} \| \Th \|_{\dot{H}^{-1+\de}}^2$, (perhaps with a different optimal constant).
\end{thm}
We must stress that the continuous extension guaranteed by this theorem does not follow from the bound of Theorem~\ref{thm:oddBound} by a simple density argument since test functions are not norm dense in $\dot{W}^{2,\infty}$.  The content of the theorem is that the nonlinearity has a stronger property that allows continuity in a weaker topology.

For the mSQG nonlinearity there is again an improvement, which is that we can handle any test function whose weak Hessian has bounded trace-free part.  We let $\mrg{W}^{2,\infty}$ denote this non-Hausdorff space of distributions.
\begin{thm} \label{thm:mSQGextend}  Let $m$ be the mSQG multiplier, $0 \leq \de < \infty$.  Let $\mrg{W}^{2,\infty}$ be the (non-Hausdorff) space of distributions whose weak Hessian has trace-free part in $L^\infty$.  Then there is a unique extension of the linear-quadratic form $B(\psi, \th)$ to $\mrg{W}^{2,\infty} \times \dot{H}^{-1+\de}$ that satisfies
\ali{
B(\Psi, \Th) &= \lim_m \lim_n B(\Psi_n, \Th_m) \label{equation:doubleLim}
}
whenever $\psi_n$ is a sequence in $\SS(\R^2)$ converging in $\SS'(\R^2)$ to $\Psi$, and $\Th_m$ is a sequence in $\SS_0(\R^2)$ converging strongly in $\dot{H}^{-1+\de}$ to $\Th$.  This extended linear-quadratic form satisfies an estimate of the form $|B(\Psi,\Th)| \lesssim \| \mrg{\nb}^2 \Psi \|_{L^\infty} \| \Th \|_{\dot{H}^{-1+\de}}^2$, (perhaps with a different optimal constant).
\end{thm}
The existence of the inner limit in \eqref{equation:doubleLim} is immediate, while the existence of the outer limit follows from the boundedness in Theorem~\ref{thm:mSQGbound}.  
We remark that the extended operator has an additional continuity analogous to Theorem~\ref{thm:extension}. namely that $B(\Psi_n, \Th_n) \to B(\Psi, \Th)$ if $\mrg{\nb}^2 \Psi_n$ converges to $\mrg{\nb}^2 \Psi$ in $L^\infty$ weak-* and $\Th_n \to \Th$ in $\dot{H}^{-1+\de}$.  However, it is not clear whether this continuity property uniquely characterizes the extension as we do not know whether $C_c^\infty$ is sequentially weak-* dense in $\mrg{W}^{2,\infty}$.

The statement of Theorem~\ref{thm:mSQGextend} implicitly takes advantage of the following fact\footnote{Without this theorem we can state that the nonlinearity extends to test functions in $\mrg{W}^{2,\infty}$ that are tempered distributions, but Theorem~\ref{thm:traceFree} assures that such distributions contain all of $\mrg{W}^{2,\infty}$.}, which we were motivated by Theorem~\ref{thm:mSQGextend} to discover and believe to be of independent interest: Every distribution whose Hessian has bounded trace free part is a tempered distribution.  More precisely, we have the following Theorem
\begin{thm}[\cite{isettTraceFree}] \label{thm:traceFree} Let $\Psi \in \DD'(\R^d)$ be a distribution such that the trace free part of the weak Hessian $\mrg{\nb}^2_{j\ell} \Psi = \nb_j\nb_\ell \Psi - \fr{1}{d}\De \Psi \de_{j\ell}$ belongs to $L^\infty$.  Then $\Psi$ is a tempered distribution.  Moreover, $\Psi$ has log-Lipschitz continuous first derivatives and satisfies the growth estimates
\ALI{
|\Psi(x)| = O(|x|^2 \log |x|), \qquad |\nb \Psi(x)| = O(|x| \log |x|) \qquad \mbox{ as } |x| \to \infty.
}
\end{thm}
These estimates are sharp in view of the example $\Psi(x) = |x|^2 \log |x|^2 = 2 |x|^2 \log |x| \in \mrg{W}^{2,\infty}$.  

Two open questions related to this theorem are the following: First, is the space of $C_c^\infty$ functions dense in $\mrg{W}^{2,\infty}$ in the weak-* topology?  (If so, the weak-* continuity and boundedness properties would uniquely characterize the continuous extension of the nonlinearity.) Likewise, is the space $\mrg{W}^{2,\infty}$ complete in the sense of a non-Hausdorff space?  That is, given $\Psi_n$ Cauchy in the semi-norm, does there exist $\Psi \in \mrg{W}^{2,\infty}$ that is the (non-unique) limit of $\Psi_n$?  Completeness is true for the space $\dot{W}^{2,\infty}$ (see the Appendix), but it is not obvious whether completeness holds in $\mrg{W}^{2,\infty}$.

We conclude this introduction with the organization of the paper and some basic definitions.

\subsection{Organization of the paper}

Sections~\ref{sec:testFunctionStep}, \ref{sec:testOnR2} and \ref{sec:hamiltonian} are devoted to proving the conservation of the Hamiltonian stated in Theorem~\ref{thm:OnsagSingularity}.  These sections rely on a divergence form expression that will be a basic tool in proving all of the remaining Theorems of the paper.  This tool is developed in Section~\ref{sec:deriveBilinear}.  In Section~\ref{sec:definingNonlinearity} we derive our new definition of the nonlinearity.  We prove Theorem~\ref{thm:extension} on the continuous extension of the operator in Section~\ref{sec:ctsExtend}.  The special estimates and continuity properties of Theorems~\ref{thm:mSQGbound} and~\ref{thm:mSQGextend} are obtained by establishing a certain divergence form in Section~\ref{sec:mSQGSpecial}.  Section~\ref{sec:mSQGSpecial} also establishes the angular momentum conservation stated in Theorems~\ref{thm:angMoment1} and \ref{thm:angMomentum}.  

Theorem~\ref{thm:sharpXBd} on the sharpness of the mSQG bound is proven in Section~\ref{sec:Sharpness}.  Theorem~\ref{thm:nonlinearCharacterization}, which characterizes the mSQG nonlinearity, is proven in Section~\ref{sec:uniqueness}.

The paper concludes with an Appendix that lays out the important properties of the homogeneous function spaces used in the paper.  

Throughout the paper, the results are obtained in both the periodic setting and the nonperiodic setting wherever it makes sense to do so.  We have made efforts to clarify which setting we consider at each instance where the presentation switches settings.

\paragraph{Acknowledgments}  The first author acknowledges the support of a Sloan fellowship and the NSF grant DMS 2055019.

\subsection{Definitions of basic function spaces} \label{sec:basicfunctionSpaces}

For $X$ a Banach space of distributions on $\T^d$ or $\R^d$, we denote by $L_t^p X$ to be the space of distributions $u$ such that $u$ is strongly measurable in time with values in $X$ and the norm $\| \|u(t)\|_X \|_{L^p}$ is finite.

The Banach space $\dot{H}^{-1 + \de}$ is defined and its properties are studied in detail in the Appendix.

We define Littlewood-Paley projections with the conventions that $P_{\leq 0}$ has a Fourier multiplier $\hat{\chi}_{\leq 0}(\xi)$ that equals $1$ in $|\xi| \leq 1/2$ and has compact support in $|\xi| \leq 1$.  With this convention $P_{\leq q}$ has Fourier multiplier $\hat{\chi}_{\leq q} \equiv \hat{\chi}_{\leq 0}(2^{-q} \xi)$ and we define $P_q \equiv P_{\leq q+1} - P_{\leq q}$, which has a Fourier multiplier $\hat{\chi}_q$ localized to $2^{q-1} \leq |\xi| \leq 2^{q+1}$. 

We say that a distribution $u$ belongs to the homogeneous Besov space $\dot{B}^\a_{p, c(\N)}$ if there exists a sequence of $C_c^\infty$ functions $u_n$ such that $\| u - u_n \|_{\dot{B}^\a_{p, c(\N)}} \to 0$, where
\ALI{
\| u \|_{\dot{B}^\a_{p, c(\N)}} &= \sup_q 2^{\a q} \| P_q u \|_{L^p}.
}
Distributions in this class inherit properties from test functions such as $\| P_q u \|_{L^p} = o(2^{-\a q})$ as $q \to \infty$.

\section{Hamiltonian conservation: The test function step} \label{sec:testFunctionStep}

\newcommand{\lh}[1]{\left \| #1 \right \|_{L^3_t \dot{H}^{-1 + \delta}_x}}


%
The first step in the proof of Hamiltonian conservation is to justify multiplying the equation by a test function that is not necessarily smooth with compact support.  This step is less technical on the torus as opposed to the full Euclidean space, so we concentrate on the $\T^2$ case first.  We explain the additional technicalities for the Euclidean case in Section~\ref{sec:testOnR2}.

Let $\th$ be an mSQG solution with $\delta \in [0, 1/2]$ fixed and $\th \in L_t^3 \dot{H}^{-1+\de}$ with $\Lambda^{-1 + \delta}\th \in L^3_t B^\a_{3, \infty}$ for $\a \geq (1 + \delta)/3$.  Note that on the torus the assumption $\th \in L_t^3 \dot{H}^{-1 + \de}$ follows from our other assumption on $\La^{-1+\de} \th$.  Indeed, if $\Lambda^{-1 + \delta}\th \in L^3_t B^\a_{3, \infty}$ for $\a > 0$ then $\th \in L^3_t \dot{H}^{-1 + \delta}_x$ due to the Littlewood-Paley inequality
\[
	\|\th \|_{\dot{H}^{-1 + \delta}(\T^2)}^2 \sim
	\sum_{q \geq 0}\|2^{(-1 +\delta)q}P_q \th \|^2_{L^2(\T^2)} 
	\leqc \sum_{q\geq 0} \left ( 2^{-\a q}\| \La^{-1 +\delta}\th \|_{B^\a_{3, \infty}(\T^2)} \right )^2
	\leqc
	\| \La^{-1 +\delta}\th \|_{B^\a_{3, \infty}(\T^2)}^2
\]  
We refer to Section~\ref{sec:basicfunctionSpaces} for the definitions of the operators $P_q$ and the Besov space $B^\a_{3, \infty}$.

Define $\phi(t) \in C_c^\infty(\R)$ to be a radial bump function of unit integral and let $\phi_\tau(t) = \tau^{-1} \phi(\tau^{-1}t)$ be a mollifying kernel in time. In addition, let $\eta(t) \in C_c^\infty(I)$ be a test function in time.  
Define the test function
\[
	\psi_{\hat{q}, \tau} \coloneqq \left (  \eta(t) \left ( P_{\leq \hat q}^2\La^{2(-1 + \delta)}\th \right )_\tau \right )_\tau \in C^\infty_c(\T^2 \times \R)
\]
where $(\cdot )_\tau$ denotes mollification in time with $\phi_\tau$. 


We may safely apply $\psi_{\hat{q}, \tau}$ as a test function to mSQG in order to obtain
\ali{\label{eq: mSQG with test function}
	- \int_\R\int_{\T^2} \th  \pr_t \psi_{\hat{q}, \tau} \, dxdt - \int B[\psi_{\hat{q}, \tau}(t), \th(t), \th(t)] dt  = 0
}
where $B[\psi_{\hat{q}, \tau}, \th, \th] \coloneqq \int_{\T^2} u^l \th \nb_l \psi_{\hat{q}, \tau} \, \textrm{`}dx\textrm{'}$ is the nonlinear term defined in Definition~\ref{defn:nonlinear}.  Using the self-adjointness of $P_{\leq \hq}^2$ and $( \cdot)_\tau$, we have that
\ALI{
\frac 12 \int_{\R \times \T^2} \eta'(t)\cdot \left | \left ( \La^{-1 + \delta}P_{\leq \hat{q}}\th \right )_\tau  \right |^2 \, dxdt 
&= \int_\R\int_{\T^2} \th  \pr_t \psi_{\hat{q}, \tau} \, dxdt 
}

By the definition of a weak solution, we get 
\ali{
\label{eq: d/dt time mollified Hamiltonian}	
	\frac 12 \int_{\R \times \T^2} \eta'(t)\cdot \left | \left ( \La^{-1 + \delta}P_{\leq \hat{q}}\th \right )_\tau  \right |^2 \, dxdt 
	=\int_{\R \times \T^2} u^l \th \nb_l  \left ( \eta \left  ( P_{\leq \hat q}^2 \La^{2(-1 + \delta)}\th \right )_\tau \right )_\tau \, \textrm{`}dxdt\textrm{'}
}
where the formal integral on the right hand side is interpreted in terms of the linear-quadratic form for mSQG as 
\ALI{
\int B\left[ \left ( \eta \left  ( P_{\leq \hat q}^2 \La^{2(-1 + \delta)}\th \right )_\tau \right )_\tau, \th(t), \th(t)\right] dt
}

We would like to take a limit of $\tau \rightarrow 0$ of both sides of \eqref{eq: d/dt time mollified Hamiltonian}. To do this we use the next Lemma. 
\begin{lem}\label{lem: time mollifcation convergence}
\[
	\psi_{\hat{q}, \tau} \longrightarrow \eta(t) \La^{2(-1 + \delta)}P_{\leq \hat{q}}^2 \th \eqqcolon \psi_{\hat{q}} \text{ in }L^3_t\dot{W}^{2,\infty} \text{ as }\tau \rightarrow 0
\]
\end{lem}
\begin{proof}
For each fixed $\hat{q}$ we note that $\La^{2(-1 +\delta)}P^2_{\leq \hat{q}}: L^3_t\dot{H}^{-1 + \delta}_x \longrightarrow L^3_t \dot{W}^{2,\infty}$ is a bounded linear operator with an operator norm that depends on $\hat{q}$. So
\[
	\| \psi_{\hat{q}, \tau} - \psi_{\hat{q}}  \|_{L^3_t \dot{W}^{2,\infty}}
	=
	\left \| \La^{2(-1 + \delta)}P_{\leq \hat{q}}^2 \left ( (\eta(t) \th_\tau )_\tau - \eta(t) \th \right )\right \|_{L^3_t\dot{W}^{2,\infty}}	
	\leqc_{\hat{q}}
	\left \| (\eta(t) \th_\tau )_\tau - \eta(t) \th \right \|_{L^3_t \dot{H}^{-1 +\delta}_x}
\]
Hence, it is enough to show 
\[
	\lim_{\tau \rightarrow 0} \left \| (\eta(t) \th_\tau )_\tau - \eta(t) \th \right \|_{L^3_t \dot{H}^{-1 +\delta}_x} = 0
\]
We proceed via an approximation argument. Define $\psi \coloneqq \eta(t) \th$.  Banach valued simple functions are dense in $L^3_t\dot{H}^{-1 + \delta}_x$ 
\cite{grafakos2014classical} and \cite{potter}.  Using this face, we can approximate $\psi$ in the $L^3_t\dot{H}^{-1 + \delta}_x$-norm by a smooth in time approximation $f \in C_c^\infty(\R ; \dot{H}^{-1 + \delta}(\T^2))$ 
\cite[Chapter 4, Theorem 4.11]{potter}. Given an $\ep > 0$, take $f \in C_c^\infty(\R ; \dot{H}^{-1 + \delta}(\T^2))$ such that 
$\lh{ \th - f} < \epsilon$.  Define $\psi_{\tau, \tau} \coloneqq \left (\eta(t) \th_\tau \right )_\tau$ and $f_{\tau, \tau} \coloneqq \left (\eta(t) f_\tau \right )_\tau$ and consider the inequality
\ali{
\label{eq: Bochner approx}
	\lh{\psi_{\tau, \tau} - \psi} 
	& \leq
	\lh{\psi_{\tau, \tau} - f_{\tau, \tau}} 
	+ \lh{f_{\tau, \tau} - \eta f}
	+ \lh{\eta f - \psi}	
	\\
	\nonumber
	&\coloneqq
	I + II + III 
}
We will discuss how to estimate each term on the RHS of \eqref{eq: Bochner approx}. Beginning with $I$, $\lh{\psi_{\tau, \tau} - f_{\tau, \tau}} <\ep$
due to a modified argument of the proof that convolutions are bounded linear operators in $L^p$-norms. 

For $II$,  $\lim_{\tau \rightarrow 0} \lh{f_{\tau, \tau} - \eta f} = 0$
by a dominated convergence argument and observing that for sufficiently small $\tau$, the family of functions $\{f_{\tau, \tau}\}$ can yield a sequence of uniformly bounded functions converging to $\eta f$ in $\dot{H}^{-1+\de}$ pointwise uniformly in time. 

Finally for $III$,  $\lh{\eta f - \psi}	\leq \| \eta\|_{L^\infty(\R)} \lh{f - \th} < \epsilon$
by Young's inequality and our assumption $ \lh{f - \th} < \epsilon$. By combining the estimates of $I, II$, and $III$ and taking a limit of \eqref{eq: Bochner approx} we get
$
	\limsup_{\tau \rightarrow 0} 
	\lh{\psi_{\tau, \tau} - \psi} \leq 
	\limsup_{\tau \rightarrow 0} \left ( 2\ep + \lh{f_{\tau, \tau} - \eta f} \right )
	\leq 
	2\ep
$. Since this inequality holds for any $\ep > 0$ we conclude that 
$
	\lim_{\tau \rightarrow 0} 
	\lh{\psi_{\tau, \tau} - \psi} = 0
$
\end{proof}
Next, by the assumption $\th \in L^3_t B^\a_{3, \infty} \subseteq L^3_t \dot{H}^{-1 + \delta}_x$, Lemma \ref{lem: time mollifcation convergence}, and the bound on the nonlinear term given 
in Theorem~\ref{thm:oddBound} we may safely take a limit of $\tau \rightarrow 0$ of line \eqref{eq: d/dt time mollified Hamiltonian}.  Formally, the identity we have established is 
\begin{equation}
	\label{eq: d/dt Hamiltonian}	
	\int_\R \int_{\T^2} \eta'  \left( P_{\leq \hat q} \La^{-1 + \delta}\th \right )^2 dxdt
	=
	-\int_\R \int_{\T^2} u^l \th \nb_l  \left ( \eta P_{\leq \hat q}^2 \La^{2(-1 + \delta)}\th \right ) \textrm{`dxdt'} 
\end{equation}
where in the line above we define
\ali{\label{defn: psi_q}
	\psi_{\hat{q}} \coloneqq \eta(t) P_{\leq \hat q}^2 \La^{2(-1 + \delta)}\th.
}
Rigorously the right hand side of \eqref{eq: d/dt Hamiltonian} must be interpreted as 
\ALI{
-\int_\R \int_{\T^2} u^l \th \nb_l  \left ( \eta P_{\leq \hat q}^2 \La^{2(-1 + \delta)}\th \right ) \textrm{`dxdt'} \coloneqq \int B(\psi_{\hat{q}}(t), \th(t), \th(t)) dt
}
where $B$ is the linear-quadratic form that is bounded on $\dot{W}^{2,\infty} \times \dot{H}^{-1+\de} \times \dot{H}^{-1+\de}$.

We observe that the left hand side of \eqref{eq: d/dt Hamiltonian} is an approximation of the weak derivative of the Hamiltonian, $\frac{d}{dt}H(t)$,  against the smooth bump function in time $-\eta(t)$ and we observe that, by dominated convergence,
\[
	\lim_{\hat q \rightarrow \infty} \int_\R \int_{\T^2} \eta'  \left( P_{\leq \hat q} \La^{-1 + \delta}\th \right )^2 dxdt
	=
	\int_\R \int_{\T^2} \eta'  \left(  \La^{-1 + \delta}\th \right )^2 dxdt
\]
Our goal is thus to show 
$
	\lim_{\hat q \rightarrow \infty}\int_{\R}\eta '(t) \int_{\T^2} \left( P_{\leq \hat q} \La^{-1 + \delta} \th \right )^2 dx dt = 0
$. In other words, we would like to prove that the RHS of \eqref{eq: d/dt Hamiltonian} obeys the limit
\ali{\label{eq: limit in q of B}
\lim_{\hat q \rightarrow \infty}  \int B(\psi_{\hat{q}}(t), \th(t), \th(t)) dt = 0
}
To investigate this limit we will make use of an approximating sequence.

\section{The test function step in Euclidean space} \label{sec:testOnR2}

In this section we show how to generalize the analysis of Section~\ref{sec:testFunctionStep} from the torus to the case of $\R^2$.  The reader interested mainly in the periodic case may wish to skip to Section~\ref{sec:hamiltonian} upon first reading.

We assume $\th \in L_t^3 \dot{H}^{-1+\de}$.  We start by observing that, as in the periodic case
\ali{
\int \eta'(t) \int \left| \La^{-1+\de} \th(t,x)\right|^2 dx dt &= \lim_{\hq \to \infty} \lim_{\tau \to 0} \int \eta'(t) \int \left| P_{\leq \hq}^2 (\La^{-1+\de} \th(t,x))_\tau\right|^2 dx dt  \label{eq:sameAsT2}
}
by a simple approximation argument using density of $\dot{H}^{-1+\de}$-valued simple functions of time.  

We now deviate from the periodic case to write
\ali{
\eqref{eq:sameAsT2} &= \lim_{\hq \to \infty} \lim_{\tau \to 0} \lim_{m \to \infty} \int \eta'(t) \int \left| P_{(-m, \hq]} (\La^{-1+\de} \th(t,x))_\tau\right|^2 dx dt \label{eq:diffFromT2}
}
where $P_{(-m,\hq]} \coloneqq P_{\leq \hq}^2 - P_{\leq -m}$ is a truncation of $P_{\leq \hq}^2$.   Again we have applied dominated convergence using that $\th \in L_t^2 \dot{H}^{-1+\de}$.  

Our next step is to introduce a physical space cutoff $\chi_R(x) = \chi(x/R)$, where $\chi(x) = 1$ on the unit ball and $\chi \in C_c^\infty$.  We claim that
\ali{
\eqref{eq:diffFromT2} &= \lim_{\hq \to \infty} \lim_{\tau \to 0} \lim_{m \to \infty} \lim_{R \to \infty} \int \eta'(t) \int \left| P_{(-m, \hq]} \La^{-1+\de} [\chi_R \th(t,x)_\tau]\right|^2 dx dt \label{eq:diffFromTorusAgain}
}
To prove this claim, we decompose the function on the right hand side as
\ALI{
\Pmhq \La^{2(-1+\de)}[ \chi_R \th_\tau ] &= I + II + III + IV \\
I &= \Pmhq \La^{2(-1+\de)} [P_{\leq -m-2} \chi_R P_{\leq -m-2} \th_\tau] \\
II &= \Pmhq \La^{2(-1+\de)} [P_{(-m-2, \hq+2]} \chi_R P_{\leq -m - 2} \th_\tau] \\
III &= \Pmhq \La^{2(-1+\de)} [P_{\leq \hq + 2} \chi_R P_{(-m-2, \hq+2]} \th_\tau] \\
IV &= \Pmhq \La^{2(-1+\de)} \sum_{k \geq \hq + 1} [P_{k+1} \chi_R P_k \th_\tau + P_k \chi_R P_{k+1} \th_\tau + P_{k+1} \chi_R P_{k+1} \th_\tau]
}
Note that $IV$ contains all the terms involving high frequencies since, if $k \geq \hq+1$, we have that $P_{k+1} \chi_R P_n \th_\tau  $ and $P_{k+1} \th_\tau P_n \th_\tau$ have frequency support disjoint from that of $\Pmhq^2$ for $n < k-1$.

By frequency support considerations $I = 0$, so it suffices to take the limit in $L^2$ of the other  terms.

We start by estimating $II$.  All estimates are taken at a fixed time
\ALI{
\| II(t) \|_{L_x^2} &\lesssim_{m, \hq}  \| P_{(-m-2, \hq+2]} \chi_R \|_{L^\infty} \| P_{\leq -m - 2} \th_\tau \|_{L^2} \\
&\lesssim_{m, \hq}\| P_{(-m-2, \hq+2]} \chi_R \|_{L^\infty} \| \th_\tau(t) \|_{\dot{H}^{-1+\de}}
}
Furthermore, as $R \to \infty$, we have 
\ALI{
\| P_{(-m-2, \hq+2]} \chi_R \|_{L^\infty} &\lesssim_{m,\hq} \| \nb \chi_R \|_{L^\infty} \to 0 \qquad \mbox{ as } R \to \infty
}
and we conclude $II$ converges to $0$ in $L^2_x$ at any time $t$, with $L^2$ norms bounded uniformly by a square-integrable function of $t$.

We now consider the limit of $III$, which is the only term that does not vanish in the limit.  As $\hq$ is fixed, the functions
\ALI{
P_{\leq \hq + 2} \chi_R(x) &= \int \chi_R(y) \eta_{\leq \hq + 2}(x-y) dy
}
are uniformly bounded in $L^\infty$ and converge pointwise (but not uniformly) in $x$ to $1 = \int \eta_{\leq \hq + 2} dy$ as $R \to \infty$.  Using dominted convergence and the $L^2$ boundedness of $\Pmhq \La^{2(-1+\de)}$, it follows that
\ALI{
III(t) &\to \Pmhq \La^{2(-1+\de)} P_{(-m-2, \hq+2]} \th_\tau = \Pmhq \La^{2(-1+\de)} \th_\tau
}
in $L^2_x$ for any time $t$, with $L^2$ norms bounded uniformly by a square-integrable function of $t$.

Finally we estimate $IV$ as follows.  Using the fact that $\Pmhq \La^{2(-1+\de)}$ is bounded on $L^2$ we have
\ALI{
\| IV \|_{L_x^2} &\lesssim_{m,\hq} \sum_{k \geq \hq} \| P_{\approx k} \chi_R \|_{L^\infty} \| P_{\approx k} \th_\tau \|_{L^2_x} \\
&\lesssim_{m,\hq} \sum_{k \geq \hq} \| P_{\approx k} \chi_R \|_{L^\infty}   2^{(1-\de)k}\| \La^{-1 + \de} P_{\approx k} \th_\tau \|_{L^2_x} \\
&\lesssim_{m,\hq} \sum_{k \geq \hq} 2^{-3k} \| \nb^3 \chi_R \|_{L^\infty}   2^{(1-\de)k}\| \th_\tau(t) \|_{\dot{H}^{-1+\de}}
}
Since $\nb^3 \chi_R \to 0$ uniformly, we have that $IV$ converges to $0$ in $L_x^2$ at any time $t$, with $L^2$ norms uniformly bounded by a square-integrable function of $t$.

This analysis together with the dominated convergence theorem justifies the claim \eqref{eq:diffFromTorusAgain}.  We now observe that, for each fixed value of $\hq, \tau, m$ and $R$, we have
\ali{
- \fr{1}{2} \int \eta'(t) \int &\left| P_{(-m, \hq]} (\La^{-1+\de} \chi_R \th(t,x))_\tau\right|^2 dx dt = - \int \th \pr_t \psi_{\hq,\tau,m,R} dx dt \\ 
\psi_{\hq,\tau,m,R} &= \chi_R \left( \eta \Pmhq^2 \La^{2(-1+\de)}(\chi_R \th_\tau) \right)_\tau. \label{eq:phewLimit}
}
using the self-adjointness of both the mollification operators involved and of multiplication by $\chi_R$.  Note that $\psi_{\hq,\tau,m,R} \in C_c^\infty(\R \times \R^2)$ is a smooth test function with compact support, which enables us to apply the definition of a weak solution to write the right hand side as
\ALI{
\eqref{eq:phewLimit} &= \int B(\psi_{\hq,\tau,m,R}(t), \th(t), \th(t)) dt.
}
We now aim to take the limits as $R \to \infty, m \to \infty$ and $\tau \to \infty$ in that order.  We first claim that 
\ali{
\psi_{\hq,\tau,m,R} \to \psi_{\hq,\tau,m} \coloneqq \left( \eta \Pmhq \La^{2(-1+\de)}(\th_\tau) \right)_\tau  \qquad \mbox{ in } L_t^3 \dot{W}^{2,\infty}, \label{eq:Rlimit} 
}
which is sufficient given that $\th(t) \in L_t^3 \dot{H}^{-1+\de}$ and the $\dot{W}^{2,\infty} \times \dot{H}^{-1+\de} \times \dot{H}^{-1+\de}$ boundedness of $B$.

To take this limit, first observe that the convergence
\ali{
\eta \Pmhq \La^{2(-1+\de)}(\chi_R \th_\tau) \to \eta \Pmhq \La^{2(-1+\de)}(\th_\tau)   \label{eq:multiplyThisOne}
}
holds in $L_t^3L_x^2$ by the same argument used to take the $R \to \infty$ limit in \eqref{eq:diffFromTorusAgain}.  The convergence also holds in the $L_t^3 C_x^2$ norm by the boundedness of the frequency localization operator $\Pmhq = P_{(-m-2, \hq+2]} \Pmhq$.  

Recalling that the $\dot{W}^{2,\infty}$ semi-norm is $\| \nb^2 f \|_{L^\infty}$, it is not hard to check that the second mollification in time preserves the $L_t^3 \dot{W}^{2,\infty}$ convergence, as does multiplying by $\chi_R$ and letting $R \to \infty$.  This observation suffices to take the $R \to \infty$ limit in \eqref{eq:Rlimit}.

We now study the $m \to \infty$ limit where we re-insert the low frequencies.  Observe that, by the Berenstein inequality,
\ALI{
\| \psi_{\hq,\tau,m}(t) - \psi_{\hq,\tau}(t) \|_{\dot{W}^{2,\infty}} &= \| \nb^2 P_{\leq-m} \La^{-1 + \de} \La^{-1 + \de} \th_\tau \|_{L^\infty} \\
&\lesssim 2^{-2m/2 } \| \nb^2 P_{\leq -m} \La^{-1 + \de} \La^{-1 + \de} \th_\tau \|_{L^2} \\
&\lesssim 2^{-m} \| \nb^2 P_{\leq -m} \La^{-1 + \de} \|_{L^2 \to L^2} \| \La^{-1 + \de} \th_\tau \|_{L^2} \\
&\lesssim 2^{-m} 2^{-(1+\de)m} \| \La^{-1+ \de} \th_\tau(t) \|_{L^2}
}
The bound for $\| \nb^2 P_{\leq -m} \La^{-1 + \de} \|_{L^2 \to L^2}$ used in the last line can be shown by a dyadic decomposition and summing a geometric series.  This estimate and an application of the dominated convergence theorem suffice for the $L_t^3 \dot{W}^{2,\infty}$ convergence.  (In the case $\de = 0$ of $2D$ Euler, we employ the bound $\| \La^{-1} \th_\tau(t) \|_{L^2} \lesssim  \| v^i \|_{L^2} \approx \| \th_\tau(t) \|_{\dot{H}^{-1}} $, where $v$ is the Helmholtz solution to $\nb_i v^i = \th_\tau$, and we use that $\La^{-1} \nb_i$ is $L^2$ bounded.  See Appendix Section~\ref{sec:hdots}.)

Next we claim that the convergence
\ALI{
\psi_{\hq,\tau} = \left( \eta P_{\leq \hq}^2 \La^{2(-1+\de)}(\th_\tau) \right)_\tau \to \psi_{\hq} = \eta P_{\leq \hq}^2 \La^{2(-1+\de)} \th \quad \mbox{in } L_t^3 \dot{W}^{2,\infty} \mbox{ as } \tau \to 0
}
for $\th \in L_t^3 \dot{H}^{-1+\de}$.  The proof follows a similar line of reasoning as in Section~\ref{sec:testFunctionStep}, but we now use the fact that $\nb^2 P_{\leq \hq}^2 \La^{-1 + \de}$ is bounded from $L_t^3L^2_x(I \times \R^2)$ to $L_t^3 L^\infty_x(I \times \R^2)$, which again can be proved by a dyadic decomposition.  

We now proceed with the proof of Hamiltonian conservation, which is primarily a high frequency analysis.

\section{Conservation of the Hamiltonian} \label{sec:hamiltonian}

We now return to the setting of $\T^d$ periodic solutions and explain the modifications needed in the nonperiodic setting in Section~\ref{sec:EuclideanAgain}.

By the convention laid out in Section~\ref{sec:basicfunctionSpaces}, $P_{\leq q}$ has a multiplier that equals $1$ for $|\xi| \leq 2^{q-1}$ and is supported in $|\xi| < 2^{q}$.  Consequently, $P_q \equiv P_{\leq q+1} - P_{\leq q}$ has a multiplier supported in $2^{q-1} \leq |\xi| \leq 2^{q+1}$.

Let $\th_N = P_{\leq N} \th$ and $\psi_{\hq} \equiv \De^{-1+\de} P_{\leq \hq}^2 \th$ (recall that $\De = \La^2$).  By the boundedness of $B$ on $\dot{W}^{2,\infty} \times \dot{H}^{-1+\de} \times \dot{H}^{-1+\de}$, and by dominated convergence, we have that
\ali{
\int B(\psi_{\hat{q}}(t), \th(t), \th(t)) dt = \lim_{N\to \infty} \int B(\psi_{\hat{q}}(t), \th_N(t), \th_N(t)) dt  \label{eq:freqTrunTh}
}
However, the right hand side (at least in the case of the torus) is a classical integral.  Specifically, we are looking at 
\ali{
B_{\hq N} = B(\psi_{\hq}; \th_N, \th_N) &\equiv \int \eta(t) \int T^\ell(\theta_N) \th_N \nb_\ell \psi_{\hq} dx dt \\
&= \int \eta(t) \int \nb_\ell (-\De)^{-1+\de} P_{\leq \hq}^2 \th P_{\leq N} \th \ep^{\ell a} \nb_a (-\De)^{-1+\de} P_{\leq N} \th dx dt \label{eq:integralForBpsihq}
}

We can assume that the mean value $\Pi_0 \th$ of $\th_N$ is zero, for in general we may write $\th = \Pi_0 \th + \Pi_{\neq 0} \th$, where $T^\ell(\Pi_0 \th) = 0$ (since the multiplier is odd), and 
\ali{
B(\psi_{\hq}; \th_N, \th_N) &=  \int \eta(t) \int T^\ell(\Pi_{\neq 0} \theta_N) \Pi_0 \th_N \nb_\ell \psi_{\hq} dx dt + \int \eta(t) \int T^\ell(\Pi_{\neq 0} \theta_N) \Pi_{\neq 0} \th_N \nb_\ell \psi_{\hq} dx dt. \label{eq:no0mode}
}
The first integral vanishes as $\Pi_0 \th_N$ is constant, $T^\ell(\Pi_{\neq 0} \th_N)$ is divergence free, and there is no boundary.

The part of \eqref{eq:integralForBpsihq} from $\ep^{\ell a} \nb_\ell (-\De)^{-1+\de} P_{\leq \hq}^2 \th \nb_a (-\De)^{-1+\de} P_{\leq \hq}^2 \th$ also vanishes, leaving us with
\ali{
B(\psi_{\hq}; \th_N, \th_N) &= \int \eta(t) \int \nb_\ell (-\De)^{-1 +\de} P_{\leq \hq}^2 \th P_{\leq N} \th T_{> \hq}^\ell(P_{\leq N} \th) dx dt \label{eq:lowFreqGone} \\
T_{> \hq}^\ell(f) &\equiv (1 - P_{\leq \hq}^2) T^\ell f = (1 - P_{\leq \hq}^2) \ep^{\ell a} \nb_a (-\De)^{-1+\de} f \notag
}
We remark here that there is an additional complication related to low frequencies in the case of $\R^2$.  We address this complication in Section~\ref{sec:EuclideanAgain}.

Since $T_{> \hq}$ restricts to frequencies above $2^{\hq-1}$, we have
\ALI{
B_{\hq N} = \sum_{q = \hq - 2}^N \de_q B_{\hq N}, \quad \de_q B_{\hq N} = \int \eta(t) \int \nb_\ell \psi_{\hq}(P_{\leq q+1} \th T_{> \hq}^\ell P_{\leq q+1} \th - P_{\leq q} \th T_{> \hq}^\ell P_{\leq q} \th ) dx dt.
}
Indeed, $P_{\leq q} \th T_{> \hq}^\ell P_{\leq q} \th = 0$ for $q < \hq - 2$.  Now decompose $\de_q B_{\hq N} = \de_{q H1} B_{\hq N} +\de_{q H2} B_{\hq N}  +\de_{q L} B_{\hq N}$,
\ali{
\de_{q H1} B_{\hq N} &= \int \eta(t) \int \nb_\ell \psi_{\hq} \left( P_{q+1} \th T_{> \hq}^\ell P_{q+1} \th \right) dx dt \label{eq:deqH1}\\
\de_{q H2} B_{\hq N} &= \int \eta(t) \int \nb_\ell \psi_{\hq} \left( P_{q+1} \th T_{> \hq}^\ell P_{q} \th + P_q \th T_{> \hq}^\ell P_{q+1} \th \right) dx dt \label{eq:deqH2} \\
\de_{qL} B_{\hq N} &= \int \eta(t) \int \nb_\ell \psi_{\hq} \underbrace{\left(P_{q+1} \th T_{> \hq}^\ell P_{\leq q-1} \th + P_{\leq q-1} \th T_{> \hq}^\ell P_{q+1}\th \right) }_{Q_{qL}(\th)} dx dt \label{eq:deqL}
}
Observe that, since $\Fsupp P_{q+1} \th \subseteq \{ 2^{q} \leq |\xi| \leq 2^{q+2} \}, \Fsupp P_{\leq q-1} \th \subseteq \{ |\xi| < 2^{q-1} \}$, we have 
\ali{
\Fsupp\, Q_{qL}(\th) &\subseteq \{ 2^{q -1} \leq |\xi| \leq 2^{q+3} \} \label{eq:FsuppQqL}
}
In particular, $\de_{q L} B_{\hq N} = 0$ for $q > \hq + 3$ since $\Fsupp \psi_{\hq} \subseteq \{ |\xi| < 2^{\hq+1}\} $.

To estimate $\de_{qL} B_{\hq N}$, write $Q_{qL}(\th) = \nb_j \nb^j \De^{-1} Q_{qL}(\th)$ and integrate by parts
\ali{
\de_{qL} B_{\hq N} &= - \int \eta(t) \int \nb_j \nb _\ell \psi_{\hq} \De^{-1} \nb^j Q_{qL}(\th) dx dt. \label{eq:IBPLterm}
}
Let $t$ be a time for which $\La^{-1 +\de} \th$ belongs to $B_{3,c(\N)}^\a$.  Then
\ALI{
\| \nb_j \nb_\ell \psi_{\hq}(t) \|_{L^3_x} &= \|\nb_j \nb_\ell (-\De)^{-1+\de} P_{\leq \hq}^2 \th(t) \|_{L_x^3} \lesssim \sum_{q = -\infty}^{\hq + 1} \| \nb_j \nb_\ell \La^{-1+\de} P_q P_{\leq \hq}^2 \La^{-1+\de} \th \|_{L_x^3} \\
&\lesssim \sum_{q = -\infty}^{\hq + 1} \| \nb_j \nb_\ell \La^{-1+\de} P_{\approx q} \| \cdot \| P_q \La^{-1+\de} \th \|_{L_x^3}
}
Here we recall that we write $\| C \ast \| \equiv \| C \|_{L^1(\R^d)}$ for a convolution operator $C\ast$, and that $P_{\approx q}$ has a Fourier multiplier that is a smooth bump function adapted to $|\xi| \sim 2^q$.  Thus
\ali{
\| \nb_j \nb_\ell \psi_{\hq}(t) \|_{L^3_x} &\lesssim \sum_{q = -\infty}^{\hq + 1} 2^{(2 - 1 +\de) q} 2^{-\a q} \| \La^{-1+\de}\th(t) \|_{\dot{B}_{3,\infty}^\a} \lesssim 2^{(1 +\de - \a) \hq} \| \La^{-1+\de} \th(t) \|_{\dot{B}_{3,\infty}^\a}. \label{eq:boundForNb2psi}
}
since $1 +\de - \a > 0$ for $ \a \leq (1+\de)/3$.  Furthermore, for any $t$ such that $\La^{-1+\de} \th \in B_{3, c(\N)}^\a$, we have
\ali{
\| \nb_j \nb_\ell \psi_{\hq}(t) \|_{L^3_x} &= \| \nb_j \nb_\ell (-\De)^{-1+\de} P_{\leq \hq}^2 \th (t)\|_{L_x^3} = o(2^{(1 +\de - \a) \hq}) \qquad \mbox{ as } \hq \to \infty \label{eq:littleOhBound}
}
by a simple approximation argument.  The key points are the uniform in $\hq$ boundedness of the operators $2^{-(1 +\de - \a) \hq} \nb_j \nb_\ell \La^{-1 +\de}$ on $B_{3,\infty}^\a$, and the fact that $\La^{-1 +\de} \th(t)$ is the strong limit in $B_{3,\infty}^\a$ of $\La^{-1 +\de} f_{(k)}$, where the $f_{(k)}$ are smooth and satisfy an $O(1)$ estimate as $\hq \to \infty$.

 Meanwhile, since we can take $\th$ to have zero mean, we have $P_{\leq q} \th = \sum_{r < q} \La^{1-\de} P_r \La^{-1 +\de} \th$.  We obtain using $1 - \de - \a \geq 2/3 - 4\de/3 > 0$ for $\de < 1/2$
\ALI{
\| P_{\leq q} \th(t) \|_{L^3} &\lesssim \sum_{r < q} \| \La^{1-\de} P_{\approx r} \| 2^{-\a r} \| \La^{-1 +\de} \th(t) \|_{\dot{B}_{3,\infty}^\a} \lesssim 2^{(1- \de - \a) q} \| \La^{-1 +\de} \th(t) \|_{\dot{B}_{3,\infty}^\a}.
}
In the case of SQG where $\de = 1/2$, $\a = (1+\de)/3 = 1/2$ at the critical exponent, we have $1 - \de - \a = 0$ and the above estimate breaks down.  Instead, in this case, we require $\th(t) \in L_x^3$ to bound $P_{\leq \hq} \th(t)$ uniformly in $L^3$ by $\lesssim \| \th(t) \|_{L^3_x}$.

We now use the Fourier support of $Q_{qL}$ in \eqref{eq:FsuppQqL}, the fact that $\de_{q L} B_{\hq N} = 0$ unless $\hq - 2 \leq q \leq \hq + 3$, and H\"{o}lder to obtain
\ali{
\| (-\De)^{-1} \nb^j Q_{qL}(\th)  \|_{L^{3/2}_x} &\lesssim 2^{-q} \| Q_{qL} \|_{L^{3/2}_x} \lesssim 2^{-q} \| P_q \th \|_{L_x^3} \| P_{\approx q} T_{> \hq}\| \| P_{\leq q - 1} \th(t) \|_{L_x^3} \notag \\
&\lesssim 2^{-q} \|P_{\approx q} \La^{1 - \de} \|\, \| P_q \La^{-1+\de} \th(t) \|\, \|P_{\approx q} \nb (-\De)^{-1+\de} \|\, \| P_{\leq q - 1} \th(t) \|_{L_x^3} \notag \\
&\lesssim 2^{(-1 + (1 - \de) - \a + (2 \de - 1) + ( 1 - \de - \a)) q } \| \La^{-1 +\de} \th(t) \|_{\dot{B}_{3,\infty}^\a}^2 \notag\\
\| (-\De)^{-1} \nb^j Q_{qL}(\th)  \|_{L^{3/2}_x} &\lesssim 2^{- 2 \a q }  \| \La^{-1 +\de} \th(t) \|_{\dot{B}_{3,\infty}^\a}^2. \label{eq:Qqlbd}
}
Substituting into \eqref{eq:IBPLterm} leaves us with 
\ali{
\sup_N \sum_{q = \hq - 2}^N |\de_{q L} B_{\hq N}| &\lesssim \int \eta(t) \| \nb_j \nb_\ell \psi_{\hq}(t) \|_{L_x^3} 2^{- 2 \a \hq }  \| \La^{-1 +\de} \th(t) \|_{\dot{B}_{3,\infty}^\a}^2 dt. \label{eq:deqLBqNbd}
}
(In the SQG case, we simply bound $\| P_q \th(t) \|_{L_x^3}, \| P_{\leq q - 1} \th(t) \|_{L_x^3} \lesssim \|\th(t)\|_{L_x^3}$ in deriving \eqref{eq:Qqlbd} to obtain $\| \th(t) \|_{L_x^3}^2$ in place of $\| \La^{-1 +\de} \th(t) \|_{\dot{B}_{3,\infty}^\a}^2$ on the right hand side of \eqref{eq:deqLBqNbd}.)


We now proceed to establish the same bound for \eqref{eq:deqH1} and \eqref{eq:deqH2}.  In fact we will only consider \eqref{eq:deqH2} since \eqref{eq:deqH1} can be treated identically.  Consider the bilinear form appearing in \eqref{eq:deqH2}
\ALI{
Q_q^\ell[\th,\th] = P_{q+1} \th T_{> \hq}^\ell P_{q} \th + P_q \th T_{> \hq}^\ell P_{q+1} \th = Q_q^\ell[P_{\approx q} \th, P_{\approx q} \th]
}
We suppress the dependence of $Q_q$ on $\hq$ to ease notation.  In Section~\ref{sec:deriveBilinear} below we derive the following divergence form for $Q_q$
\ali{
Q_q^\ell[\th,\th] &= \nb_j B_q^{j\ell}[P_{\approx q} \th, P_{\approx q} \th] \label{eq:divFormQq} \\
B_q^{j\ell}[f, f]&= \int_{\R^2 \times \R^2}  f(t, x - h_1) f(t, x - h_2) \widetilde{K}_q^{j\ell}(h_1,h_2) dh_1 dh_2 \notag \\
\| \widetilde{K}_q \|_{L^1(\R^2 \times \R^2)} &\lesssim 2^{2(-1+\de)q}, \quad \| B_q^{j\ell}[f, f] \|_{L_x^{3/2}} \leq \| \widetilde{K}_q \|_{L^1(\R^2 \times \R^2)} \| f \|_{L_x^3}^2, \label{eq:KqBq:bd}
}
where again we suppress the dependence of $B_q$ and $K_q$ on $\hq$.  The last estimate follows by H\"{o}lder and Fubini using the duality between $L^{3/2}$ and $L^3$ together with the translation invariance of the $L^3$ norm.  

Substituting \eqref{eq:divFormQq} into \eqref{eq:deqH2}, integrating by parts and using \eqref{eq:KqBq:bd}, we have
\ali{
\de_{q H2} B_{\hq N} &= - \int \eta(t) \int \nb_j \nb_\ell \psi_{\hq} B_q^{j\ell}[P_{\approx q} \th, P_{\approx q} \th] dx dt \notag \\
|\de_{q H2} B_{\hq N}| &\leq \int \eta(t) \, \|\nb_j \nb_\ell \psi_{\hq}(t)\|_{L_x^3} \|B_q^{j\ell}[P_{\approx q} \th, P_{\approx q} \th] \|_{L_x^{3/2}} \notag \\
&\lesssim \int \eta(t) \, \| \nb_j \nb_\ell \psi_{\hq}(t) \|_{L_x^3} 2^{2(-1+\de)q} \| P_{\approx q} \th(t) \|_{L_x^3}^2 dt \label{eq:L3BoundAtq} \\
&\lesssim \int \eta(t) \,\| \nb_j \nb_\ell \psi_{\hq}(t) \|_{L_x^3} \| P_{\approx q} \La^{-1 +\de} \th(t) \|_{L_x^3}^2 dt \notag \\
&\lesssim \int \eta(t) \, \| \nb_j \nb_\ell \psi_{\hq}(t) \|_{L_x^3} 2^{-2 \a q} \| \La^{-1 +\de} \th(t) \|_{\dot{B}_{3,\infty}^\a}^2 dt \notag \\
\sup_N \sum_{q = \hq - 2}^N |\de_{q H2} B_{\hq N}| &\lesssim \int \eta(t) \| \nb_j \nb_\ell \psi_{\hq}(t) \|_{L_x^3} 2^{-2\a \hq} \| \La^{-1 +\de} \th(t) \|_{\dot{B}_{3,\infty}^\a}^2 dt \label{eq:boundFordeqTerms}
}
In the SQG case ($\de = \a = 1/2$), we replace $\| \La^{-1 +\de} \th(t) \|_{\dot{B}_{3,\infty}^\a}$ with $\|\th(t) \|_{L_x^3}$, in which case the bound follows from \eqref{eq:L3BoundAtq}.

We similarly obtain $\sup_N \sum_{q = \hq - 2}^N |\de_{q H1} B_{\hq N}|$ bounded by the right hand side of \eqref{eq:boundFordeqTerms}, which, in view of \eqref{eq:deqLBqNbd}, means that our bound for \eqref{eq:integralForBpsihq} is
\ALI{
|B(\psi_{\hq}; \th, \th)| = \lim_N |B(\psi_{\hq}; \th_N, \th_N)| \leq  \int \eta(t) \| \nb_j \nb_\ell \psi_{\hq}(t) \|_{L_x^3} 2^{-2\a \hq} \| \La^{-1 +\de} \th(t) \|_{\dot{B}_{3,\infty}^\a}^2 dt.
}
Using \eqref{eq:boundForNb2psi}, the $dt$ integrand on the right hand side is bounded by $\eta(t) \| \La^{-1+\de} \th(t) \|_{\dot{B}_{3,\infty}^\a}^3$ (or bounded by $\eta(t) \| \th(t)\|_{L_x^3}^3$ in the SQG case), which is $L^1$ in time since $\La^{-1+\de} \th(t) \in L_t^3 \dot{B}_{3,c(\N)}^\a$.  Moreover, the little-Oh bound \eqref{eq:littleOhBound} shows that the $dt$ integrand goes to $0$ for almost every $t$ as $\hq \to \infty$.  By the dominated convergence theorem, we conclude that $|B(\psi_{\hq}; \th, \th)| \to 0$ as $\hq \to \infty$.

\subsection{Hamiltonian conservation in the Euclidean case} \label{sec:EuclideanAgain}

In the Euclidean case we replace $\th_N = P_{\leq N} \th$ by $\th_N = P_{[-N,N]} \th \equiv (P_{\leq N} - P_{\leq -N})\th$.  Then equation \eqref{eq:freqTrunTh} still holds, but it is not quite a classical integral.  Rather, for $P_{[-N,\hq]} \th \equiv (P_{\leq \hq}^2 - P_{\leq -N})\th$ we have
\ali{
B_{\hq N} = B(\psi_{\hq}; \th_N, \th_N) &\equiv \int \eta(t) \int T^\ell(P_{[-N, N]}\theta) P_{[-N,N]} \th \nb_\ell \De^{-1+\de} P_{[-N, \hq]} \th dx dt \label{eq:mainTermEuc} \\
&+ \int \eta(t) B(P_{\leq -N} \De^{-1+\de} \th(t), P_{[-N,N]} \th(t)) dt \label{eq:errEuc}
}
The main point is that the term \eqref{eq:errEuc} is an error term, for $\th \in L_t^3 \dot{H}^{-1+\de}$.  In fact,
\ALI{
|B(P_{\leq -N} \De^{-1+\de} \th, P_{[-N,N]} \th)| &\lesssim \| \nb^2 P_{\leq -N} \De^{-1+\de} \th \|_{L^\infty} \| P_{[-N,N]} \th \|_{\dot{H}^{-1+\de}}^2 \\
&\lesssim \| P_{\leq -N} \nb \La^\de \|_{L^2\to L^\infty} \| \nb \La^{-2+\de} \th \|_{L^2} \| \th \|_{\dot{H}^{-1+\de}}^2 \\
&\lesssim o(1) \cdot \| \th(t) \|_{\dot{H}^{-1+\de}}^3
}
The remainder of the proof in the Euclidean case, which involves estimating \eqref{eq:mainTermEuc}, follows along the lines of the torus case but with $-N$ as the lower bound for the frequency sums.  
Specifically, after observing the key cancellation, the quantity that must be estimated is
\ALI{
\int \eta(t) \int \nb_\ell (-\De)^{-1 +\de} P_{[-N,\hq]} \th P_{[-N,N]} \th T_{> \hq}^\ell(P_{[-N,N]} \th) dx dt, 
}
and one takes $P_{[-N, q]} \th \equiv (P_{\leq q} - P_{\leq -N})\th$ in place of $P_{\leq q} \th$ during the analysis.

\section{Deriving the bilinear form} \label{sec:deriveBilinear}
In this section, we include the derivation of the bilinear form \eqref{eq:divFormQq} and the bound \eqref{eq:KqBq:bd}.  We follow \cite{isett2021direct}, which in turn generalizes a computation of \cite{buckShkVicSQG}.  We only consider the case where the domain is the torus since the nonperiodic case is simpler.

Consider the bilinear form appearing in \eqref{eq:deqH2}
\ali{
Q_q^\ell[\th,\th] &= P_{q+1} \th T_{> \hq}^\ell P_{q} \th + P_q \th T_{> \hq}^\ell P_{q+1} \th = Q_q^\ell[P_{\approx q} \th, P_{\approx q} \th]  \label{eq:QqellReminder} \\
\widehat{T_{> \hq}^\ell f}(\xi) &= m_{\hq}^\ell(\xi) \hat{f}(\xi) =  ( 1 - \hat{\eta}_{\leq \hq}^2(\xi) )\ep^{\ell a} (i \xi_a) |\xi|^{2(-1+\de)} \hat{f}(\xi). \notag
}
The most crucial point is the anti-symmetry of the operator $T$, reflected in the fact that $m_{\hq}^\ell$ is odd.  Without this feature, \eqref{eq:QqellReminder} may not have integral $0$ and an anti-divergence would not exist.  The idea to obtain the divergence form will be to Taylor expand in frequency space.  Another example of Taylor expansion in frequency space to derive an anti-divergence operator can be found in \cite{IOnonpd}.

We start by approximating the periodic function $\th$ by Schwartz functions.  By mollifying in frequency space, we may assume that the Schwartz approximations to $P_q \th$ and $P_{q+1} \th$ have frequency support in $\{ 2^{q - 2} \leq |\xi| < 2^{q+2} \}$, and that they converge pointwise in physical space while also remaining uniformly bounded there.  By abuse of notation, assume now that $P_q \th$ and $P_{q+1} \th$ are the Schwartz approximations.  For these approximations we construct a bilinear anti-divergence as follows. 

\begin{align}
	\hat{Q}_q^\ell(\xi) 
	&= 
	\int_{\hat{\R}^2} 
	m_{\hq}^\ell(\xi - \eta) \widehat{P_{q+1}\th}(\xi - \eta) \widehat{P_{q}\th}(\eta) + 
	\widehat{P_{q+1}\th}(\xi - \eta) m_{\hq}^\ell(\eta)\widehat{P_{q}\th}(\eta) \fr{d\eta}{(2\pi)^2} \nonumber \\
	&= 
	\int_{\hat{\R}^2} 
	m_{\hq}^\ell(\xi - \eta)\hat{\chi}_{q}\widehat{P_{q+1}\th}(\xi - \eta)\widehat{P_{q}\th}(\eta) + 
	\widehat{P_{q+1}\th}(\xi - \eta) m_{\hq}^\ell(\eta)\hat{\chi}_{q}\widehat{P_{q}\th}(\eta) \fr{d\eta}{(2\pi)^2} \nonumber \\
	&\text{Define } m_{q}(\cdot) \coloneqq m_{\hq}(\cdot)\hat{\chi}_{q}(\cdot) \nonumber \\
	&=
	\int_{\hat{\R}^2} 
	\left [m^{\ell}_{q}(\xi - \eta) + m^{\ell}_{q}(\eta) \right ] 
	\widehat{P_{q+1}\th}(\xi - \eta) 
	\widehat{P_{q}\th}(\eta) \fr{d\eta}{(2\pi)^2} \nonumber \\
	&\text{By the oddness of the multiplier } m^{\ell}_{q}\nonumber \\
	&=
	\int_{\hat{\R}^2} 
	\left [m^{\ell}_{q}(\xi - \eta) - m^{\ell}_{q}(-\eta) \right ] 
	\widehat{P_{q+1}\th}(\xi - \eta) 
	\widehat{P_{q}\th}(\eta) \fr{d\eta}{(2\pi)^2} \nonumber \\
	&=
	i\xi_j
	\int_{\hat{\R}^2} \int_0^1 
	-i \nb^j m^{\ell}_{q}(u_\si) d\si  
	\widehat{P_{q+1}\th}(\zeta) 
	\widehat{P_{q}\th}(\eta) \fr{d\eta}{(2\pi)^2} \nonumber \\
\hat{Q}_q^\ell(\xi)	&\eqqcolon
	i\xi_j
	\int_{\hat{\R}^2}
	\hat{K}^{j\ell}_q (\zeta, \eta) 
	\widehat{P_{q+1}\th}(\zeta) 
	\widehat{P_{q}\th}(\eta) \fr{d\eta}{(2\pi)^2}, \label{eq: hat Q formula}
\end{align}
where we have set
\[
	\zeta \coloneqq \xi - \eta, \quad u_\si \coloneqq \si (\xi - \eta) - (1 - \si)\eta = \si \zeta - (1-\si) \eta, 
\]
and we have used Taylor's Remainder Theorem in the last equality. In the last line we defined 
\begin{equation}\label{def: hat K }
	\hat{K}^{j\ell}_q (\zeta, \eta) \coloneqq 
	\int_0^1 
	-i \nb^j m^{\ell}_{q}(u_\si)
	\hat{\chi}_{q}(\zeta)
	\hat{\chi}_{q}(\eta)
	d\si.
\end{equation}
Here we suppress the dependence of $K_q^{j\ell}$ on $\hq$.  For $q > \hq + 1$, the cutoff $(1 - \hat{\eta}_{\leq \hq}^2)$ in $m_q$ is $1$ on the support of $\hat{\chi}_q$, and the homogeneity of $m^\ell$ leads to the scaling $\widehat{K}_q^{j\ell}(\zeta, \eta) = 2^{2(-1+\de)q} f^{j\ell}(2^{-q} \zeta, 2^{-q} \eta)$, where $f^{j\ell}$ is Schwartz, which in turn implies the bound
\ali{
\| K_q \|_{L^1(\R^2 \times \R^2)} &\lesssim 2^{2(-1 + \de)q} \label{eq:boundForKq}
}
for $q > \hq + 1$.  The same bound also holds for $\hq - 1 \leq q \leq \hq + 1$ by scaling; for example $K_{\hq - 1}^{j\ell}(\zeta, \eta) = 2^{2(-1+\de)\hq} f_0^{j\ell}(2^{-\hq} \zeta, 2^{-\hq} \eta)$, for some Schwartz function $f_0$.  Finally, $K_q = 0$ for $q < \hq - 1$.

Take the inverse Fourier Transform of \eqref{eq: hat Q formula} to obtain 
\ALI{
	Q_q^\ell(x) &= \nb_j \int_{\hat{\R}^2}e^{i \xi \cdot x} \int_{\hat{\R}^2}
	\hat{K}^{j\ell}_{q} (\zeta, \eta)
	\widehat{P_{q+1}\th}(\zeta) 
	\widehat{P_{q}\th}(\eta)
	\fr{d\eta}{(2\pi)^2} \fr{d\xi}{(2\pi)^2} \\
	&=
	 \nb_j \int_{\hat{\R}^2\times \hat{\R}^2} e^{i (\zeta + \eta) \cdot x}
	 \hat{K}^{j\ell}_{q} (\zeta, \eta)
	\widehat{P_{q+1}\th}(\zeta) 
	\widehat{P_{q}\th}(\eta)
	\fr{d\eta}{(2\pi)^2} \fr{d\zeta}{(2\pi)^2}.
}
Observe that $\hat{K}_{q}^{j\ell}$ is 
compactly supported in frequency and smooth due to the factors of $\hat{\chi}_{\approx \la}$, making $\hat{K}_{q}^{j\ell}$ a Schwartz function. Thus we can define the corresponding physical space kernel  $K^{j\ell}_{q} (h_1, h_2)$ as the inverse Fourier Transform of $\hat{K}^{j\ell}_{q} (\zeta, \eta)$ so that 
\[
	\hat{K}^{j\ell}_{q} (\zeta, \eta) = \int_{\hat{\R}^2\times \hat{\R}^2} e^{-i(\zeta, \eta) \cdot (h_1, h_2)}K^{j\ell}_{q} (h_1, h_2) dh_1 dh_2.
\]
Then 
\ALI{
	Q_q^\ell(x) &=
	\nb_j \int_{\R^2 \times \R^2} 
	e^{i (\zeta + \eta) \cdot x}	
	\int_{\hat{\R}^2\times \hat{\R}^2}
	e^{-i(\zeta, \eta) \cdot (h_1, h_2)}
	K^{j\ell}_{q}(h_1, h_2)
	dh_1 dh_2 
	\widehat{P_{q+1}\th}(\zeta) 
	\widehat{P_{q}\th}(\eta)
	\fr{d\eta}{(2\pi)^2} \fr{d\zeta}{(2\pi)^2} \\
	&=
	\nb_j \int_{\R^2 \times \R^2} 
	K^{j\ell}_{q}(h_1, h_2)
	\int_{\hat{\R}^2}
	e^{i\zeta \cdot (x - h_1)}
	\widehat{P_{q+1}\th}(\zeta) 
	\fr{d\zeta}{(2\pi)^2}
	\int_{\hat{\R}^2}
	e^{i\eta \cdot (x - h_2)}
	\widehat{P_{q}\th}(\eta)
	\fr{d\eta}{(2\pi)^2}
	dh_1 dh_2 \\
	&=
	\nb_j \int_{\R^2 \times \R^2}K^{j\ell}_{q}(h_1, h_2)P_{q+1}\th(x - h_1)P_{q}\th (x - h_2) dh_1 dh_2 \\
	&\eqqcolon \nb_j B^{j\ell}_{q}[\th, \th](x)
}
Here we define the bilinear anti-divergence operator as
\begin{align} 
	B^{j\ell}_{q}[F_1, F_2](x) &\coloneqq  \int_{\R^2 \times \R^2}K^{j\ell}_{q}(h_1, h_2)P_{q+1}F_1(x - h_1 )P_qF_2(x - h_2) dh_1 dh_2, \notag \\
	B^{j\ell}_{q}[F_1, F_2](x)	&= \int_{\R^2 \times \R^2}\widetilde{K}^{j\ell}_{q}(h_1, h_2) F_1(x - h_1) F_2(x-h_2)  dh_1 dh_2 \label{def: Bilinear form}
\end{align}
for scalar fields $F_1, F_2 \in C^\infty(\T^2)$.  Here we obtain $\widetilde{K}_q = (P_{q+1} \otimes P_q) K_q$ from $K_q$ by writing $P_{q+1} F_1$ (and $P_q F_2$) as a convolution in the variable $h_1$ ($h_2$ in the latter case) and we pass $P_q$ and $P_{q+1}$ onto the kernel $K_q$ by a change of variables.  This bi-convolution operator satisfies the equality
\ali{
T^\ell[P_{q+1}\th] \nb_\ell P_{q}\th + T^\ell[P_{q}\th]\nb_\ell P_{q+1}\th =
	\nb_j B^{j\ell}_{q}[P_{\approx q} \th, P_{\approx q}\th](x) \label{eq:bilin:antidiv}
}
for the 
Schwartz approximations to $P_{q+1}\th, P_{q}\th$ whose Fourier supports lie in $\{ |\xi| \sim 2^q \}$. Since $\|\widetilde{K}_{q} \|_{L^1(\R^2)}$ is finite and the Schwartz approximations to $P_{\approx q} \th$ are bounded and converge pointwise, we may apply dominated convergence to pass to the limit in \eqref{def: Bilinear form} with $F_1 = F_2 \to P_{\approx q} \th$ to conclude that \eqref{eq:bilin:antidiv} holds (both in $\DD'$ and classically) for the periodic functions $P_{q+1}\th, P_{q}\th$.  Thus we have shown \eqref{eq:divFormQq}.  
Finally, note that the bound \eqref{eq:boundForKq} for $K_q$ holds as well for $\widetilde{K}_q$, which establishes \eqref{eq:KqBq:bd}.

$ $

\noindent {\bf Remark:}  For future reference, we note here the important property that $K_q$ is trace-free in the special case of the mSQG multiplier.  To check that $K_q^{j\ell}$ is trace-free, we use the fact that the frequency cutoff $\chi_q(\xi) = \tilde{\chi}_q(|\xi|^2)$ is radial and it suffices check that $\nb^j m_q^\ell(\xi)$ is trace-free:
\ali{
m_q^\ell(\xi) &= \ep^{\ell a} i \xi_a |\xi|^{2(-1 + \de)} \chi_q(\xi) = \ep^{\ell a} i \xi_a |\xi|^{2(-1 + \de)} \tilde{\chi}_q(|\xi|^2) \notag \\ 
\nb^j m_q^\ell(\xi) &= i \ep^{\ell j} |\xi|^{2(-1 + \de)} \chi_q(\xi) + \ep^{\ell a} i \xi_a \xi^j F_{q,\de}(|\xi|^2). \label{eq:itsTraceFree}
}
 Here $F_{q,\de}$ is some radial smooth function adapted to $|\xi|^2 \sim (2^q)^2$.  As the above expression is trace-free ($\ep^{\ell j}$ is antisymmetric and $\ep^{\ell a} \xi_a$ is perpendicular to $\xi$), this calculation confirms that $K_q^{j\ell}$ is trace-free.

\section{Defining the nonlinearity} \label{sec:definingNonlinearity}

For scalar functions $\psi$, $\th \in C^\infty(\T^2)$, we define the form
\ali{
B[\psi, \th] &= -\int \psi(x) T^\ell\th(x) \nb_\ell \th(x) dx = \int \nb_\ell \psi \th T^\ell \th dx
}
The main result in this section is the estimate
\ali{
|B[\psi,\th]| &\lesssim_\de \| \nb^2 \psi \|_{C^0} \| \th \|_{\dot{H}^{-1 + \de}}^2, \qquad  0 < \de < \infty, \label{eq:bilinBd}
}
where the constant depends on $\de > 0$.  The same bound also holds for the case of Euler ($\de = 0$), in which case the nonlinearity is given by $B[\psi, \th] = -\int \psi \cdot \ep_{a_\ell} \nb^a  \nb_j(  T^j \th T^\ell \th ) dx $, $\ep$ the Levi-Civita symbol and $T^\ell \th = -\ep^{\ell b} \nb_b \De^{-1} \th$.

The associated trilinear form $T[\psi,\th,\phi] \coloneqq \int \nb_\ell \psi (\th T^\ell \phi + \phi T^\ell \th) dx$, for which $B[\psi, \th] = \fr{1}{2}T[\psi, \th,\th]$ and $T[\psi, \th, \phi] = \fr{1}{4}(B[\psi, \th + \phi] - B[\psi, \th - \phi])$, will consequently satisfy a bound
\ali{
|T[\psi, \th, \phi]| \lesssim \| \nb^2 \psi \|_{C^0} \| \th \|_{\dot{H}^{-1 + \de}} \| \phi \|_{\dot{H}^{-1 + \de}}. \label{eq:trilinBd}
}
This bound on $T$ guarantees a unique continuous extension to $\th \in L_t^2 \dot{H}^{-1 + \de}$ for the nonlinearity defining the SQG equation by a Cauchy sequence argument.  We now begin the proof of \eqref{eq:bilinBd}.

Note that when $\psi$ is a constant, $B[\psi,\th] = 0$.  Thus we obtain the following decomposition of paraproduct type.
\ali{
B[\psi, \th] &= \sum_k \int P_k \nb_\ell \psi \th \, T^\ell \th dx = LL + HH + HL + LH \label{eq:BpsithExpand}\\
LL &= \sum_k \int P_k \nb_\ell \psi P_{\leq k - 3} \th \, T^\ell P_{\leq k - 3} \th dx = 0 \label{eq:LL} \\
HL &= \sum_k \sum_{q \geq k - 3 } \int P_k \nb_\ell \psi P_{q+1} \th \, T^\ell P_{\leq q-1} \th dx  \\
LH &= \sum_k \sum_{q \geq k - 3 } \int P_k \nb_\ell \psi P_{\leq q-1} \th \, T^\ell P_{q+1} \th dx  \\
HH &= \sum_k \sum_{q \geq k - 3 } \int P_k \nb_\ell \psi \underbrace{\left( P_{q} \th \, T^\ell P_{q+1} \th + P_{q+1}  \th \, T^\ell P_{q} \th  + P_{q+1} \th \, T^\ell P_{q+1} \th \right)}_{Q_q^\ell(\th)}  dx  \label{eq:HH}
}
That $LL = 0$ follows from Plancharel and the bound on the frequency support of $P_{\leq k - 3} \th T^\ell P_{\leq k - 3}$.  

We now estimate $HH$.   From Section~\ref{sec:deriveBilinear}, we have a divergence form for $Q_q^\ell(\th) = \nb_j B_q^{j\ell}[P_{\approx q} \th]$, where the multiplier for $P_{\approx q}$ is a bump function adapted to $\{ \xi \in \widehat{\R}^2 ~:~ 2^{q - 4} \leq |\xi| \leq 2^{q+4} \}$,
\ali{
B_q^{j\ell}[P_{\approx q} \th](x) &= \int_{\R^2 \times \R^2} K_q^{j\ell}(h_1,h_2) P_{\approx q} \th(x - h_1) P_{\approx q} \th(x - h_2) dh_1 dh_2,  \label{eq:BqisKqof}
}
and $\| K_q \|_{L^1(\R^2 \times \R^2)} \lesssim 2^{2(-1+\de)q}$.  (The $K_q$ is the sum of the kernels coming from the interaction of $P_{q+1} \th$ with itself, and the other coming from the symmetric interaction of $P_{q+1} \th$ with $P_q \th$.)  Plugging in the anti-divergence, interchanging the order of summation (which is easily justified for smooth $\psi$ and $\th$), and integrating by parts yields
\ali{
-HH &= \sum_q \int \left(\sum_{k \leq q - 3} P_k \nb_j \nb_\ell \psi \right) B_q^{j\ell}[P_{\approx q} \th] dx \label{eq:divFormulaHH} \\ 
|HH| &\leq \sum_q \| P_{\leq q-3} \nb^2 \psi \|_{C^0} \| B_q[P_{\approx q} \th] \|_{L^1} dx \notag \\
|HH| &\leq \sup_\ell \| P_{\leq \ell} \nb^2 \psi \|_{C^0} \sum_q \| K_q \|_{L^1(\R^2 \times \R^2)} \| P_{\approx q} \th \|_{L^2(\T^2)}^2 \notag \\
&\lesssim  \| \nb^2 \psi \|_{C^0} \sum_q \left( 2^{(-1 + \de)q} \| P_{\approx q} \th \|_{L^2(\T^2)} \right)^2 \notag \\
|HH| &\lesssim \| \nb^2 \psi \|_{C^0} \| \th \|_{\dot{H}^{-1 + \de}}^2 \label{eq:IHHbd}
}
Now we consider $HL$ and $LH$.  Notice that $\Fsupp (P_{q+1} \th \, T^\ell P_{\leq q - 1} \th) \subseteq \{ |\xi| \geq 2^{q-1} \}$ and similarly for the quadratic term in $LH$.  Thus the sum over $q$ extends only from $k - 3$ to $k + 2$.  To simplify the presentation, we treat the sum as though there is only one term with $q = k$.
\ali{
HL &\sim \sum_k \int P_k \nb_\ell \psi P_{k+1} \th \, T^\ell P_{\leq k-1} \th dx, \qquad 
LH \sim \sum_k \int P_k \nb_\ell \psi P_{\leq k-1} \th \, T^\ell P_{k+1} \th dx \notag
} 

We first treat $HL$.  Recall from the remarks surrounding \eqref{eq:no0mode} that we may assume that the zero Fourier mode of $\th$ vanishes.  Thus we decompose
\ali{
HL &\sim \sum_k \sum_{j \leq k-1} \int P_k \nb_\ell \psi P_{k+1} \th T^\ell P_j \th dx \label{eq:decomposeHL}
}
Using $\| T^\ell P_j \th \|_{L^2} = \| \nb^\ell (-\De)^{-1+\de} P_j \th \|_{L^2} \lesssim 2^{(-1+2\de)j} \| P_j \th \|_{L^2}$, we obtain
\ali{
|HL| &\lesssim \sum_k\sum_{j \leq k-1} \| P_k \nb \psi \|_{L^\infty} \| P_{k+1} \th \|_{L^2} 2^{\de j} 2^{(-1+\de)j} \| P_j \th \|_{L^2} \notag \\
&\lesssim \sum_j \sum_{k \geq j + 1} \| P_k \nb \psi \|_{L^\infty} 2^{(1-\de)k} \| \La^{-1 + \de} P_{k+1} \th \|_{L^2} 2^{\de j} \| \La^{-1 + \de} P_j \th \|_{L^2} \notag \\
&\lesssim \sum_j \sum_k \| P_k \nb^2 \psi \|_{L^\infty} 2^{-\de|j - k|} \| \La^{-1 + \de} P_{k+1} \th \|_{L^2}  \| \La^{-1 + \de} P_j \th \|_{L^2} \notag \\
&\lesssim \sup_\ell \| P_\ell \nb^2 \psi \|_{L^\infty} \sum_{j,k} 2^{-\de|j - k|} \| \La^{-1 + \de} P_{k+1} \th \|_{L^2} \| \La^{-1 + \de} P_j \th \|_{L^2} \notag \\
|HL| &\lesssim \sup_\ell \| P_\ell \nb^2 \psi \|_{L^\infty} \| \th \|_{\dot{H}^{-1 + \de}}^2  \label{eq:HL1bd}
}
upon applying the Schur test to the last term.  Here we have used the standard bound
\ALI{
2^k \| P_k \nb_\ell \psi \|_{L^\infty} &= 2^k \| \nb^j \De^{-1} P_{\approx k} P_k \nb_j \nb_\ell \psi \|_{L^\infty}\lesssim \|P_k \nb^2 \psi \|_{L^\infty}.
}

Finally, the bound for $LH$ follows without interchanging the order of summation
\ali{
LH &\sim \sum_k \sum_{j \leq k-1} \int P_k \nb_\ell \psi \, P_j \th \, T^\ell P_{k+1} \th dx  \label{eq:decomposeLH} \\
|LH|&\lesssim \sum_k \sum_{j \leq k-1} \| P_k \nb_\ell \psi \|_{L^\infty} \| P_j \th \|_{L^2} 2^{(-1 + 2 \de)k} \| P_{k+1} \th \|_{L^2} \notag \\
&\lesssim \sum_k \sum_{j \leq k-1} \| P_k \nb_\ell \psi \|_{L^\infty} 2^{(1-\de)j} \| \La^{-1 + \de} P_j \th \|_{L^2} 2^{\de k} \| \La^{-1 + \de} P_{k+1} \th \|_{L^2} \notag \\
&\leq \sum_k \sum_{j \leq k-1} 2^k \| P_k \nb_\ell \psi \|_{L^\infty} 2^{-\de|k - j|} \| \La^{-1 + \de} P_j \th \|_{L^2}  \| \La^{-1 + \de} P_{k+1} \th \|_{L^2} \notag \\
|LH| &\lesssim \sup_\ell \| P_\ell \nb^2 \psi \|_{L^\infty} \sum_{j, k} 2^{-\de|k - j|} \| \La^{-1 + \de} P_j \th \|_{L^2}  \| \La^{-1 + \de} P_{k+1} \th \|_{L^2} \notag \\
|LH| &\lesssim \sup_\ell \| P_\ell \nb^2 \psi \|_{L^\infty} \| \th \|_{\dot{H}^{-1+\de}}^2. \label{eq:IHL2bd}
}
Again we have applied the Schur test.	 Combining \eqref{eq:IHHbd}, \eqref{eq:HL1bd} and \eqref{eq:IHL2bd}, we have proven \eqref{eq:bilinBd}.  

\subsection{The nonperiodic case}  

We now explain how to extend the estimate \eqref{eq:bilinBd} where $\T^2$ is replaced by $\R^2$.  In this case, we define
\ALI{
B[\psi, \th] &= -\int \psi(x) \nb_\ell [T^\ell\th(x) \th(x)] dx = \int \nb_\ell \psi \th T^\ell \th dx
}
for $\psi$ in the Schwartz class and for $\th \in \SS_0$ in the class of Schwartz functions with compact frequency support away from $0$.  The additional details involve handling the low frequencies.

For any integer $m \geq 0$, we have $\psi(x) = P_{\leq -m} \psi + \sum_{k = -m+1}^\infty P_k \psi$, and the series converges uniformly.  Furthermore, $\| P_{\leq -m} \nb \psi \|_{L^\infty(\R^2)} \lesssim 2^{-2m} \| \nb \psi \|_{L^1(\R^2)}$, $\| P_k \nb \psi \|_{L^\infty(\R^2)} \lesssim 2^{2k} \| \nb \psi \|_{L^1(\R^2)}$, and $\| P_k \nb \psi \|_{L^\infty(\R^2)} \lesssim 2^{-k} \| \nb^2 \psi \|_{L^\infty(\R^2)}$  together with $\th T^\ell \th \in L^1$ for $\th \in \SS_0$ allow us to obtain the expansion \eqref{eq:BpsithExpand} by letting $m \to \infty$ and applying dominated convergence and Fubini-Tonelli to interchange the sum and the integral.  Equations \eqref{eq:LL} through \eqref{eq:HH} then follow routinely from $\int P_k \nb_\ell \psi \th T^\ell \th dx = \lim_{N \to \infty} \int P_k \nb_\ell \psi P_{\leq N}\th T^\ell P_{\leq N} \th dx$ once the absolute convergence of the series has been obtained.

Bounding each term in \eqref{eq:LL} through \eqref{eq:HH} requires the decompositions \eqref{eq:decomposeHL} and \eqref{eq:decomposeLH}.  In the case of $\T^2$, the summation over $j$ includes only finitely terms, but it extends to all $j > -\infty$ on $\R^2$.  Since $P_{\leq k - 1} = P_{\leq -m} + \sum_{j = -m+1}^{k-1} P_j$ for any $m$, the key point is to verify the limit
\ali{
\lim_{m \to \infty} \int P_k \nb_\ell \psi (P_{k+1} \th \, T^\ell P_{\leq -m} \th dx + P_{\leq -m} \th \, T^\ell P_{k+1} \th ) dx = 0 
}
for each fixed $k$, every Schwartz function $\psi$ and every $\th \in \SS_0$.  This limit is clear by the compact frequency support of $\th$ away from the origin, since the left hand side is $0$ for sufficiently large $m$.



\subsection{Remark on generalized estimates of this type} \label{sec:remkGen}

Our technique for estimating the nonlinearity can be applied to a more general class of nonlinearities that have the form $\int \psi \th T[\th] dx$.  The only properties of the bilinear form $B(\psi, \th) = \int \nb_\ell \psi \th T^\ell \th dx$ that we have used in the proof are the scaling and anti-symmetry of $T^\ell$, reflected in the oddness and the homogeneity of the multiplier.  
For example, the bound
\ali{
\left| \int \psi_\ell \th \nb^\ell \De \th dx \right| &\lesssim \| \nb \psi \|_{L^\infty} \| \nb \th \|_{L^2}^2, \qquad \psi_\ell, \th \in C_c^\infty \label{eq:localBound}
}  
can be proven by a similar (nonlocal) argument based on the fact that the symbol of $\nb \De$ is odd and degree $3$ homogeneous.  

It is natural to wonder whether the estimate \eqref{eq:localBound} can be proven by an integration by parts argument, but it appears that a nonlocal proof is necessary.  To see that an integration by parts argument would not apply, note that the bound is false if we replace the $L^2(\R^d)$ norms with local $L^2(\supp \psi)$ norms.  To see this failure of the localized version of the estimate, consider a $\th_0$ with $\nb \De \th_0$ nonzero, and choose $\psi_\ell = \chi \nb_\ell \De \th_0$ where $\chi$ is a bump function supported where $\nb \De \th_0$ is nonzero.  Now set $\th_R = \th_0 + R \tilde{\chi}$, where $R$ is large and $\tilde{\chi} = 1$ on the support of $\chi$.  Then $\| \nb \th_R \|_{L^2(\supp \psi)} = \| \nb \th_0 \|_{L^2(\supp \psi)}$ does not depend on $R$, but the left hand side of \eqref{eq:localBound} becomes unbounded as $R$ gets large.

These considerations also motivate why our definition of the space $\dot{H}^{-1+\de}$ requires that we exclude polynomials.  After all the estimate $|\int \psi_\ell \th \nb^\ell \De \th dx| \lesssim \| \nb \psi \|_{L^\infty} \| \th \|_{\dot{H}^1}^2$ would be false for high degree polynomials for which the right hand side would vanish.

\section{Continuous extension of the operator} \label{sec:ctsExtend}
We now have the homogeneous estimate 
\ali{
|B[\psi,\th]| \lesssim_\de \| \nb^2 \psi \|_{L^\infty} \| \th \|_{\dot{H}^{-1 + \de}}^2 \label{eq:homogeneousBound}
}
for $\psi, \th$ Schwartz, and this bound extends by continuity for any Schwartz $\psi$ and any $\th \in \dot{H}^{-1+\de}$ using the corresponding bound \eqref{eq:trilinBd} for the trilinear form and a density and Cauchy sequence argument.  The completion of the Schwartz space under the semi-norm $\| \nb^2 \psi \|_{L^\infty}$ is the space $\dot{C}_0^2$ of $C^2$ functions with bounded second derivatives that satisfy $|\nb^2 \psi(x)| = o(1)$ as $|x| \to \infty$ (see Appendix Section~\ref{sec:functionSpaces}).  The more delicate matter is proving the continuous extension with respect to weak-* convergence, which is the subject of the next section.

\subsection{Extension to \texorpdfstring{$\dot{W}^{2,\infty}$}{the homogeneous Sobolev space}: Definition, boundedness and uniqueness } \label{sec:ExtensionDef}

Now let $\dot{W}^{2,\infty}(\R^2)$ and $\dot{W}^{2,\infty}(\T^2)$ denote the space of distributions $\Psi$ such that $\nb^2 \Psi \in L^\infty$.  This space includes functions (such as quadratic polynomials) for which there is no sequence $\psi_n \in \dot{C}_0^2$ satisfying $\nb^2 \psi_n \to \nb^2 \Psi$ in $L^\infty$, as the latter convergence would imply $\Psi \in \dot{C}_0^2$.  Consequently, our estimate \eqref{eq:homogeneousBound} on $|B[\psi,\th]|$ does not on its own provide a unique continuous extension to allow for $\Psi \in \dot{W}^{2,\infty}$.  

Nonetheless we can uniquely extend the bilinear form $B(\psi, \th)$ to the space of $(\Psi, \th) \in \dot{W}^{2,\infty} \times \dot{H}^{-1+\de}$ such that the extension $B(\Psi, \th)$ also satisfies the homogeneous estimate \eqref{eq:homogeneousBound}, and such that whenever $(\Psi_n, \th_n)$ is a sequence in $\dot{W}^{2,\infty} \times \dot{H}^{-1+\de}$ satisfying $\nb^2 \Psi_n \rightharpoonup \nb^2 \Psi$ in the $L^\infty$ weak-* topology, and $\th_n \to \th$ in the $\dot{H}^{-1+\de}$ norm, we have that $B(\psi_n, \th_n) \to B(\Psi, \th)$.



To define this extension, first let $\th \in \dot{H}^{-1+\de}$ be fixed and $\Psi \in \dot{W}^{2,\infty}$ be given.  Choose a sequence $\psi_n \in C_c^\infty$ such that $\nb^2 \psi_n \hookrightarrow \nb^2 \Psi$ in $L^\infty$ weak-*.  (A construction can be found in Section~\ref{sec:functionSpaces}.)  In particular, there is a uniform bound on $\| \nb^2 \psi_n \|_{L^\infty}$.  
For the sequence $\psi_n$, we use the formulas
\ali{
B(\psi_n, \th) &= HL(\psi_n, \th) + LH(\psi_n,\th) + HH(\psi_n, \th) \\
HL(\psi_n, \th) + LH(\psi_n,\th) &\sim \sum_k \sum_{j \leq k-1} \left\{\int P_k \nb_\ell \psi_n (P_{k+1} \th T^\ell P_j \th + P_j \th \, T^\ell P_{k+1} \th) dx \right\} \label{eq:HLremind} \\
HH(\psi_n,\th) &= \sum_q \left\{\int P_{\leq q-3} \nb_j \nb_\ell \psi_n B_q^{j\ell}[P_{\approx q} \th] dx \right\} \label{eq:HHremind}
}
For $\Psi \in \dot{W}^{2,\infty}$, $\th \in \dot{H}^{-1+\de}$, we define $B(\Psi, \th)$ to be the right hand side with $\psi_n$ replaced by $\Psi$.  For each fixed $k,q \in \Z$, the functions $P_{\leq q-3} \nb_j \nb_\ell \psi_n$ and $2^k P_k \nb_\ell \psi_n = [2^k \De^{-1}\nb^i P_{\approx k}] P_k \nb_i \nb_\ell \psi_n$ (along with their derivatives) are uniformly bounded and they converge {\it uniformly} to $P_{\leq q-3} \nb_j \nb_\ell \Psi$ and $2^k P_k \nb_\ell \Psi$ by a compactness and contradiction argument.  The terms in the braces then converge as $n \to \infty$ for any fixed $k, j \in \Z^2$ or $q \in \Z$, and they are bounded by an absolutely summable series in $(k,j)$ or $q$ by repeating the proof of the bounds for $LH$, $HL$ and $HH$ in Section~\ref{sec:definingNonlinearity}.  By dominated convergence, we obtain $B(\psi_n, \th) \to B(\Psi, \th)$ for any fixed $\th \in \dot{H}^{-1+ \de}$.  

If we desire $B(\Psi, \th)$ to be sequentially continuous in $\Psi$ with respect to weak-* convergence, this argument proves that $B(\Psi, \th)$ is uniquely determined and has the desired continuity property.  The continuous extension maintains the homogeneous bound \eqref{eq:homogeneousBound}, though the optimal constant may in principle be different.  

We also desire the property that $B(\Psi_n, \th_n) \to B(\Psi, \th)$ whenever $\th_n \to \th$ in the $\dot{H}^{-1 + \de}$ seminorm, and $\nb^2 \Psi_n \rightharpoonup \nb^2 \Psi$ in $L^\infty$ weak-*.  This last statement is a corollary of the continuity proven in Section~\ref{sec:continuitySec} below, but it is desirable to have a direct proof.  If $\th_n$ were constant in $n$, the convergence follows as in the previous argument after noting that $\sup_n \| \nb^2 \Psi_n \|_{L^\infty} < \infty$ by the uniform boundedness principle.  Now let $T[\Psi, \th, \phi]$ be the trilinear form, symmetric in $\th,\phi$ for which $B(\Psi, \th) = T[\Psi, \th,\th]/2$.  Then $T[\cdot,\cdot,\cdot]$ is bounded on $\dot{W}^{2,\infty} \times \dot{H}^{-1+\de} \times \dot{H}^{-1+\de}$ as in \eqref{eq:trilinBd}.  We then conclude that
\ALI{
B(\Psi_n, \th_n) - B(\Psi, \th) = T[\Psi_n, \th_n - \th, \th_n]/2 + T[\Psi_n, \th, \th_n - \th]/2 + B(\Psi_n - \Psi, \th)
}
converges to $0$ as $n \to \infty$.

$ $


\subsection{Continuity property of the extended operator} \label{sec:continuitySec}
In this section, we prove another continuity property of the extended operator as the proof will be useful to us later on.

Let $X \subseteq \dot{W}^{2,\infty}$ be bounded in the semi-norm.  We want to show that $B$ is continuous on $X \times \dot{H}^{-1+\de}$ with the appropriate (weak-*$\times$strong) topologies.  We obtain this continuity by showing that the trilinear form $T[\cdot,\cdot,\cdot]$ is continuous as a map on the product $X \times \dot{H}^{-1+\de} \times \dot{H}^{-1+\de}$ into $\R$ when we endow $\dot{H}^{-1+\de}$ with the strong topology induced by the homogeneous norm, and $X$ with the (non-Hausdorff) weak-* topology obtained as the pullback of the map $\Psi \mapsto \nb^2 \Psi$ into $L^\infty$ weak-*.  The continuity of $B(\cdot,\cdot)$ follows then by restricting to the diagonal in $\dot{H}^{-1+\de}\times \dot{H}^{-1+\de}$.  (One can approach this problem by passing to quotients to make $\dot{W}^{2,\infty}$ a normed vector space and cite the metrizability of the weak-* topology on norm bounded subsets of $L^\infty(\R^2)$ together with the sequential continuity just established; however, we believe the following direct approach is more insightful.)

The key point is to obtain a representation of the form
\ali{
B(\Psi,\th) = T[\Psi,\th,\th]/2 &= \int \nb_j \nb_\ell \Psi \, \widetilde{T}^{j\ell}[\th,\th] dx \label{eq:IntegralRep}
}
for some bilinear forms $\widetilde{T}^{j\ell}$ that are symmetric and bounded mapping $\dot{H}^{-1+\de} \times \dot{H}^{-1+\de}$ into $L^1$.  Inspecting \eqref{eq:HLremind}-\eqref{eq:HHremind}, we can write $P_k \nb_{\ell} \Psi = \De^{-1}\nb^j P_k \nb_j \nb_\ell \Psi$, take appropriate adjoint operators, and define $\widetilde{T}^{j\ell} = \widetilde{T}^{j\ell}_{L} + \widetilde{T}^{j\ell}_{H}$, where each form is obtained by polarizing the quadratic forms
\ali{
\widetilde{T}^{j\ell}_{L}[\th,\th] &\sim - \sum_k \sum_{i \leq k-1} \De^{-1} \nb^j P_k \left\{ P_{k+1} \th T^\ell P_i \th + P_i \th \, T^\ell P_{k+1} \th \right\} \label{eq:TLLH} \\
 \widetilde{T}^{j\ell}_{H}[\th,\th] &= \sum_q  P_{\leq q-3}\left\{ B_q^{j\ell}[P_{\approx q} \th]  \right\} \label{eq:THH}
}
The bound $\| \widetilde{T}^{j\ell}[\th,\th] \|_{L^1} \leq \widetilde{C} \| \th \|_{\dot{H}^{-1+\de}}^2$ follows as in the bounds on $LH, HL$, and $HH$ established before, where we use the uniform bounds on the operator norms $\sup_q \| P_{\leq q} \| + \sup_k \| 2^k \De^{-1} \nb P_k \| < \infty$.

Once we have the representation \eqref{eq:IntegralRep}, we can obtain continuity of $T[\cdot,\cdot,\cdot]$ as a mapping from $X \times (\dot{H}^{-1+\de})^2$ into $\R$ with the appropriate product topology (weak-* on $X$ and normed on $\dot{H}^{-1+\de}$) as follows.  Fix a trio $(\Psi_0, \th_0, \phi_0) \in X \times (\dot{H}^{-1+\de})^2$ and let $\ep > 0$ be given.  Letting $\widetilde{C}$ be the constant above and setting $M_X = \sup_{\Psi \in X} \| \nb^2 \Psi \|_{L^\infty}$, we define the following subset of $X \times \dot{H}^{-1 + \de} \times \dot{H}^{-1+\de}$ 
\ALI{
U_{\de_1, \de_2} \coloneqq \Big\{ (\Psi, \th, \phi) : \| \th_0 {-} \th_0 \|_{\dot{H}^{-1+\de}} {+} \| \phi {-} \phi_0 \|_{\dot{H}^{-1+\de}} < \de_1, \left| \int \nb_j \nb_\ell(\Psi {-} \Psi_0) \widetilde{T}^{j\ell}[\th_0,\phi_0] dx \right|  < \de_2 \Big\}.
}
Then $U_{\de_1, \de_2}$ is an open neighborhood of $(\Psi_0, \th_0, \phi_0)$ because the $\widetilde{T}^{j\ell}$ take values in $L^1$ and it is the (finite) product of open neighborhoods in $X$, and $(\dot{H}^{-1+\de})^2$.  If we now choose $\de_1$ and $\de_2$ small enough depending on $\ep$, $\widetilde{C}$, $M_X$, $\| \th_0 \|_{\dot{H}^{-1+\de}}$ and $\| \phi_0 \|_{\dot{H}^{-1+\de}}$, and use the polarization of \eqref{eq:IntegralRep} we have the inequality $|T[\Psi, \th,\phi] - T[\Psi_0, \th_0, \phi_0]| < \ep$ for $(\Psi, \th, \phi) \in U_{\de_1,\de_2}$, which establishes the continuity of $T[\cdot,\cdot,\cdot]$.

\subsection{On the optimal constant of the extension}

For $\de > 0$, define the ``optimal constants'' for the form $B(\cdot,\cdot)$ on the Euclidean space and its extensions 
\ali{
\| B \|_c &= \sup \{ |B(\psi, \th)| ~:~ \| \nb^2 \psi \|_{C^0} \leq 1, \| \th \|_{\dot{H}^{-1+\de}}^2 \leq 1, \psi, \th \in C_c^\infty(\R^2) \} \\
\| B \|_0 &= \sup \{ |B(\Psi, \th)| ~:~ \| \nb^2 \Psi \|_{L^\infty} \leq 1,  \| \th \|_{\dot{H}^{-1+\de}}^2 \leq 1, \Psi \in \dot{W}_0^{2,\infty}, \th \in {\dot H}^{-1+\de}(\R^2) \}. \label{eq:B0Const} \\
\| B \|_{\infty} &= \sup \{ |B(\Psi, \th)| ~:~ \| \nb^2 \Psi \|_{L^\infty} \leq 1, \| \th \|_{\dot{H}^{-1+\de}}^2 \leq 1, \Psi \in \dot{W}^{2,\infty}, \th \in {\dot H}^{-1+\de}(\R^2) \}. \label{sec:BinftyConst}.
}
Here $\dot{W}_0^{2,\infty}$ is the space of distributions with $\nb^2 \Psi(x) = o(1)$ as $|x| \to \infty$, $\nb^2 \Psi \in L^\infty$.  Since $C_c^\infty$ is norm dense in $\dot{H}^{-1+\de}$ (see Appendix Section~\ref{sec:hdots}) and the form $B(\cdot,\cdot)$ is bounded, we can restrict attention to $\th \in C_c^\infty$ in \eqref{eq:B0Const}-\eqref{sec:BinftyConst}.  

It is clear that $\| B \|_c \leq \| B \|_0 \leq \| B \|_{\infty}$ and the reverse inequalities would also be clear from the boundedness of $B$ if $C_c^\infty$ were also norm dense in $\dot{W}^{2,\infty}_0$ or $\dot{W}^{2,\infty}$.  While this strong density is not the case, we do still have $\| B \|_0 = \| B \|_c$.  To see this fact, given $\Psi \in \dot{W}^{2,\infty}_0$ and $\th \in C_c^\infty$, we can choose a sequence $\psi_n \in C_c^\infty$ such that $\nb^2 \psi_n \rightharpoonup \nb^2 \Psi$ in $L^\infty$ weak-* and $\| \nb^2 \psi_n \|_{L^\infty} \to \| \nb^2 \Psi \|_{L^\infty}$ as $n \to \infty$.  (See Appendix Section~\ref{sec:functionSpaces}.)  Then for each such $\Psi \in \dot{W}^{2,\infty}_0$ we have
\ALI{
|B(\Psi, \th)| &= |\lim_{n \to \infty} B(\psi_n, \th) | \leq \liminf_{n \to \infty} \| B \|_c \| \nb^2 \psi_n \|_{C^0} \| \th \|_{\dot{H}^{-1+\de}} = \| B \|_c \| \nb^2 \Psi \|_{L^\infty} \| \th \|_{\dot{H}^{-1+\de}}.
}
This inequality shows that $\| B \|_0 \leq \| B \|_c$, so the two constants are equal.

For $\Psi \in \dot{W}^{2,\infty}$, it is not clear whether such a sequence exists, and we do not know whether $\| B \|_\infty = \| B \|_c$.  (The approximating sequence $\psi_n$ in Section~\ref{sec:ExtensionDef} does satisfy $\| \nb^2 \psi_n \|_{L^\infty} \lesssim \|\nb^2 \Psi \|_{L^\infty}$, but may not satisfy $\| \nb^2 \psi_n \|_{L^\infty} \to \| \nb^2 \Psi \|_{L^\infty}$ as $n \to \infty$.)  However, the above argument does show the equality of constants $\| B \|_\infty = \| B \|_c$ in the periodic setting, where we restrict to mean-zero $\th \in C^\infty(\T^2)$.

\subsection{Special properties and conservation laws of mSQG} \label{sec:mSQGSpecial}

In this Section, we now prove the estimates that are unique to SQG, 2D Euler and the mSQG family, which stem from the following Theorem:
\begin{thm}\label{thm:divForm} Let $m^\ell = \ep^{\ell a} \nb_a (-\De)^{-1+\de}$ be the mSQG multiplier, $\de \geq 0$, and $T^\ell$ denote the corresponding operator.  Then there is a bounded bilinear operator $\widetilde{T}^{j\ell}$ mapping $\dot{H}^{-1 + \de} \times \dot{H}^{-1+\de}$ to $L^1$ that takes values in the space of symmetric {\bf trace-free} tensor fields such that
$\nb_\ell[\th T^\ell \th] = \nb_\ell \nb_j \widetilde{T}^{j\ell}[\th, \th]$ in the sense of distributions.
\end{thm}
This theorem implies the bound in \eqref{thm:oddBound} by integration by parts, and thus concludes the proof of Theorem~\ref{thm:oddBound}.  We note that in the case of the Euler equations ($\de = 0$), the theorem follows easily from the derivation of the Euler equations in vorticity form.  Indeed, in this case the nonlinearity is the curl of the divergence of $T\th \otimes T \th$, which is the second order divergence of a trace-free tensor field ($\nb_j \nb_\ell[ T^j \th \ep^{\ell a} T_a \th]$, where the Biot-Savart law $T^\ell \th = \ep^{\ell a} \nb_a \De^{-1} \th$ is bounded from $\dot{H}^{-1}$ to $L^2$).  The condition of symmetry can then be enforced by replacing this tensor with its symmetric part using the formula $\nb_j \nb_\ell A^{j\ell} = \nb_j \nb_\ell A_S^{j\ell}$, $A_S^{j\ell} = \fr{1}{2} (A^{j\ell} + A^{\ell j})$.  For SQG and more generally mSQG with $\de > 0$, Theorem~\ref{thm:divForm} is apparently much less immediate.

In our analysis so far, the key properties we have used to bound the nonlinearity are the homogeneity and oddness of the multiplier.  These properties are only enough to construct a bilinear form $\widetilde{T} = \widetilde{T}_L + \widetilde{T}_H$ defined in \eqref{eq:TLLH}-\eqref{eq:THH} that has the desired boundedness but does not necessarily take on trace-free values.

We derive the desired trace-free bilinear form $\widetilde{T}$ by examining the derivations of equations \eqref{eq:TLLH}-\eqref{eq:THH} in the specific case of mSQG\footnote{We remark that one can alternatively obtain the double-divergence form by decomposing $\th T^\ell \th$ without decomposing $\psi$, as is used in \cite{isettLooi}.}.  These equations give a representation $\nb_\ell[\th T^\ell \th] = \nb_j \nb_\ell[\widetilde{T}_L^{j\ell} + \widetilde{T}_H^{j\ell}]$.  Let us first consider the term $\widetilde{T}_L^{j\ell}$ that arises from interactions involving low frequencies.  This term may not be trace free, but there exists a trace-free tensor field $\widetilde{T}_{L2}^{j\ell}$ with the same second order divergence and the same boundedness properties.


To obtain this tensor field, consider the following operator ${\cal R}$, which acts on scalar fields and inverts the second order divergence operator while taking values in the space of symmetric, trace-free tensor fields: ${\cal R}^{j\ell}[f] = A \de^{j\ell} \De^{-1} f + B \De^{-2} \nb^j \nb^\ell f$.   The constants $A$ and $B$ solve $A+ B= 1$, $d A+ B = 0$ so that ${\cal R} f$ is trace free and solves $\nb_j\nb_\ell {\cal R}^{j\ell} f = f$ for $f$ with Fourier support in an annulus excluding $0$.  

We now wish to define $\widetilde{T}_{L2}^{j\ell} \equiv {\cal R}^{j\ell} \nb_a \nb_b \widetilde{T}_L^{ab}$ in the hope that we will have $\nb_j\nb_\ell\widetilde{T}_{L2}^{j\ell} = \nb_j \nb_\ell \widetilde{T}_L^{j\ell}$.   However, the fact that the zeroth order operator ${\cal R}^{j\ell} \nb_a \nb_b$ is unbounded on $L^1$ makes it unclear this bilinear form is bounded from $\dot{H}^{-1 + \de} \times \dot{H}^{-1 + \de}$ into $L^1$, or even that it maps into the space of distributions.  
Nonetheless, the boundedness of $\widetilde{T}_{L2}$ is true: ${\cal R}^{j\ell} \nb_a \nb_b \De^{-1} \nb^a P_k = {\cal R}^{j\ell} \nb_b P_k$ satisfies estimates similar to $\De^{-1} \nb^j P_k$ (since ${\cal R}$ has a symbol that is degree $-2$ homogeneous and smooth away from the origin), and
\ali{
\widetilde{T}^{j\ell}_{L2}[\th,\th] &\sim - \sum_k \sum_{i \leq k-1} {\cal R}^{j\ell} \nb_a \nb_b\De^{-1} \nb^a P_k \left\{ P_{k+1} \th T^b P_i \th + P_i \th \, T^b P_{k+1} \th \right\}
}
has the same structure as $\widetilde{T}_L$, so that the same boundedness argument based on Schur's Test applies (modifying the proofs of lines \eqref{eq:HL1bd} and \eqref{eq:IHL2bd}), and the equality of the second order divergences is also guaranteed.

This last argument does not apply to $\widetilde{T}_H$ and does not use the full structure of the mSQG multiplier.  The unique feature of the mSQG multiplier is that the $\widetilde{T}_H^{j\ell}$ defined in \eqref{eq:THH} is in fact already trace-free.  Namely, thanks to formulas \eqref{eq:BqisKqof} and \eqref{def: hat K }, this trace-free property follows from the calculation of line \eqref{eq:itsTraceFree}, which shows that the Jacobian of the frequency-truncated multiplier is trace-free.  From the identity $\nb_\ell[\th T^\ell \th] = \nb_j \nb_\ell[ \widetilde{T}_{L2}^{j\ell} + \widetilde{T}_{HS}^{j\ell}]$, where $\widetilde{T}_{HS}^{j\ell}$ is the symmetric part of $\widetilde{T}_H$, we conclude Theorem~\ref{thm:divForm}.

We now explain how Theorem~\ref{thm:divForm} quickly implies the conservation of impulse and angular momentum for general weak solutions to mSQG, whose Hamiltonian is locally integrable.  The proof is based on the following Lemma:
\begin{lem} \label{lem:conserveLem} Let $Q(x)$ be a polynomial on $\R^2$ of the form $A + B x_1 + Cx_2 + D(x_1^2 + x_2^2)$, so that the trace free part of the Hessian of $Q$ vanishes.  Let $\mrg{F}^{j\ell}$ be a trace-free tensor field in $L^1(\R^2)$ and let $\chi \in C_c^\infty(\R^2)$ be equal to $1$ in a neighborhood of the origin.  Then 
\ALI{
 \lim_{R \to \infty} \int_{\R^2} \mrg{F}^{j\ell} \nb_j \nb_\ell( \chi(x/R) Q(x) ) dx = 0
}
\end{lem}
Indeed, the term $\mrg{F}^{j\ell} \nb_j \nb_\ell Q = 0$ vanishes pointwise, since it is equal to $\mrg{F}^{j\ell}\mrg{\nb}^2_{j\ell} Q$, and the trace-free part of the Hessian $\mrg{\nb}^2_{j\ell} Q = \nb_j\nb_\ell Q - (1/d) \De Q \de_{j\ell}$ vanishes.  When the derivatives hit the cutoff, we can bound $|\mrg{F}| |\nb[\chi(\cdot/R)]| |\nb Q|$ and $|\mrg{F}| |\nb^2 [\chi(\cdot/R)]| |Q|$ pointwise by $\lesssim |\mrg{F}| \in L^1(\R^2)$.  We then apply the dominated convergence theorem using the fact that the derivatives of the cutoff are supported in $|x| \sim R$, to conclude that these latter terms tend to zero as $R \to \infty$.

Now let $\th \in L_t^2 \dot{H}^{-1 + \de}$ be a weak solution to mSQG on the time interval $[0,T]$ and fix $t_0 \in (0,T]$.  We define the restriction to time $t_0$, $\th(t_0,\cdot)$, as the unique distribution on $\R^2$ that satisfies
\ali{
\th(t_0,x) &= \th(t,x) - \int_{t_0}^t \nb_j \nb_\ell \widetilde{T}^{j\ell}[\th(s), \th(s)] ds,  \label{eq:restrictTime}
}
where $\widetilde{T}$ is the bilinear form of Theorem~\ref{thm:divForm}.  The equation $\pr_t \th + \nb_\ell(\th T^\ell \th) = 0$ implies that the element of $\DD'((0,T) \times \R^2)$ defined by the right hand side satisfies $\pr_t \th(t_0, x) = 0$ in $\DD'((0,T) \times \R^2)$, so we can identify it with a unique element of $\DD'(\R^2)$.

In particular, taking the difference $\th(t, \cdot) - \th(0, \cdot)$, we have
\ali{
\th(t, x) &= \th(0,x) + \nb_j \nb_\ell \mrg{F}^{j\ell}_{t}(x), \qquad \mrg{F}_t^{j\ell} = - \int_0^t \widetilde{T}^{j\ell}[\th(s), \th(s)] ds.
}
From $\th \in L_t^2 \dot{H}^{-1+\de}$ and Theorem~\ref{thm:divForm}, we have $\mrg{F}_t \in L^1(\R^2)$ for any $t$, so that Lemma~\ref{lem:conserveLem} applies to give
\ali{
\lim_{R \to \infty} \int_{\R^2} \th(t,x) Q(x) \chi(x/R)  dx &= \lim_{R \to \infty} \int_{\R^2} \th(0,x) Q(x) \chi(x/R) dx, \label{eq:momentConserve}
}
for all polynomials whose Hessian has vanishing trace-free part, which includes the polynomials $x_1^2 + x_2^2$ and $x_i$ that (up to constants) define angular momentum and the components of impulse.  In particular the left hand limit is defined if and only if the right hand limit is defined, where both integrals are interpreted in the sense of distributions, viewing $Q(x)\chi(x/R)$ as a test function.

Note that even if the limit \eqref{eq:momentConserve} does exist, it may a priori depend on the cutoff function $\chi$.  A sufficient condition for the limit to exist and to be independent of $\chi$ is that the initial condition has the form $\th(0,x) = f(x) + \nb_j g^j(x) + \nb_j \nb_\ell h^{j\ell}(x)$ for tensors $f, g$ and $h$ such that $|x|^2 |f|, |x| |g|$, and $|h|$ belong to $L^1$.  The proof is similar to that of Lemma~\ref{lem:conserveLem}.



\section{Sharpness of the estimate on the nonlinearity} \label{sec:Sharpness}

Let $B(\psi, \th)$ be the multilinear form associated to mSQG and let $\| \, \|_X$ be a semi-norm that satisfies
\ali{
|B(\psi, \th)| \lesssim \| \psi \|_X \| \th \|_{H^{-1+\de}}^2 \label{eq:assumeBdX}
}
We have shown that the $\mrg{W}^{2,\infty}$ semi-norm $\| \psi \|_X = \| \mrg{\nb}^2 \psi \|_{L^\infty}$ satisfies the above estimate.  We now prove Theorem~\ref{thm:sharpXBd}, which states that every such semi-norm $X$ that satisfies \eqref{eq:assumeBdX} must control the $\mrg{W}^{2,\infty}$ semi-norm, meaning that the estimate we have proven is sharp (when we assume only $H^{-1+\de}$ control of the scalar).

The proof proceeds by a geometric optics technique that is related to the convex integration construction in \cite{buckShkVicSQG, isett2021direct}, but with a different normalization that extracts an approximate delta function from the bilinear form.  In particular we employ different techniques to estimate the error.

Let $\psi(x_1,x_2)$ be an arbitrary $C_c^\infty$ test function.  Consider the sequence 
\ali{
\begin{split}
\Th_\la = P_{\la} F_\la &\equiv \la^{1-\de} P_{\la} (e^{i \la x_1} \la^{1/2} \varphi(\la^{1/2} x) ), \label{eq:Thladef} \\ 
\bar{\Th}_\la = P_{-\la} \bar{F}_\la &\equiv \la^{1-\de} P_{-\la} (e^{- i \la x_1} \la^{1/2} \varphi(\la^{1/2} x) ),  
\end{split}
}
where $P_{\la}$ is a convolution operator whose symbol $\hat{u}_{\approx \la}$ is a bump function of size $1$ adapted to the frequency region $|\xi - \la e_1| \leq \la/3$. The family $\th_\la = \Th_\la + \bar{\Th}_\la$ is then real-valued.  Let us normalize $\varphi$ to be a bump function supported in the unit ball in physical space with $\int \varphi^2 dx = 1$.  Observe that the family $\th_\la$ is then normalized to have size $\sim 1$ in $\dot{H}^{-1+\de}$, so by our assumption on the $X$ seminorm we have $\sup_\la |B(\psi, \th_\la)| \lesssim \| \psi \|_X$.

Our claim is that 
\ali{
\lim_{\la \to \infty} B(\psi, \th_\la) = 2(-1 + \de) \pr_1 \pr_2 \psi(0,0), \label{claim:geomOptics}
}
hence $|\pr_1 \pr_2 \psi(0,0)| \lesssim \| \psi \|_X$ according to \eqref{eq:assumeBdX}.  By performing a rotation we can modify the construction to show $|(\pr_1 - \pr_2)(\pr_1 + \pr_2)\psi(0,0)| \lesssim \| \psi \|_X$, which implies $|\mrg{\nb}^2 \psi(0,0)| \lesssim \| \psi \|_X$ when combined with \eqref{claim:geomOptics}.  By translation invariance, we conclude that $\| \mrg{\nb}^2 \psi \|_{L^\infty} \lesssim \| \psi\|_X$ as desired.

We have thus proven Theorem~\ref{thm:sharpXBd} under the assumption of claim~\eqref{claim:geomOptics}, whose proof occupies the remainder of this section.  We start by expanding
\ali{
B(\psi, \th_\la) = B(\psi, \Th_\la + \bar{\Th}_\la) &= B(\psi, \Th_\la) + B(\psi, \bar{\Th}_\la) \label{eq:highFreqTerms0} \\
&+ T(\psi, \Th_\la, \bar{\Th}_\la) + T(\psi, \bar{\Th}_\la, \Th_\la), \label{eq:interactionLowFreqTerms}
}
with $T$ is the usual trilinear form associated to $B$.  

The terms of line~\eqref{eq:highFreqTerms0} converge weakly to zero.  For example, by Plancharel we have
\ALI{
B(\psi, \Th_\la) &= \int_{\hat{\R}^2} \hat{\psi} \left(\widehat{\Th_\la} \ast \widehat{T^\ell}[\Th_\la] \right) d\xi
}
Due to the projection $P_{\approx \la}$ in \eqref{eq:Thladef}, the convolution is supported in frequencies of size $\sim \la$.  It is also uniformly bounded by a polynomial in frequency.  Since $\hat{\psi}$ is Schwartz, we conclude that this term, and the similar term $B(\psi, \bar{\Th}_\la)$, tend to $0$ by dominated convergence.

We now isolate the main term of line \eqref{eq:interactionLowFreqTerms} and show that it converges to \eqref{claim:geomOptics}.  We express \eqref{eq:interactionLowFreqTerms} more fully as
\ali{
\eqref{eq:interactionLowFreqTerms} &= \int \psi (T^\ell \Th_\la \nb_\ell \bar{\Th}_\la + T^\ell \bar{\Th}_\la \nb_\ell \Th_\la) dx \label{eq:againstTestFunction}
}
We now exploit the divergence form of the symmetric interaction from Section~\ref{sec:deriveBilinear}.
\ali{
T^\ell \Th_\la \nb_\ell \bar{\Th}_\la &+ T^\ell \bar{\Th}_\la \nb_\ell \Th_\la = \nb_j \nb_\ell B_{\la}^{j\ell}[\Th_\la, \bar{\Th}_\la] \notag \\
B_{\la}^{j\ell}[\Th_\la, \Th_\la] &= \int F_\la(x - h_1) \bar{F}_\la(x - h_2) \tilde{K}_\la^{j\ell}(h_1,h_2) dh_1 dh_2 
\label{eq:bilinearFormConvolve1} 
}
In making this decomposition, we have absorbed the operator $P_\la \otimes P_{-\la}$ in the definition~\eqref{eq:Thladef} into the kernel $\tilde{K}_\la$, which is Schwartz, similar to line \eqref{def: Bilinear form}.  Doing so results in appearances of the bump function $\hat{u}$ in the following formulas:
\ali{
\hat{K}_\la^{j\ell}(\zeta, \eta) &= \int_0^1 \nb^j m_\la^\ell(\si \zeta - (1-\si) \eta) \hat{u}(\zeta) \hat{u}(-\eta) d\si \\
m_\la^\ell(\xi) &= i \ep^{\ell a} \xi_a |\xi|^{2(-1 + \de)} \chi_{\approx \la}(\xi) \notag \\
\nb^j m_\la^\ell(\xi) &= 2 (-1 + \de) i \ep^{\ell a} \xi_a \xi^j |\xi|^{2(-2 + \de)} \chi_{\approx \la} + ... + O(\ep^{\ell j}) 
}
The anti-symmetric part, which is proportional to $\ep^{\ell j}$, will not be important in our calculations, and we have also neglected the term where the derivative hits the cutoff function $\chi_{\approx \la}$ since it will also vanish in the forthcoming computation of the main term.

Using coordinates $h_1 = (h_1^1, h_1^2)$ and $h_2 = (h_2^1, h_2^2)$ on $\R^2 \times \R^2$, we now obtain a leading order term expansion of \eqref{eq:bilinearFormConvolve1}
\ali{
\eqref{eq:bilinearFormConvolve1} &= \la^{2-2\de} \int e^{i \la h_1^1} e^{-i \la h_2^1} \la \varphi^2(\la^{1/2} x) \tilde{K}_\la^{j\ell}(h_1,h_2) dh_1 dh_2   \label{eq:leadingTerm} \\
&+ \underbrace{ \la^{2-2\de} \int e^{i (\la h_1^1 - \la h_2^1)} \int_0^1 \fr{d}{d\si} \la \varphi(\la^{1/2}(x - \si h_1)) \varphi(\la(x - \si h_2)) \tilde{K}_\la^{j\ell}(h_1,h_2) dh_1 dh_2 d\si}_{\equiv E_\la} \label{eq:errorTermEla}
}
Let us first isolate the leading term, \eqref{eq:leadingTerm}, and show how it leads to \eqref{claim:geomOptics}.  Here we recognize that the terms involving $\chi$ do not depend on $(h_1, h_2)$ and what appears is the Fourier transform of $\tilde{K}$.   We have
\ali{
\eqref{eq:leadingTerm}^{j\ell} &= \la^{2 - 2\de}  \la \varphi^2(\la^{1/2} x) \widehat{\tilde{K}^{j\ell}}(\la e_1, - \la e_1) \\
&= 2(-1 + \de) \la^{2 - 2\de}  \la \varphi^2(\la^{1/2} x) \de^j_1 \de^{\ell}_2 \la^{2(-1 + \de)} + O(\ep^{\ell j}) 
}
Here we have used that $\chi_{\approx \la}(\xi)$ and the $\hat{u}$ factors are equal to $1$ in an $O(\la)$ neighborhood of $\la e_1$; in particular, the factor $\nb^j \chi_{\approx \la} $ vanishes at $\la e_1$.

Returning to \eqref{eq:againstTestFunction}, we integrate the main term against the test function $\psi$ to obtain
\ali{
\int \psi \nb_j \nb_\ell \eqref{eq:leadingTerm} dx &= 2 (-1 + \de) \int \nb_j \nb_\ell \psi \la \varphi^2(\la^{1/2} x) \de^j_1 \de^{\ell}_2 dx \\
&= 2(-1 + \de) \pr_1\pr_2 \psi(0,0) + o(1) \quad \mbox{ as $\la \to \infty$}. \label{eq:mainTermRevealed}
}
Hence the leading term establishes the key claim~\eqref{claim:geomOptics}.  It remains to treat the error term $E_\la$ from \eqref{eq:errorTermEla}.

We now treat the error term $E_\la$ from \eqref{eq:errorTermEla}.  We claim that $E_\la$ converges to $0$ in $L^1$.  To see this, we apply Fubini-Tonelli and apply Cauchy-Schwartz, using the translation invariance of the $L^2$ norm, to obtain
\ALI{
&\int |E_\la(x)| dx \lesssim \la^{2 - 2 \de + 1 + \fr{1}{2}} \int_0^1 \left[ \int |\nb \varphi(\la^{1/2} (x - \si h_1))| |\varphi(\la^{1/2}(x - \si h_2))| dx \right] |h| |\tilde{K}_\la(h_1, h_2)| d h_1 dh_2 d\si \\
&\lesssim  (\la^{1/2} \| \nb \varphi(\la^{1/2} \cdot) \|_{L^2}) \cdot (\la^{1/2} \| \varphi(\la^{1/2} \cdot) \|_{L^2}) \cdot \la^{2 - 2 \de + 1/2} \int |h| |\tilde{K}_\la(h_1, h_2)| d h_1 dh_2 \lesssim (1) \cdot (1) \cdot \la^{-1/2}.
}
The last bound follows due to the scaling of the kernel $\tilde{K}_\la(h_1, h_2) = \la^{-2 + 2 \de} \la^{2d} \tilde{K}_1(\la h_1, \la h_2)$, which is Schwartz.  Note that the estimate relies crucially on the gain of $\la^{-1}$ coming from the small factor of $|h|$ and scaling.  Thus $\| E_\la \|_{L^1} = O(\la^{-1/2})$ goes to $0$ as $\la \to \infty$.

\section{Uniqueness of the main estimate for the mSQG multiplier} \label{sec:uniqueness}


In this section we show that the special estimate for the mSQG nonlinearity in Theorem~\ref{thm:oddBound} is not satisfied by any other active scalar nonlinearities with the same translation and scaling symmetries other than the constant multiples of the mSQG nonlinearity.

Consider an active scalar nonlinearity on $\R^2$, which is a nonlinearity of the form 
\ALI{
B(\psi, \th) \equiv \int_{\R^2} \psi T^\ell \th \nb_\ell \th dx
} 
initially defined for $\th$ in the space $\SS_0$ of Schwartz functions with compact Fourier support away from the origin and for $\psi \in C_c^\infty$.  We restate here the theorem we are aiming to prove for the convenience of the reader.
\begin{thm}  Assume that the operator $T^\ell$ has a Fourier multiplier $m^\ell$ that is homogeneous of degree $-1 + 2 \de$, $\de \geq 0$, $\de \neq 1$.  If the nonlinearity $B(\psi, \th)$ satisfies the bound
\ali{
|B(\psi, \th)| &\lesssim \| \mrg{\nb}^2 \psi \|_{L^\infty} \| \th \|_{H^{-1+\de}}^2 \label{eq:specialBound} \tag{S}
}
for $\psi \in C_c^\infty(\R^2)$ and $\th \in \SS_0(\R^2)$, then $T^\ell$ is a constant multiple of the mSQG Biot Savart law.  
\end{thm}

\subsection{Reducing to the case of a smooth multiplier}

Let us reduce to the case where the multiplier is smooth.  We can argue by contradiction.  Suppose $m^\ell$ is a homogeneous multiplier that is not a constant multiple of the mSQG multiplier, but for which the corresponding nonlinearity nonetheless satisfies \eqref{eq:specialBound}.  We first reduce to the case where $m^\ell$ is smooth away from the origin by exploiting a natural action of the group of rotations.  

For $\a \in \R$, let $R_\a$ denote the one parameter group of linear operators that act on $\R^2$ by rotation by the angle $\a$, and let $[R_\a]_k^\ell$ denote the corresponding matrix entries.  The multiplier $m^\ell(\xi)$ can be regarded as a vector field on $\widehat{\R}^2$, and the group $R_\a$ acts on this space of vector fields by pushforward: $(R_\a m)^\ell(\xi) \equiv [R_\a]_k^\ell m^k(R_{-\a} \xi)$.  The mSQG multiplier, which we denote by $M^\ell(\xi)$, is up to a constant the unique vector field invariant under this action with the $-1 + 2\de$ homogeneity.  

Letting $f \circ R_\a(x) = f(R_\a x)$, since $R_\a$ is orthogonal, the formula $\widehat{f \circ R_\a}(\xi) = \hat{f}(R_{\a} \xi)$ yields that the rotation group correspondingly acts on the operator $T^\ell$ associated with $m^\ell$ by $(R_\a T)^\ell[\th](x) \equiv [R_\a]^\ell_k T^k[\th \circ R_{\a}](R_{-\a} x)$.  From this identity, we can check that the action of the rotation group respects the estimate \eqref{eq:specialBound}, for if \eqref{eq:specialBound} holds for $T^\ell$ then for the rotated form $B_\a$, we have
\ALI{
B_\a[\psi, \th] &\equiv \int \psi(x) (R_\a T)^\ell[\th] \nb_\ell \th(x) dx = \int \psi(x) [R_\a]_k^\ell T^k[\th \circ R_{\a}](R_{-\a} x) \nb_\ell \th(x) dx \\
&= \int \psi(R_{\a} x)  T^k[\th \circ R_{\a}](x) [R_\a]_k^\ell \nb_\ell \th(R_{\a} x) = \int \psi(R_{\a} x) T^k[\th \circ R_{\a}](x) \nb_k[\th \circ R_{\a}](x) \\
B_\a[\psi, \th] &= B_0[\psi \circ R_\a, \th \circ R_\a]
}
Thus the bilinear form $B_\a$ associated to $R_\a T$ maintains the bound \eqref{eq:specialBound}, since both $\| \mrg{\nb}^2 \psi \|_{L^\infty}$ and $\| \th \|_{\dot{H}^s}$ are invariant under rotations.

Now consider again the operator $T$ that we have assumed is not the mSQG operator but still satisfies \eqref{eq:specialBound}.  Define the twisted mollification of $T$ to be
\ali{
(T_\ep)^\ell = \int_\R (R_{\a} T)^\ell \eta_\ep(\a) d\a
}
where $\eta_\ep$ is a standard mollifier supported in $|\a| \leq \ep$.  Then $T_\ep$ inherits the estimate \eqref{eq:specialBound} since we are simply averaging over rotations that preserve the estimate.  Moreover, we claim the multiplier of $T_\ep$ is smooth and distinct from the mSQG multiplier for $\ep$ small enough. 

Indeed, since $m^\ell$ is a homogeneous distribution on $\widehat{\R^2}$, away from the origin $m$ has the form $m^\ell(\xi) = h^\ell(\om) |\xi|^{-1 + 2 \de}$, where $\om \in \R / \Z$ is the polar angle on $\widehat{\R^2}$ and $h^\ell$ is an $\R^2$-valued distribution on the $1D$ torus $\R / \Z$ (see e.g. \cite{hormanderBook}).  The operation of twisted mollification forming the operator $T_\ep$ simply replaces the distribution $h$ by the distribution $h_\ep^\ell(\om) = \int_\R  h^k(\om - \a) \eta_\ep(\a) [R_\a]_k^\ell d\a $, where the integral and the translation of $h$ are interpreted in the sense of distributions using duality.  

In the same way that a usual convolution produces a smooth function from a distribution, the distribution $h_\ep$ can be shown to be smooth.  Furthermore, since $R_\a$ is smooth and $R_0 = \mbox{Id}$, one can check that $h_\ep$ converges to $h$ in $\DD'(\R/\Z)$ as $\ep \to 0$ by calculating the dual operator that acts on test functions, and finally showing that the dual twisted mollification converges in $C^k(\R/\Z)$ for any test function and any $k \geq 0$.  We omit the details, which are similar to the usual process of convolution for distributions (see e.g. \cite{hormanderBook}).

As a result, $T_\ep$ must be distinct from the mSQG operator for $\ep$ sufficiently small, for we have assumed that the limiting multiplier $m^\ell(\xi) = h^\ell(\om) |\xi|^{-1 + 2 \de}$ is distinct from the mSQG multiplier.  Also, as pointed out $T_\ep$ inherits the estimate \eqref{eq:specialBound}.  Thus, if we can show $T_\ep$ is the mSQG Biot-Savart Law for every $\ep > 0$, the limiting operator $T$ is also.  

We now assume $T$ has a smooth multiplier satisfying \eqref{eq:specialBound} and $-1+2\de$ homogeneity, and proceed to show $T$ is a multiple of the Biot-Savart law for mSQG.


\subsection{The case of an odd multiplier}

By the analysis of the previous section, if the bound \eqref{eq:specialBound} holds for an operator $T^\ell$ with multiplier $m^\ell$, the same bound holds for the odd part of the operator $T_o^\ell = (T^\ell + (R_\pi T)^\ell)/2$ whose multiplier is $m_o^\ell(\xi) = (m^\ell(\xi) - m^\ell(-\xi))/2$, and also for the even part of the operator $T_e^\ell = (T^\ell - (R_\pi T)^\ell)/2$.  We will show that $T_o^\ell$ is a multiple of the mSQG Biot-Savart law and that $T_e^\ell$ is $0$ (except possibly in the case where the multiplier is degree $1$ homogeneous).

Let us now consider the case where the multiplier $m^\ell = m_o^\ell$ is odd.  We first establish that, assuming oddness of the multiplier, the velocity field $T^\ell$ must be incompressible in the sense that $m(\xi) \cdot \xi = 0$.

\begin{prop} \label{prop:notIncompressibleCase} If the velocity field $T^\ell$ is not incompressible, then there is a test function $\psi$ and a sequence $\th_\la$ with Fourier support in an annulus $\{ |\xi| \sim \la \}$ such that $\| \th_\la \|_{\dot{H}^{\de}} = O(1)$ and such that $B(\psi, \th_\la)$ has a nontrivial limit.  As a consequence, there cannot exist any estimate of the form
\ali{
|B(\psi, \th)| \lesssim \|\psi\|_X \| \th \|_{H^s}^2, \qquad \psi \in \CC_c^\infty, \th \in H^s \label{eq:illegalDivodd}
}
for any $s < \de$ and any norm $X$ on $\CC_c^\infty$.
\end{prop}
Indeed, if an estimate of the type \eqref{eq:illegalDivodd} were to exist, it would force $B(\psi, \th_\la) \to 0$ for the example we now construct in view of the estimate $\| \th_\la \|_{H^s} \lesssim \la^{s- \de} \| \th_\la \|_{\dot{H}^\de}$ for functions with Fourier support in a annulus of radius $\sim \la$.

\begin{proof} As $T$ is not incompressible-valued, there exists a frequency $\bar{\xi}$, $|\bar{\xi}| = 1$ such that $m(\bar{\xi}) \cdot \bar{\xi} \neq 0$.  Let $\eta_\la$ be a Schwartz function that is a bump function in frequency space adapted to $|\la \bar{\xi} - \xi| \leq \la/2$.  Define the operator $P_\la = \eta_\la \ast$ and define a bump function $\chi_\la(x) = \la^{-\de} \la^{d/4} \chi(\la^{1/2} x)$ (with $d = 2$)
 and define
\ALI{
\th_\la = (\Th_\la + \overline{\Th}_\la) = (\Th_\la + \Th_{-\la}),  \qquad \Th_\la = P_\la[e^{i \la \bar{\xi} \cdot x} \chi_\la(x)]. \qquad
}
We need the following geometric optics Lemma:
\begin{lem}\label{lem:leading order}  There is a leading order expansion of the form
\ali{
T^\ell \Th_\la &= e^{i \la \bar{\xi} \cdot x} \la^{-1 + 2\de} \chi_\la m^\ell(\bar{\xi}) + \mathrm{ErrV}^\ell \label{eq:leadingVelocity}\\
\| \mathrm{ErrV} \|_{L^2} &= O(\la^{-1/2} \la^{-1 + 2 \de} \la^{-\de}) = O(\la^{-3/2 + \de}) \label{eq:errorVelocBd1}\\
\nb_\ell \Th_\la &= (i \la) \bar{\xi}_\ell e^{i \la \bar{\xi} \cdot x} \chi_\la  + \mathrm{ErrS}_\ell \label{eq:leadScal1} \\
\| \mathrm{ErrS} \|_{L^2} &= O(\la^{1 - \de - 1/2})\label{eq:errScal1}
}
\end{lem}
We prove the expansion for the velocity field.  The proof for the scalar gradient is similar.  Let $K_\la^\ell$ denote the convolution kernel representing $T^\ell P_\la$.  Observe that
\ali{
T^\ell \Th_\la &= \int e^{i\la\bar{\xi}\cdot(x - h)} \chi_\la(x-h) K_\la^\ell(h) dh \\
&= e^{i\la\bar{\xi}\cdot x}\int e^{-i\la\bar{\xi}\cdot h} \chi_\la(x-h) K_\la^\ell(h) dh \\
&= e^{i\la\bar{\xi}\cdot x}\int e^{-i\la\bar{\xi}\cdot h} \chi_\la(x) K_\la^\ell(h) dh + \mathrm{ErrV} 
}
Notice that the leading order term is equal to the one in \eqref{eq:leadingVelocity}.  We estimate the error term by writing
\ALI{
\mathrm{ErrV} &= e^{i\la\bar{\xi}\cdot x}\int \int_0^1  e^{-i\la\bar{\xi}\cdot h} \nb_a\chi_\la(x - \si h) h^a K_\la^\ell(h) d\si dh \\
\| \mathrm{ErrV} \|_{L^2} &\lesssim \sup_{\si,h} \| \nb \chi_\la(\cdot - \si h) \|_{L^2} \cdot \| h |K_\la| \|_{L^1} \\
&\lesssim \la^{1/2 - \de} \cdot \la^{-1 - 1 + 2 \de } = O(\la^{-3/2 + \de})
}
Note that the reason we gain by a factor of $\la^{-1/2}$ compared to the main term is that differentiating $\chi$ costs only $\la^{1/2}$, but gains a factor of $\la^{-1}$ coming from the factor of $h$.  The same observations lead to the other equations \eqref{eq:leadScal1} and \eqref{eq:errScal1}.

We now let $\psi$ be any test function with $\psi(0) = 1$.  We claim that
\ALI{
B(\psi, \th_\la) &= \int \psi T^\ell \th_\la \nb_\ell \th_\la dx = \int \psi (T^\ell \Th_\la \nb_\ell \Th_{-\la} + T^\ell \Th_\la \nb_\ell \Th_{-\la} )  dx + o(1) 
}
To see that the other terms are $o(1)$, consider $\int \psi T^\ell \Th_\la \nb_\ell \Th_\la$.  We can write this as an integral in frequency space by Plancharel, and note that the Fourier transform of $T^\ell \Th_\la \nb_\ell \Th_\la$, which is localized to frequencies of the order $\la$, has $L^2$ norm bounded by a polynomial in $\la$.  By applying dominated convergence in frequency space, the integral tends to zero as claimed.

We now consider the main term, which we expand as
\ALI{
B(\psi, \th_\la) &= (- i \la ) \la^{-1 + 2 \de} \int \psi \chi_\la^2 ( m^\ell(\bar{\xi}) \bar{\xi}_\ell - m^\ell(-\bar{\xi}) \bar{\xi}_\ell) dx  \\
&+ \la^{-1 + 2 \de} \int \psi \chi_\la m^\ell(\bar{\xi}) \mathrm{ErrS}_\ell - (i\la) \int \psi \chi_\la \bar{\xi}_\ell \mathrm{ErrV}^\ell \\
&+ \int \psi \mathrm{ErrS}_\ell \mathrm{ErrV}^\ell dx
}
By bounding $\| \psi \|$ in $L^\infty$ and taking the other factors in $L^2$ using $\|\chi_\la \|_{L^2} \lesssim \la^{-\de}$ and \eqref{eq:errorVelocBd1},\eqref{eq:errScal1}, we see that the terms on the second line are $O(\la^{-1/2})$ while the error term in the last line is $O(\la^{-1})$.

The term on the first line is therefore the main term.  The cutoff $\chi_\la$ is normalized so that $\la^{2\de} \chi_\la^2~\rightharpoonup~A \de_0$ for some $A \neq 0$.  Since we have assumed a choice of $\bar{\xi}$ such that $ m^\ell(\bar{\xi}) \bar{\xi}_\ell - m^\ell(-\bar{\xi}) \bar{\xi}_\ell = 2 m^\ell(\bar{\xi})\bar{\xi}_\ell$ is not zero (recall that $m$ is odd) and a choice of $\psi$ that does not vanish at the origin, we see as desired that $B(\psi, \th_\la)$ has a nontrivial limit as $\la \to \infty$.
\end{proof}

We have now established that the velocity field is incompressible; that is, the multiplier satisfies $m(\xi) \cdot \xi = 0$.  Our next task is to show that the mSQG nonlinearity is the only one consistent with the special bound \eqref{eq:specialBound}.  This fact follows from the following proposition.

\begin{prop}  Let $T^\ell$ be incompressible with odd multiplier $m$ that is homogeneous of degree $-1 + 2\de$.  If $m$ is not a scalar multiple of the mSQG multiplier, then there cannot exist a bound of the form \eqref{eq:specialBound} for the corresponding nonlinearity.
\end{prop}
\begin{proof}
Let $M^\ell$ be the mSQG multiplier that is homogeneous of degree $-1 + 2 \de$, which is a vector field tangent to the unit circle of constant length when restricted to the unit circle.  Since we work in two dimensions and $m(\xi)$ satisfies $m(\xi) \cdot \xi = 0$ (i.e. $m$ is also tangent to the unit circle), we can deduce that $m$ has the form $m^\ell(\xi) = H(\xi) M^\ell(\xi)$ for some function $H$ that is homogeneous of degree zero and smooth away from the origin.

Suppose now that $m$ is not a constant multiple of the mSQG multiplier.  Then $H$ is not constant on the unit circle, and thus there exists a $\bar{\xi} \in S^1$ such that $\nb_\ell H(\bar{\xi}) M^\ell(\bar{\xi}) \neq 0$. 

We now construct a family of scalars $\Th_\la$ with frequency support in $\xi \sim \la$ that oscillate in the $\bar{\xi}$ direction.  Set
\ALI{
\Th_\la &= \la^{1-\de} P_{\approx \la}[e^{i\la \bar{\xi} \cdot x} \varphi_\la(x)] 
}
where the multiplier for $P_{\approx \la} $ is a bump function adapted to $|\xi - \la \bar{\xi}| \leq \la/4$, and where $\varphi_\la(x) = \la^{d/2} \varphi(\la x)$ is a rescaled bump function normalized so that $\int \varphi_\la^2 dx = \int \varphi(x) dx = 1$.  Define $\th_\la = \Th_\la + \bar{\Th}_\la$, so that $\th_\la$ is real-valued.  Note that the family $\th_\la$ is uniformly bounded in $\dot{H}^{-1 + \de}$.  Also note the direct analogy with the construction of  $\Th_\la$ in line \eqref{eq:Thladef}.

Now let $\psi$ be an arbitrary test function.  By incompressibility we have
\ALI{
B(\psi,\th_\la) = \int \psi T^\ell \th_\la \nb_\ell \th_\la dx = - \int \nb_\ell \psi \th_\la T^\ell \th_\la dx 
}
We also have $\int \nb_\ell \psi \,T^\ell \Th_\la \, \Th_\la dx = o(1)$ as $\la \to \infty$ and similarly for $\Th_\la$ replaced by $\bar{\Th}_\la$ since $\nb_\ell \psi$ is Schwartz and $T^\ell \Th_\la \Th_\la$ has frequency support in $|\xi| \sim \la$ and has $L^2$ norm bounded by a polynomial in $\la$. 

This vanishing leaves us with 
\ALI{
- B(\psi, \th_\la) &= \int \nb_\ell \psi (\underbrace{T^\ell \Th_\la \bar{\Th}_\la + T^\ell \bar{\Th}_\la \Th_\la}_{\equiv N^\ell_\la} ) dx + o(1).
}
In other words the main term is the high-high to low frequency cascade once we establish it is nonzero.

Now recall that the machinery of Section~\ref{sec:deriveBilinear} gives us a divergence form representation of the form
\ALI{
N^\ell_\la &= \nb_j B_\la^{j\ell}[\Th_\la, \bar{\Th}_\la] \\ 
B_\la^{j\ell}[\Th_\la, \bar{\Th}_\la] &= \int_{\R^2 \times \R^2} K_\la^{j\ell}(h_1, h_2) \Th_\la(x - h_1) \bar{\Th}_\la(x - h_2) dh_1 dh_2 \\
\widehat{K}_\la^{j\ell}(\zeta, \eta) &= -i \int_0^1 \nb^j m_\la^\ell(\si \zeta - (1-\si) \eta) d\si \hat{\chi}_\la(\zeta) \hat{\chi}_\la(\eta) 
}
where $m_\la$ is a cutoff version of $m$ adapted to frequency $\la$ and where $\hat{\chi}_\la$ is a cutoff adapted to frequencies of size $\la$.  We can also absorb the factors of $P_{\la}$ into the kernel $K$ to get 
\ali{
&N^\ell_\la = \nb_j \tilde{B}_\la^{j\ell}[ e^{i \la \bar{\xi} \cdot x} \varphi_\la, e^{-i\la \bar{\xi} \cdot x} \varphi_\la] \\
\tilde{B}_\la^{j\ell}&[\Th_\la, \bar{\Th}_\la] = \la^{2 - 2\de} \int_{\R^2 \times \R^2} \tilde{K}_\la^{j\ell}(h_1, h_2) e^{i \la \bar{\xi} \cdot (x - h_1)} \varphi_\la(x - h_1) e^{- i \la \bar{\xi} \cdot (x - h_2)} \varphi_\la(x - h_2) dh_1 dh_2  \label{eq:bilinFormToExpand} \\
&\eqref{eq:bilinFormToExpand} = \la^{2-2\de} \int e^{i \la \bar{\xi} \cdot h_1} e^{-i \la \bar{\xi} \cdot h_2} \la \varphi^2(\la^{1/2} x) \tilde{K}_\la^{j\ell}(h_1,h_2) dh_1 dh_2   \label{eq:leadingTerm2} \\
&+ \underbrace{ \la^{2-2\de} \int e^{i \la \bar{\xi} \cdot( h_1 - h_2)} \int_0^1 \fr{d}{d\si} \la \varphi(\la^{1/2}(x - \si h_1)) \varphi(\la^{1/2}(x - \si h_2)) \tilde{K}_\la^{j\ell}(h_1,h_2) dh_1 dh_2 d\si}_{\equiv E_\la} \label{eq:errorTermEla2}
}
We then have
\ali{
B(\psi, \th_\la) &= \int \nb_j \nb_\ell \psi \eqref{eq:leadingTerm2}^{j\ell} dx + \int \nb_j\nb_\ell\psi E_\la^{j\ell} dx + o(1)
}
By the computation following \eqref{eq:mainTermRevealed}, we have that $\int \nb_j\nb_\ell\psi E_\la^{j\ell} dx = o(1)$.  The leading term is then derived using the degree $-1 + 2\de$ homogeneity of the multiplier as
\ALI{
 \eqref{eq:leadingTerm2}^{j\ell} &= \la^{2-2\de} \la \varphi^2(\la^{1/2} x) \widehat{ \tilde{K}_\la^{j\ell}}(\la \bar{\xi}, - \la \bar{\xi})  \\
&= \la \varphi^2(\la^{1/2} x) \nb^j m^\ell(\bar{\xi}) \\
&= \la \varphi^2(\la^{1/2} x) (\nb^j H(\bar{\xi}) M^\ell(\bar{\xi}) + H \nb^j M^\ell(\bar{\xi}))
}
where we recall that $m^\ell = H(\xi) M^\ell$, where $M^\ell$ is the mSQG Biot-Savart law.  We can thus write
\ALI{
\eqref{eq:leadingTerm2}^{j\ell} &= \la \varphi^2(\la^{1/2} x)\left( \fr{1}{d}\nb_i H(\bar{\xi}) M^i(\bar{\xi}) \de^{j\ell} + \mrg{A}^{j\ell} \right)
}
where $\mrg{A}^{j\ell}$ is a (fixed, constant) trace-free matrix, since we know $\nb^j M^\ell$ is trace-free.

Assuming the special bound \eqref{eq:specialBound}, we conclude that 
\ALI{
\left| \fr{\nb_i H(\bar{\xi}) M^i(\bar{\xi})}{d} \De \psi(0) + \mrg{A}^{j\ell} \nb_j \nb_\ell \psi(0) \right| = \lim_{\la \to \infty} |B(\psi, \th_\la)| \lesssim \| \mrg{\nb}^2 \psi \|_{L^\infty} 
}
since $\limsup \| \th_\la \|_{\dot{H}^{-1+\de}}^2$ is bounded.  This bound, however, implies the estimate $\| \De \psi \|_{L^\infty} \lesssim \| \mrg{\nb}^2 \psi \|_{L^\infty}$ for all test functions, which is a false estimate (as can be seen by considering test functions that approximate $\chi(x) |x|^2 \log |x|^2$, $\chi$ a smooth bump function).

Our conclusion is then that the function $H$ in the decomposition $m^\ell(\xi) = H(\xi) M^\ell(\xi)$ must be a constant, so that there is no point on the unit circle where $M(\xi) \cdot \nb H(\xi) \neq 0$.  Thus, $m$ is a constant multiple of the mSQG Biot-Savart law under the assumption that the multiplier is odd.
\end{proof}

\subsection{The case of an incompressible velocity field with even multiplier}

We want to show that, outside a certain exceptional case, the active scalar nonlinearities that satisfy the exceptional bound \eqref{eq:specialBound} must have vanishing even part.  We first eliminate the case of an incompressible even multiplier (that is, $\nb_\ell T^\ell[\th] = 0$).  In this case the nonlinearity has the form
\ali{
B(\psi, \th) &= - \int \nb_\ell \psi \, \th \, T^\ell \th dx  \label{eq:incompForm}
}
after an integration by parts. 
\begin{prop}\label{prop:ifIncompressibleEven}  If the velocity field $T^\ell$ is incompressible and the multiplier $m$ has nonzero even part, then there is a test function $\psi$ and a sequence $\th_\la$ with Fourier support in an annulus $\{ |\xi| \sim \la \}$ such that $\| \th_\la \|_{\dot{H}^{-1/2+\de}} = O(1)$ and such that $B(\psi, \th_\la)$ has a nontrivial limit.  As a consequence, there cannot exist any estimate of the form
\ali{
|B(\psi, \th)| \lesssim \|\psi\|_X \| \th \|_{H^s}^2, \qquad \psi \in \CC_c^\infty, \th \in H^s \label{eq:illegalDivodd2}
}
for any $s <-1/2 + \de$ and any norm $X$ on $\CC_c^\infty$.
\end{prop}
\begin{proof}
Indeed, suppose the even part of $m$ takes on a nonzero value at some $\bar{\xi}$ on the unit circle.  We construct 
\ALI{
\Th_\la &= \la^{1/2 - \de} P_{\la}[ e^{i \la \bar{\xi} \cdot x} \la^{d/4} \varphi(\la^{1/2} x) ], \qquad \th_\la = \Th_\la + \bar{\Th}_\la
}
with the properties that $\varphi$ is a bump function with $\int |\varphi|^2 dx = 1$, $P_\la$ is a Fourier multiplier adapted to $|\xi - \la \bar{\xi}| < \la/5$.  

Similar to Lemma~\ref{lem:leading order}, we have a leading order expansion where the $L^2$ norm of the errors improves on the main term by a factor of $\la^{-1/2}$
\begin{lem}\label{lem:leading order2}  There is a leading order expansion of the form
\ali{
T^\ell \Th_\la &= \la^{-1/2 + \de} e^{i \la \bar{\xi} \cdot x}  \la^{d/4} \varphi(\la^{1/2}x) m^\ell(\bar{\xi}) + \mathrm{ErrV}^\ell \label{eq:leadingVelocity2}\\
\| \mathrm{ErrV} \|_{L^2} &= O(\la^{-1/2} \la^{-1/2 + \de} ) = O(\la^{-1 + \de}) \label{eq:errorVelocBd12}\\
\Th_\la &= \la^{1/2-\de} e^{i \la \bar{\xi} \cdot x} \la^{d/4} \varphi(\la^{1/2} x)  + \mathrm{ErrS} \label{eq:leadScal12} \\
\| \mathrm{ErrS} \|_{L^2} &= O(\la^{-1/2} \la^{1/2 - \de}) = O(\la^{- \de}) \label{eq:errScal12}
}
\end{lem}
We omit the proof, which is essentially the same as that of Lemma~\ref{lem:leading order}.  

Arguing as in the proof of Proposition~\ref{prop:notIncompressibleCase} and using \eqref{eq:incompForm} we obtain
\ALI{
- B(\psi, \th_\la) &= \int \nb_\ell \psi ( \bar{\Th}_\la T^\ell \Th_\la  + \Th_\la T^\ell \bar{\Th}_\la ) dx + o(1) \\
&=  \int \nb_\ell \psi \la^{d/2} \varphi^2(\la^{1/2} x) ( m^\ell(\bar{\xi}) + m^\ell(-\bar{\xi}) ) dx + o(1) \\
&= \nb_\ell \psi(0) (m^\ell(\bar{\xi}) + m^\ell(-\bar{\xi})) + o(1).
}
Since $\bar{\xi}$ was chosen so that $m^\ell(\bar{\xi}) + m^\ell(-\bar{\xi}) \neq 0$, this limit is nontrivial if $\psi$ is chosen appropriately, which concludes the proof of Proposition~\ref{prop:ifIncompressibleEven}.
\end{proof}

\subsection{The case of a more general even multiplier}
 
It is more difficult to rule out an estimate of the form \eqref{eq:specialBound} when we consider compressible velocity fields with an even multiplier.  In this case we must do a more careful analysis of the nonlinearity
\ALI{
Q[\th] &= T^\ell \th \nb_\ell \th.
}
There is a cancellation in this nonlinearity because the multiplier $m^\ell$ of $T^\ell$ is even while the multiplier $\iota_\ell = i \xi_\ell$ for $\nb_\ell$, has the opposite parity, being odd.  In what follows we will skip several details since they are similar to what has already been presented.

Suppose that $\th = \Th_\la$ has frequency support in $|\xi_\la|$ and let $m_\la$ be a localized version of $m$ with frequency support in $|\xi| \sim \la$.  We can perform frequency space calculations that are analogous to those of Section~\ref{sec:deriveBilinear} (in particular the derivation of \eqref{eq: hat Q formula}) to obtain
\ALI{
m_\la^\ell(\xi - \eta) \iota_\ell(\eta) &+ m_\la^\ell(\eta) \iota_\ell(\xi - \eta) = - m_\la^\ell(\xi - \eta) \iota_\ell(-\eta) + m_\la^\ell(-\eta) \iota_\ell(\xi - \eta) \\
&= \int_0^1 \fr{d}{d\si} \left( \iota_\ell(\si \xi - \eta) m_\la^\ell((1-\si) \xi - \eta) \right) d\si \\
&= \xi_a \int_0^1 \left( \nb^a \iota_\ell(\si \xi - \eta) m_\la^\ell((1-\si) \xi - \eta) - \iota_\ell(\si \xi - \eta) \nb^a m_\la^\ell((1-\si) \xi - \eta) \right) d\si \\
&= i \xi_a  \hat{K}^{a}_\la(\xi - \eta, \eta).
}
The definition of $\iota_\ell = i \xi_\ell$ gives $\nb^a \iota_\ell = i \de^a_\ell$. We set
\ali{
\widehat{\tilde{K}^a_\la}(\zeta, \eta) &\equiv i \int_0^1 \left( m^a(\si \zeta - (1-\si) \eta ) - (\si \zeta_\ell - (1-\si) \eta_\ell) \nb^a m^\ell(\si \zeta - (1-\si) \eta) \right) d\si \chi_\la (\zeta) \chi_\la(-\eta), \notag
}
where $\chi_\la$ is the frequency cutoff adapted to $|\xi - \la \bar{\xi}| \leq \la/5 $ that is the symbol for the $P_\la$ in our usual definition of 
\ALI{
\Th_\la = \la^{1/2 - \de} P_\la[e^{i \la \bar{\xi} \cdot x} \la^{d/4} \phi(\la^{1/2} x)].
}
We then have
\ALI{
\widehat{\tilde{K}^a_\la}(\la \bar{\xi}, -\la \bar{\xi}) = 
\hat{K}^a_\la(\la \bar{\xi}, -\la \bar{\xi}) &= i m^a(\la \bar{\xi}) -i \la \bar{\xi}_\ell \nb^a m^\ell( \la \bar{\xi}) =i  \la^{-1 + 2\de} ( m^a(\bar{\xi}) - \bar{\xi}_\ell \nb^a m^\ell(\bar{\xi}))  \\
\widehat{\tilde{K}^a_\la}(\la \bar{\xi}, -\la \bar{\xi}) &= \la^{-1 + 2\de}( m^a(\bar{\xi}) - \bar{\xi}_\ell \nb^a m^\ell(\bar{\xi})),
}
and using this expression we obtain for any test function $\psi$ 
\ALI{
B(\psi, \Th_\la + \bar{\Th}_\la) &= \int \psi (T^\ell \Th_\la \, \bar{\Th}_\la + T^\ell \bar{\Th}_\la \, \Th_\la) dx + o(1) \\
&= \int \psi \la^{1-2\de} \la^{d/2} \nb_a \int \tilde{K}_\la^a(h_1, h_2) e^{i \la \bar{\xi} \cdot( h_1 - h_2)} \varphi(x -h_1) \varphi(x - h_2) dh_1 dh_2 dx + o(1) \\
&= - \int \nb_a \psi \la^{1-2\de} \widehat{\tilde{K}^a_\la}(\la \bar{\xi}, -\la \bar{\xi}) \la^{d/2} \varphi^2(\la^{1/2}(x)) dx + o(1) \\
&= - \nb_a \psi(0) ( m^a(\bar{\xi}) - \bar{\xi}_\ell \nb^a m^\ell(\bar{\xi})) + o(1).
}
As in the previous arguments, we can 
conclude that the nonlinearity has no bounded extension to $C_c^\infty \times \dot{H}^s \times \dot{H}^s$ for $s < -1/2 + \de$ if there exists any $\bar{\xi} \in S^1$ that makes the above expression nonzero.  Thus \eqref{eq:specialBound}, which implies boundedness on $\dot{W}^{2,\infty} \times \dot{H}^{-1+\de} \times \dot{H}^{-1+\de}$, here implies
\ALI{
m^a(\bar{\xi}) - \bar{\xi}_\ell \nb^a m^\ell(\bar{\xi}) = 2m^a(\bar{\xi}) - \nb^a(\bar{\xi}_\ell m^\ell(\bar{\xi})) = 0, \qquad \mbox{ for all } \bar{\xi} \in S^1.
}
This equality implies that $m$ is the gradient of $F(\xi ) = \fr{1}{2} \xi \cdot m(\xi)$.  In particular, $m = \nb F$ implies $F(\xi) = \fr{1}{2} \xi \cdot \nb F$, which tells us that $F$ must be homogeneous of degree two by Euler's theorem on homogeneous functions, and hence $m = \nb F$ must be homogeneous of degree one.

At this point we are able to conclude that the special estimate \eqref{eq:specialBound} implies that the velocity field $T^\ell$ must be a scalar multiple of the mSQG Biot-Savart law unless the multiplier is degree $1$ homogeneous, the velocity field is compressible, and the even part of the multiplier is the gradient of an (odd) homogeneous function of degree $2$.  We do not know if the latter class contains further examples of nonlinearities satisfying \eqref{eq:specialBound}.  However, our methods can be extended to show that every such multiplier is bounded on $\dot{W}^{2,\infty} \times \dot{H}^{-1+\de}\times \dot{H}^{-1+\de}$ by deriving a double-divergence form for the nonlinearity using a second order Taylor expansion in frequency.





\appendix

\section{Homogeneous function spaces}\label{sec:functionSpaces}
In this section we gather relevant facts about the homogeneous function spaces employed in this work.  Since multiple definitions of homogeneous Sobolev and Besov spaces exist in the literature, the precise definitions are important.  We focus more on the case of $\R^2$ since the case of $\T^2$ is simpler but similar.

\subsection{The space \texorpdfstring{$\dot{W}^{2,\infty}$}{for the test function}}  
We define the spaces
\ali{
\dot{C}_0^2 &= \{ \psi \in C^2(\R^d) ~:~ \| \nb^2 \psi \|_{L^\infty} < \infty, |\nb^2 \psi(x)| = o(1) \mbox{ as } |x| \to \infty \} \\
\dot{W}_0^{2,\infty} &= \{ \Psi \in \DD'(\R^d) ~:~ \| \nb^2 \Psi \|_{L^\infty} < \infty, |\nb^2 \psi(x)| = o(1) \mbox{ as } |x| \to \infty \} \\
\dot{W}^{2,\infty} &= \{ \Psi \in \DD'(\R^d) ~:~ \| \nb^2 \Psi \|_{L^\infty} < \infty \} 
}
We equip them with the semi-norm $[ \Psi ]_{\dot{W}^{2,\infty}} = \| \nb^2 \Psi \|_{L^\infty}$.  Each space fails to be Hausdorff (unless one considers the Hausdorff quotient), since linear polynomials have semi-norm equal to $0$.  

{\bf The space $\dot{C}_0^2$ is the closure of $C_c^\infty$ in the semi-norm $\| \nb^2 \Psi \|_{L^\infty}$.}  That is, if $\nb^2 \psi_n \to \nb^2 \Psi$ uniformly, then $\nb^2 \Psi$ is continuous, bounded, and satisfies $|\nb^2 \Psi(x)| = o(1)$ as $|x| \to \infty$.  Conversely, let $\Psi \in \dot{C}_0^2$ and define $\psi_n(x) = \phi_n(x)  \eta_{1/n} \ast \Psi$, where $\phi_n(x) = \phi(x/n)$ is a bump function equal to $1$ on $|x| \leq n$, and $\eta_\ep(x) = \ep^{-2} \eta(x/\ep)$ is a standard mollifying kernel that is even, non-negative, and supported in $|x| < \ep$.  Since $|\nb^k \phi_n(x)| \lesssim_k |x|^{-k}$, to check $\| \nb^2 (\Psi - \psi_n)\|_{L^\infty} \to 0$ as $n \to \infty$ it suffices to show
\ali{
|\nb^2 \Psi(x)| = o(1), \qquad |\nb \Psi(x)| = o(|x|), \qquad |\Psi(x)| = o(|x|^2), \qquad \mbox{as } |x| \to \infty.  \label{eq:littleOh2}
}
The first of these is our assumption $\Psi \in \dot{C}^2_0$, and it implies the latter two.  For example, to prove the last one, use Taylor's theorem with remainder to write
\ALI{
\Psi(x) = \Psi(0) + x^a \nb_a \Psi(0) + x^a x^b \int_0^1 \nb_a \nb_b \Psi( \si x ) (1-\si) d\si 
}
The first two terms are clearly $o(|x|^2)$.  To check that the last one is, let $\ep > 0$, $0 < \de < 1$ be given, and choose $R_0 > 0$ so that $|\nb^2 \Psi(y)| < \ep$ for all $|y| > \de R_0$.  Then for all $|x| \geq R_0$, the integrand is bounded by $\max \{ \| \nb^2 \Psi \|_{L^\infty}, \ep \cdot 1_{\si \geq \de} \}$, which means that the integral itself is $o(1)$ as $|x| \to \infty$ and $\Psi$ itself is $o(|x|^2)$ as desired.  The proof that $|\nb \Psi(x)| = o(|x|)$ is similar.

Now suppose $\Psi \in \dot{W}^{2,\infty}$.  Some basic properties of $\Psi$ are that {\bf $\Psi(x)$ and $\nb \Psi(x)$ are Lipschitz continuous and satisfy the bounds $|\Psi(x)| = O(|x|^2), |\nb \Psi(x)| = O(|x|)$ as $|x| \to \infty$.}  Indeed, let $\Psi_\ep = \eta_\ep \ast \Psi$ be a regularization of $\Psi$ with a non-negative, even mollifying kernel $\eta_\ep$.  Non-negativity of the mollifying kernel implies that $|\nb^2 \Psi_\ep(x)| \leq |\nb^2 \Psi (x)|$ pointwise.  Then the Taylor expansion formula $\nb \Psi_\ep(y) - \nb \Psi_\ep(x) = (y^a - x^a) \int_0^1 \nb_a\nb \Psi_\ep((1-\si) x + \si y) d\si$ shows that $\nb \Psi_\ep$ are uniformly Lipschitz.  
For any particular $\ep > 0$ we can subtract a linear polynomial $L_\ep(x)$ from $\Psi_\ep(x)$ so that $\widetilde{\Psi}_\ep(x) \coloneqq \Psi_\ep(x) - L_\ep(x)$ satisfies $\widetilde{\Psi}_\ep(0) = 0$, $\nb \widetilde{\Psi}_\ep(0) = 0$.  Then the sequence $\nb \widetilde{\Psi}_\ep$ is locally uniformly bounded and uniformly Lipschitz.  
The integral formulas $\widetilde{\Psi}_\ep(x) = x^a x^b \int_0^1 \nb_a \nb_b \Psi_\ep(\si x) (1-\si) d\si$ and $\nb \widetilde{\Psi}_\ep(x) = x^a \int_0^1 \nb_a \nb \Psi_\ep(\si x) d\si$ then guarantee that 
\ali{
|\widetilde{\Psi}_\ep(x)| \leq \| \nb^2 \Psi \|_{L^\infty} |x|^2, \qquad |\nb \widetilde{\Psi}_\ep(x)| \leq \| \nb^2 \Psi \|_{L^\infty} |x| \label{eq:growthBoundsPsiep}
}
hold uniformly in $\ep$.  Since $\widetilde{\Psi}_\ep$ and $\nb \widetilde{\Psi}_\ep$ are equi-bounded and Lipschitz locally uniformly in $\ep$ and they converge  uniformly on compact sets by Arzel\`{a}-Ascoli along a subsequence to some $\widetilde{\Psi}$ that satisfies the growth bounds in \eqref{eq:growthBoundsPsiep} and $\nb^2 \widetilde{\Psi} = \nb^2 \Psi$ as a distribution.  Since $\Psi$ and $\widetilde{\Psi}$ differ by a linear polynomial, we have that $|\Psi(x)| = O(|x|^2)$ and $|\nb \Psi(x)| = O(|x|)$, and that $\nb \Psi = \nb \widetilde{\Psi} + C$ is Lipschitz continuous as desired.

{\bf There exists a sequence $\psi_n$ in $C_c^\infty(\R^2)$ such that $\nb^2 \psi_n \rightharpoonup \nb^2 \Psi$ in $L^\infty$ weak-*.}  It suffices to take $\psi_n(x) = \phi(x/n) \widetilde{\Psi}_{1/n}(x)$.  Then for any test function $\phi^{j\ell}(x)$ supported in $|x| < n$, we have that
\ALI{ 
\int \phi^{j\ell}(x) \nb_j \nb_\ell \psi_n(x) dx &= \int \phi^{j\ell}(x) \nb_j \nb_\ell \widetilde{\Psi}_{1/n}(x) dx = \int \eta_{1/n} \ast \phi^{j\ell}(x) \nb_j \nb_\ell \Psi(x) dx 
}
converges to $\int \phi^{j\ell}(x) \nb_j \nb_\ell \Psi(x) dx$ as $n \to \infty$.  The same convergence holds for any $\phi^{j\ell} \in L^1$ by approximation using the bound $\sup_n \| \nb^2 \psi_n \|_{L^\infty} < \infty$, which is a consequence of \eqref{eq:growthBoundsPsiep}.

{\bf When $\Psi \in \dot{W}_0^{2,\infty}$, this sequence has the additional property that $\| \nb^2 \psi_n \|_{L^\infty} \to \| \nb^2 \Psi \|_{L^\infty}$ as $n \to \infty$.}  To see this fact, note that $|\nb^2 \widetilde{\Psi}_\ep(x)| \leq \| \nb^2 \Psi \|_{L^\infty(B_\ep(x))}$ by the support and non-negativity of $\eta_\ep$.  We then obtain $\sup_{\ep \leq 1} |\nb \widetilde{\Psi}_\ep(x)| = o(|x|)$ and $\sup_{\ep \leq 1} |\widetilde{\Psi}_\ep(x)| = o(|x|^2) $ by the same integral formulas used to prove \eqref{eq:littleOh2}.  With these properties, $|\nb^k \phi_n(x)| \lesssim |x|^{-k}$ and $\supp \nb^k \phi_n \subseteq \{ |x| \geq n \}$ for $k = 1, 2$, we can see that $\| \nb^2 (\phi_n(x) \widetilde{\Psi}_{1/n}) \|_{L^\infty}$ converges to $\| \nb^2 \Psi \|_{L^\infty}$ as $n \to \infty$ as desired.

Finally, {\bf the space $\dot{W}^{2,\infty}$ is complete in that every Cauchy sequence converges,} and as a consequence so are $\dot{C}_0^2$ and $\dot{W}_0^{2,\infty}$ since they are closed subspaces.  Namely suppose $\Phi_n$ is Cauchy in $\dot{W}^{2,\infty}$, and let $A_{j\ell} \in L^\infty$ be the strong limit of $\nb_j\nb_\ell \Phi_n$.   The functions $\widetilde{\Phi}_n$ normalized so that $\widetilde{\Phi}_n(0) = 0, \nb \widetilde{\Phi}_n(0) = 0$, and $\nb^2 \widetilde{\Phi}_n = \nb^2 \Phi_n$ are also Cauchy.  Then by the growth bounds similar to \eqref{eq:growthBoundsPsiep}, $\widetilde{\Phi}_n$ and its first derivatives are continuous, uniformly bounded on compact sets, and uniformly Lipschitz, implying the existence of a subsequence $n_k$ such that $\widetilde{\Phi}_{n_k}$ converges in $C^1$ on compact sets to a some $\widetilde{\Phi}_\infty$.  This $\widetilde{\Phi}_\infty$ does not depend on the subsequence; in fact, its weak second derivative is given by $\nb_j\nb_\ell \widetilde{\Phi}_\infty = \lim_{k \to \infty} \nb_j \nb_\ell \widetilde{\Phi}_{n_k} = A_{j\ell}$.  
Now the full sequence must converge to $\widetilde{\Phi}_{\infty}$, for otherwise we would have a $\de > 0$ such that $\| \nb_j\nb_{\ell} \Phi_{n} - A_{j\ell} \|_{L^\infty} = \| \nb_j\nb_{\ell} \Phi_n - \nb_j\nb_{\ell} \widetilde{\Phi}_{\infty} \|_{L^\infty} \geq \de$ for infinitely many $n$, contradicting the definition of $A_{j\ell}$.

\subsection{The space \texorpdfstring{$\dot{H}^{s}$}{for the scalar field}} \label{sec:hdots}

We employ a definition based on Littlewood-Paley theory, since it connects directly to the analysis used here.  For clarity, we focus on Euclidean space and generalize the dimension to $\R^d$.  We will establish among other things density results for $\dot{H}^{-1+\de}(\R^2)$, $\de \geq 0$.

For $s \in \R$ and $\th$ a tempered distribution, define
\ALI{
\| \th \|_{\dot{H}^{s}}^2 \coloneqq \sum_k \left(2^{sk} \| P_k \th \|_{L^2}\right)^2 = \| 2^{sk} \| P_k \th \|_{L^2} \|_{\ell_k^2}^2
}
This norm is equivalent to the norm $\| \| \La^{sk} P_k \th \|_{L^2} \|_{\ell_k^2}$ with $\La = \sqrt{-\De}$.

We define the space $\dot{H}^{s}$ to be the space of tempered distributions with $\| \th \|_{\dot{H}^{s}}^2$ finite, and with the vanishing zero mode property that $\lim_{m\to \infty} P_{\leq -m} \th = 0$ weakly in $\SS'(\R^d)$.  (For periodic $\th$, $\lim_{m\to \infty} P_{\leq -m} \th = \dashint_{\T^d} \th(x) dx$ is the mean value of $\th$, so we are assuming $\th$ has mean zero in that case.)  This additional condition makes $\| \cdot \|_{\dot{H}^{s}}$ a norm rather than a semi-norm, since the series expansion $\th = P_{\leq -m} \th + \sum_{k \geq -m} P_k \th$ always converges weakly in $\SS'$ for $m$ fixed.  It also guarantees that $\SS_0$, the space of Schwartz functions with compact Fourier support away from $0$, is dense in $\dot{H}^{s}$.  

In particular, in contrast to other definitions (e.g. \cite{grafakos2009modern}) we exclude constants and polynomials, whose Fourier transforms are linear cominations of the delta function at the origin and its derivatives, as they would have zero norm.  

A common alternative definition (e.g. \cite{bahouri2011fourier}) is to require that $\widehat{\th} \in L_{loc}^1(\widehat{\R^d})$ and $\int |\xi|^{2s} |\widehat{\th}(\xi)|^2 d\xi < \infty$, which always implies membership in the space $\dot{H}^s$ defined above.  Conversely, when $s < d/2$ and $\th \in \dot{H}^s$, it is true that $\widehat{\th} \in L_{loc}^1(\widehat{\R^d})$, and the two definitions of $\dot{H}^s$ are equivalent.  Namely, local integrability of $\widehat{\th}$ away from $0$ is true for any $s$ by Plancharel, while for $s < d/2$ the sequence $\widehat{P_{\leq 0} \th} - \widehat{P_{\leq -m} \th}$ is Cauchy in $L^1(\widehat{\R^d})$ as $m \to \infty$, and converges weakly in $\SS'(\widehat{\R^d})$ to $\widehat{P_{\leq 0} \th}$, implying local integrability of $\hat{\th}$ near $0$.  In fact, for $s < d/2$, the $\dot{H}^s$ defined by either condition is complete (as explained below).

We will use often the following inequality, valid for Schwartz $\psi$ on $\R^d$
\ali{
\| P_{\approx k} \psi \|_{L^2(\R^n)} &\lesssim_L \min \{ 2^{(d/2)k} \| \psi \|_{L^1(\R^n)}, 2^{-L k} \| \nb^L P_{\approx k} \psi \|_{L^2(\R^n)} \}. \label{eq:keyTestIneq}
}
The $L^1$ control is useful at low frequencies, while the control of $\nb^L \psi$ in $L^2$ will be used at high frequency.  (On the torus, we use instead that $P_k \psi = 0$ for sufficiently large, negative $k$ and the analysis is simpler, though the properties of $\dot{H}^s$ are slightly different.)

\paragraph{The regime $s > -d/2$}  We now consider $s > -d/2$ and consider the special case of $\dot{H}^{-1}(\R^2)$ at the end.  The following properties hold for $s = -d/2 + \de$ where $\de > 0$.  Writing $\de = z + \bar{\de}$ with $z \geq 0$ an integer and $0 < \bar{\de} \leq 1$, we have the interpolation estimate
\ali{
\| \th \|_{\dot{H}^{-d/2 + z + \bar{\de}}} &\lesssim \| \nb^{z+1} \th \|_{L^2}^\a \| \nb^{z} \th \|_{L^1}^{1-\a}, \qquad \a = 2\bar{\de}/(2+d), \qquad \th \in \SS(\R^d). \label{eq:interpBound}
}
A standard approach to proving bounds such as \eqref{eq:interpBound} is to bound
\ali{
\| \th \|_{\dot{H}^{-d/2 + z + \bar{\de}}}^2 \lesssim \sum_{k \leq \bar{k}} \left(2^{(-d/2 + \bar{\de})k} \| \nb^z P_k \th \|_{L^2} \right)^2 +\sum_{k \geq \bar{k}} \left(2^{(- d/2 + \bar{\de} - 1) \bar{k}} \| \nb^{z+1} P_k \th \|_{L^2} \right)^2. \label{eq:lowAndHighBounds}
}
Applying \eqref{eq:keyTestIneq} for $\psi = \nb^z \th$ at low frequency, using the almost orthogonality of Littlewood-Paley projections for the high frequency sum, and optimizing over $\bar{k}$ yields the result.  

Every Schwartz function satisfies the vanishing condition $P_{\leq -m} \psi \rightharpoonup 0$ in $\SS'$ and belongs to $\dot{H}^{-d/2 + \de}$ by \eqref{eq:interpBound} or \eqref{eq:lowAndHighBounds}.  From this fact and \eqref{eq:interpBound} or \eqref{eq:lowAndHighBounds} we obtain the density of $C_c^\infty$ in $\dot{H}^{-d/2+\de}$.


\paragraph{The regime $s < d/2$}  Now assuming $s < d/2$, two important properties of $\dot{H}^s$ are that it is complete, and that $\dot{H}^s$ convergence implies weak convergence in $\SS'$ (for our specific definition of $\dot{H}^s$).  To see the latter claim, let $\psi$ be Schwartz and $\th \in \dot{H}^s$.  We use the vanishing zero mode condition to obtain 
\ALI{
\langle \th, \psi\rangle = \lim_{m \to \infty} \left[\langle \th, P_{\leq - m} \psi \rangle + \sum_{k = -m}^\infty \langle P_k \th, \psi \rangle \right] = \lim_{m \to \infty} \sum_{k = -m}^\infty \langle P_k \th, P_{\approx k} \psi \rangle_{L^2}
}
Using Cauchy-Schwartz, we bound this sum by $\sum_k (2^{sk} \| P_k \th \|_{L^2}) (2^{-sk} \|P_{\approx k} \psi \|_{L^2}) \lesssim \| \th \|_{\dot{H}^s} \| \psi \|_{\dot{H}^{-s}} $.  But $-s > -d/2$, so \eqref{eq:interpBound} with $\th$ replaced by $\psi$ implies $\| \psi \|_{\dot{H}^{-s}} < \infty$.  Thus the series converges absolutely and the functional $\th \mapsto \langle \th, \psi \rangle$ is bounded on $\dot{H}^s$, so that $\dot{H}^s$ convergence implies ${\cal S}'$ convergence.

We now sketch a proof of completeness.  If $\th_n$ is Cauchy in $\dot{H}^{s}$, then $H \equiv \sup_n \| \th_n \|_{\dot{H}^s}$ is bounded.  Also, the sequences $P_k \th_n$ are Cauchy in $L^2$ for every fixed $k$ and therefore converge strongly in $L^2$ to functions $\Th_k$ that satisfy $P_{\approx k} \Th_k = \Th_k$ and $\| 2^{sk} \|\Th_k \|_{L^2} \|_{\ell_k^2} \leq H < \infty$ by Fatou's Lemma.   

We can define $\langle \Th, \psi \rangle \coloneqq \lim_{M \to \infty} \int \th^{(M)}(x) \psi(x) dx$ as a linear function on the Schwartz space by checking that the latter sequence is Cauchy in $M$.  Apply Cauchy-Schwartz to bound 
\ALI{
\Big|\int (\th^{(M+\ell)} - \th^{(M)}) \psi(x) dx\Big| &\leq \sum_{M < |k| \leq M + \ell} \left|\int \Th_k P_{\approx k} \psi dx \right| \\
\leq \sum_{M < |k| \leq M + \ell}\left(2^{-sk} \| P_{\approx k} \psi \|_{L^2} \right) \left( 2^{sk} \| \Th_k \|_{L^2} \right) &\lesssim \Big( \sum_{M-1 < |k|} (2^{-sk} \| P_k \psi \|_{L^2})^2 \Big)^{1/2} \| 2^{sk} \| \Th_k \|_{L^2}\|_{\ell_k^2}.
}
We have already noted that $\| 2^{sk} \| \Th_k \|_{L^2}\|_{\ell_k^2} \leq H < \infty$.  Since $-s > -d/2$, we have that $\| \psi \|_{\dot{H}^{-s}} < \infty$.  By dominated convergence, the final bound goes to $0$ as $M$ gets large, and hence the sequence in $M$ defining $\langle \Th, \psi \rangle$ is Cauchy.  Similarly one obtains $|\langle \Th, \psi \rangle| \lesssim \| \psi \|_{\dot{H}^{-s}} H$, $-s > -d/2$, making $\Th$ a tempered distribution that is the weak limit of the $\th^{(M)}$ by \eqref{eq:interpBound}.  

We likewise obtain using \eqref{eq:keyTestIneq} again that
\ALI{
|\langle P_{\leq -m} \Th, \psi \rangle| = \lim_M |\langle \th^{(M)}, P_{\leq -m} \psi \rangle | \lesssim 2^{-(d/2-s)m} \| \psi \|_{L^1} H,
}
which implies the vanishing zero mode condition for $s < d/2$.  Furthermore, $P_k \Th = \Th_k = \lim_{n \to \infty} P_k \th_n$ (strong $L^2$ limit), since $P_k \th^{(M)}$ is independent of $M$ for $M > |k|$.  In particular, $\Th \in \dot{H}^s$ with $\| \Th \|_{\dot{H}^s} \leq H$.  Finally, the inequality
\ALI{
\left| 2^{sk} \| P_k \th_n - P_k \Th \|_{L^2} - 2^{sk} \| P_k \th_m - P_k \Th \|_{L^2}\right| \leq 2^{sk} \| P_k (\th_n - \th_m) \|_{L^2}
} 
implies that $2^{sk} \| P_k \th_n - P_k \Th \|_{L^2}$, regarded as a sequence with values in $\ell^2_k$, is Cauchy in $n$, so it converges as $n \to \infty$ in $\ell_k^2$ to its pointwise limit, which is the $0$ element of $\ell_k^2$.  That is, $\| 2^{sk} \| P_k \th_n - P_k \Th \|_{L^2} \|_{\ell_k^2} \to 0$ as $n \to \infty$, and hence $\th_n \to \Th$ in $\dot{H}^s$.

\paragraph{The space $\dot{H}^{-1}(\R^2)$}

Finally, we address the case of $\dot{H}^{-1}(\R^2)$ that is of relevance to 2D Euler.  In general, a distribution $\th \in \DD'(\R^d)$ belongs to $\dot{H}^{-1}(\R^d)$ if and only if $\th = \nb_i u^i$ is the divergence (in the weak sense) of a vector field $u \in L^2$.  That is, if $\th = \nb_i u^i$, the vanishing zero mode condition follows from $\| P_{\leq -m} \th \|_{L^\infty} \leq \| \nb \eta_{\leq -m} \|_{L^2} \| u\|_{L^2} \lesssim 2^{-(d/2 + 1) m} \| u \|_{L^2}$, while the bound
\ali{
\| \th \|_{\dot{H}^{-1}}^2 &\lesssim \sum_k \left(2^{-k} \| \nb_i P_k u^i \|_{L^2}\right)^2  \lesssim \| u \|^2_{L^2}. \label{eq:H-1adivbd}
}
implies $\| \th \|_{\dot{H}^{-1}} < \infty$.  Conversely, for $\th \in \dot{H}^{-1}$, the sequence of vector fields $u_{(m)}^i = \De^{-1} \nb^i(1 - P_{\leq -m})\th$ is Cauchy in $L^2$ as $m \to \infty$, and its limit $u$ satisfies $\nb_i u^i = \th$ by the vanishing $0$ mode condition.

In dimension $2$, it is no longer true that every $\th \in C_c^\infty(\R^2)$ belongs to $\dot{H}^{-1}(\R^2)$.  Rather, if $\th \in L^1(\R^2)$ has a nonzero integral, then $\th$ cannot belong to $\dot{H}^{-1}(\R^2)$ since then $\int \th(x) dx = \widehat{\th}(0) \neq 0$ implies that $\| \widehat{P_k \th} \|_{L^2(\widehat{\R^2})}^2 \sim 2^{2k}$ for large $k < 0$, and $\| \th \|_{\dot{H}^{-1}}$ is infinite.  

For Schwartz functions $\th$, the vanishing of $\int \th(x) dx = 0$ then provides a necessary and sufficient condition to belong to $\dot{H}^{-1}$, for if this condition holds we can represent $\th = \nb_a u^a$ as the divergence of a Schwartz vector field $u$.  The existence of a Schwartz anti-divergence is simple when the support of $\widehat{\th}$ is bounded away from the origin, for in that case the Helmholtz solution $u_{\infty}^a = \De^{-1} \nb^a \th$ is Schwartz and provides a suitable anti-divergence.  For more general Schwartz functions $\th$ with $\int \th(x) dx = \widehat{\th}(0) = 0$, we can Taylor expand at the origin to obtain $\widehat{\th}(\xi) = i \xi_a \hat{u}_0^a$ for a smooth $\hat{u}_0$ that may have poor decay at high frequency.  Gluing $\hat{u}_0$ and $\hat{u}_{\infty}$ with a smooth cutoff in frequency yields a Schwartz vector field $u^a$ that solves $\widehat{\th}(\xi) = (i \xi_a) \hat{u}^a = \widehat{\nb_a u^a}(\xi)$ and is thus a Schwartz anti-divergence for $\th$.

As a corollary, we have that the space of compactly supported test functions $\th \in C_c^\infty$ with $\int \th(x) dx = 0$ is a dense subspace of $\dot{H}^{-1}$, for the space of Schwartz functions with frequency support away from the origin is already dense in $\dot{H}^{-1}$, and we can obtain any such $\th = \nb_a u^a$, $u \in L^2$, as the limit of the compactly supported sequence $\th_n(x) = \nb_a( \phi(x/n) \eta_n \ast  u^a )$, where $\phi$ as before is a bump function and $\eta_n$ is a suitable mollifier, which converges in $\dot{H}^{-1}$ by \eqref{eq:H-1adivbd}.

\bibliographystyle{abbrv}
\bibliography{eulerOnRn}

\end{document}